\documentclass{amsart}
\usepackage{geometry}
\usepackage{tikz}
\usetikzlibrary{3d,calc,patterns, decorations.pathreplacing}
\usepackage{verbatim}
\usepackage{amssymb}
\usepackage{relsize}
\usepackage{hyperref}
\hypersetup{colorlinks, citecolor=green!50!black, linkcolor=red!50!black}
\usepackage[capitalise]{cleveref}

\newtheorem{thm}{Theorem}[section]
\Crefname{thm}{Theorem}{Theorems}
\newtheorem{lem}[thm]{Lemma}
\Crefname{lem}{Lemma}{Lemmas}
\newtheorem{cor}[thm]{Corollary}
\Crefname{cor}{Corollary}{Corollaries}
\newtheorem{clm}[thm]{Claim}
\Crefname{clm}{Claim}{Claims}

\theoremstyle{definition}
\newtheorem{defn}[thm]{Definition}
\Crefname{defn}{Definition}{Definitions}
\newtheorem{rem}[thm]{Remark}
\Crefname{rem}{Remark}{Remarks}
\newtheorem{ex}[thm]{Example}
\Crefname{ex}{Example}{Examples}

\newcommand{\myfootnote}[1]{{\hypersetup{linkcolor=red}\footnote{\hypersetup{linkcolor=red!50!black}#1}}}
\newcommand{\sepclos}[1]{\overline{#1}}
\newcommand{\sepclosf}{\sepclos{\mathcal{F}}}
\newcommand{\sepinn}[1]{\mathring{#1}}
\newcommand{\sepinnf}{\sepinn{\mathcal{F}}}

\newcommand{\localconn}{\bigsqcap\!}
\newcommand{\bigsqcap}{\mathchoice%
	{\mathop{\mathlarger{\mathlarger{\sqcap}}}}
	{\mathop{\mathlarger{\mathlarger{\mathlarger{\sqcap}}}}}
	{\mathop{\mathlarger{\sqcap}}}
	{\mathop{\mathlarger{\sqcap}}}}

\newcommand{\restrict}{{\upharpoonright}}
\newcommand{\delete}{\backslash}
\newcommand{\contract}{/}
\newcommand{\contractonto}{.}

\newcommand{\rlbracket}{\mathord{[}}
\newcommand{\lrbracket}{\mathord{]}}

\newcommand{\flowerradius}{2cm}
\newcommand{\basicflower}{\draw [thick] (0,0) ellipse [radius=\flowerradius];}
\newcommand{\newedge}[2][]{
	\draw[#1] (0,0) -- (#2:\flowerradius);}
\newcommand{\labelofset}[4][white]{
	\fill [#1!10] (0,0)--(#2:\flowerradius) arc [start angle=#2,end angle=#3,radius=\flowerradius] -- (0,0);\path (0.5*#2+0.5*#3:0.6*\flowerradius) node {#4};}
\tikzset{separrow/.pic={\draw [->,thin] (-2mm,0) -- (2mm,0);},
	striped/.style={pattern=north east lines}}

\title{Pseudoflowers in infinite connectivity systems}
\author{Ann-Kathrin Elm}
\address{Research conducted at Universit\"at Hamburg, Department of Mathematics, Bundesstraße 55, 20146 Hamburg, Germany}
\curraddr{Universit\"at Heidelberg, Institut f\"ur Informatik, Im Neuenheimer Feld 205, 69120 Heidelberg, Germany}
\email{elm@informatik.uni-heidelberg.de}

\begin{document}
\begin{abstract}
Given a graph or a matroid, a tree of tangles is a tree decomposition that displays the structure of the  connectivity: every edge of the decomposition tree induces a \emph{separation}, that is, a way to divide the graph or matroid into two parts; and for every two highly connected areas (encoded as tangles) that live on different sides of some separation, some separation induced by an edge distinguishes them. Separations induced by a tree of tangles cannot cross. One approach to display even more connectivity structure is to insert even more structure into a tree of tangles, for example the \emph{flowers} that were introduced by Oxley, Semple and Whittle in 2007 for matroids and generalised to finite connectivity systems by Clark and Whittle in 2013. Most of the separations displayed by a flower are crossing. In order to extend this theory to the infinite case, we generalise the notion of flowers to infinite connectivity systems, and show that there are maximal generalised flowers. Also, we show in the special case of infinite matroids that of the two types of flowers (\emph{anemones} and \emph{daisies}) only anemones can be extended to truly infinite objects, and provide for general connectivity systems a characterisation of when infinite daisies exist. Furthermore we describe a more abstract view on the interaction of tangles and separations distinguishing them, which among other things provides additional motivation for why there should be maximal generalised flowers.
\end{abstract}
\maketitle
\section{Introduction}
\subsection{Motivation}

In a matroid on ground set $E$, an (oriented) \emph{separation} is an (oriented) bipartition of the ground set.
(More on graph separations and more general separations can be found in \cite{AbstractSepSys}.)
The \emph{order} of a separation $\{E_1,E_2\}$ is a measure of how strongly the two sides are connected, and in the case of a matroid this is the connectivity of $E_1$.
Roughly speaking, for a positive integer $k$, a $k$-tangle is a choice of one side of every separation of order less than $k$ that satisfies a consistency requirement.
A $k$-tangle encodes a substructure in a matroid or graph that is so highly connected that a separation of order less than $k$ cannot cut right through it, by choosing for every such separation the side on which (most of) the highly connected substructure lies.

Tree decompositions are an important tool in graph theory, and thus matroid theory, to represent the structure of the connectivity.
A \emph{tree decomposition} of a matroid on edge set $E$ is a decomposition tree $T$ together with a map
$E\rightarrow V(T)$.
Deleting an edge $xy$ of the decomposition tree yields two subtrees $T_x$ and $T_y$ containing $x$ and $y$ respectively.
If $E_x$ is the set of elements of $E$ that are mapped to a vertex of $T_x$ and $E_y$ is the set of elements of $E$ mapped to a vertex of $T_y$, then $\{E_x,E_y\}$ is the separation \emph{displayed} by $e$.
Furthermore, $(E_x,E_y)$ is the oriented separation displayed by the orientation of $e$ that is oriented towards $y$.
Similarly, for a vertex $v$ of $T$ that is not the image of any element of $E$, deleting $v$ yields components $T_x$ that each correspond to a subset $E_x$ of $E$, and the set of all $E_x$ forms a partition of $E$.

One important application of tree decompositions is as underlying tree structure of trees of tangles.
In the special case of a set $\mathcal{T}$ of $k$-tangles that contain a common $(k-1)$-tangle, a tree of tangles is a tree decomposition where every two tangles in $\mathcal{T}$ are \emph{distinguished} by a separation displayed by the tree decomposition in the sense that the tangles choose distinct sides of that separation, and where every separation displayed by the tree distinguishes two tangles in $\mathcal{T}$ and thus has connectivity $k$.
In such a tree of tangles, the tangles in $\mathcal{T}$ are naturally mapped injectively to the vertices of the decomposition tree and are in this sense arranged in a tree-like fashion by the tree of tangles.
The tree of tangles, its set of displayed separations and the resulting tree-like structure of $\mathcal{T}$ are by no means unique.
Even if two trees of tangles result in the same tree-like structure of $\mathcal{T}$, the set of displayed separations can vary a lot; but the set of displayed equivalence classes (where two separations distinguishing elements of $\mathcal{T}$ are equivalent if they are contained in the same elements of $\mathcal{T}$) is then fixed.
In general there is no tree decomposition that displays all equivalence classes at once.
The main problem is that no two separations displayed by the same tree decomposition can cross, but there can be equivalence classes such that every separation of the first crosses every separation of the second.

This problem can be solved by introducing the notion of flowers.
Flowers are partitions of $E$ with a cyclic order of the partition classes with certain connectivity constraints.
If a union of partition classes has connectivity $k$, then this set together with its complement forms a separation of order $k$; the separations of this type are the separations \emph{displayed} by the flower.
Flowers display many separations of the same order at once, and most of them cross.
In \cite{ClarkWhittle13} it was shown that there is a tree decomposition where every vertex of $T$ that is not the image of an element of $E$ or of an element of $\mathcal{T}$ induces a partition that is a flower, and such that every equivalence class is displayed either by the tree decomposition itself or by one of the induced flowers\footnote{Actually, this was shown for a finer equivalence relation, which then made an additional technical assumption necessary in order for the stronger theorem to hold.}.
Tree decompositions where vertices of the decomposition tree that do not have tangles instead induce flowers can also be used to make the tree structure of $\mathcal{T}$ unique.

\subsection{Results}

For the definition of separations, tangles, tree decompositions and flowers, one does not really need the structure of a matroid but only its connectivity function.
And most of the results, both of this paper and those in \cite{ClarkWhittle13}, hold in the more general setting of connectivity functions.

The goal of this paper is to extend as much of the theory of \cite{ClarkWhittle13} to the infinite setting as possible.
In this paper that is done for infinite connectivity systems, which are closely related to universes of bipartitions.
For similar results in the more general setting of universes of vertex separations (which also generalise graph separations) see \cite{EH:vertexflowers}.
An important part of \cite{ClarkWhittle13} is finding maximal flowers.
For reasons explained in \cref{sec:defsbips}, in order to find these in the infinite setting we need to work with a weaker notion of flower, which we call pseudo\-flower.
As was done in \cite{AikinOxley08} for flowers, we analyse in \cref{sec:ordersinkpa} what the connectivity of different unions of partition classes of a pseudo\-flower can be.

In the infinite setting, some $k$-tangles are troublesome in the sense that they can contain a chain of separations but not the supremum of the chain.
This could easily cause a problem for our techniques for finding maximal pseudo\-flowers.
However, we show that such $k$-tangles are redeemed by always showing another very useful property: they uniquely extend to $l$-tangles for all $l\geq k$ (\cref{notlimitclosedinducesaleph0}).
With these tools we then show that pseudo\-flowers extend to maximal pseudo\-flowers, in some cases uniquely (\cref{existenceleqmaxkpf,finestrefinement,maxpreccurlyeqkpf,exmaxpreccurlyeqA}).

We then analyse the abstract structure of the equivalence classes of separations, which becomes a unique tree-like structure by adding in flowers and gives rise to a unique tree-like structure of $\mathcal{T}$\footnote{Another way to obtain a unique tree-like structure is by finding a way to canonically construct trees of tangles as is done for example in \cite{ProfilesNew} or \cite{Entanglements}. But then the tree-like structure again depends on the canonical construction chosen. More details can be found in \cref{sec:abstractinfinite}.}.
In the finite case the tree-structure is easily obtained from a result by Cunningham and Edmonds~\cite{CunninghamEdmonds80}.
In the infinite case, again several generalisations have to be made.

There are two types of flowers, anemones and daisies.
We show that in an infinite matroid, there are no infinite daisies and nearly every finite daisy can be extended to a maximal finite daisy (\cref{noinfinitedaisies}).
Lastly, we consider the problem whether an infinite chain of separations of order~$k$ that are pairwise not displayed by a common $k$-anemone can have a supremum whose order is also~$k$.
We show for general infinite connectivity systems that this happens if and only if there are infinite daisies.

\section{Tools and terminology}

For sets $X$ and $e$, $X+e$ and $X-e$ are used as shorthand for $X\cup \{e\}$ and $X\setminus \{e\}$ respectively.
For a function $f$ defined on the power set $2^X$ of $X$ and $x\in X$, $f(x)$ is shorthand for $f(\{x\})$.
For two partitions $\mathcal{P}$ and $\mathcal{Q}$ of a ground set $X$, the \emph{common refinement} of $\mathcal{P}$ and $\mathcal{Q}$ is the coarsest partition of $X$ that is a refinement of both $\mathcal{P}$ and $\mathcal{Q}$.

\subsection{Connectivity systems}
Similarly to \cite{ClarkWhittle13}, the setting of this paper is that of a \emph{connectivity system}: A set $E$ with a map $\lambda$ from the set of subsets of $E$ to $\mathbb{N}\cup \{\infty\}$ that satisfies for all subsets $A$ and $B$ of $E$ that $\lambda(A)=\lambda(E\setminus A)$ (symmetry) and $\lambda(A)+\lambda(B)\geq \lambda(A\cup B)+\lambda(A\cap B)$ (submodularity).
As we work in the infinite setting we additionally ask that $\lambda$ be \emph{limit-closed} in the sense that if $k\in \mathbb{N}$ and $(A_i)_{i\in I}$ is a chain of subsets of $E$ of connectivity at most $k$, then $\bigcup_{i\in I}A_i$ also has connectivity at most $k$.
Note that the ground set of a (possibly infinite) matroid together with its connectivity function forms a connectivity system and that the connectivity function is limit-closed \cite{InfMatFinConn}.
Because of limit-closedness, the connectivity of an infinite $X\subseteq E$ can be related to the connectivity of its finite subsets:
\begin{lem}\label{fintoinfconn}
	Let $k\in \mathbb{N}$ and let $X\subseteq E$ be a set such that all finite subsets have connectivity at most $k$.
	Then also $\lambda(X)\leq k$.
\end{lem}
\begin{proof}
	Among all finite subsets of $X$ let $Y$ be one with maximal connectivity.
	By Zorn's Lemma there is a maximal set $Z$ with $Y\subseteq Z\subseteq X$ and $\lambda(Z)\leq \lambda(Y)$.
	If there is $e\in X\setminus Z$, then $\lambda(Z+e)\leq \lambda(Z)+\lambda(Y+e)-\lambda(Y)\leq \lambda(Z)\leq \lambda(Y)$ by submodularity, a contradiction to the choice of $Z$.
	So $X=Z$ and hence $\lambda(X)\leq \lambda(Y)\leq k$.
\end{proof}

Note that the previous lemma can be rephrased as follows: for every $k \in \mathbb{N}$, every set $X \subseteq E$ with $\lambda(X) \geq k$ has a finite subset $Y$ with $\lambda(Y) \geq k$.

Basic facts about connectivity functions for finite sets can also be found in, for example, \cite{GroheSchweitzer}.
In particular the observation that $\lambda(\emptyset) \leq \lambda(X)$ for all $X \subseteq E$ still holds for infinite connectivity systems.
Furthermore, the following variation of \cite[Lemma 2.13]{GroheSchweitzer} can be proven with the same proof idea.

\begin{lem}[Similar statement and proof as {\cite[Lemma 2.13]{GroheSchweitzer}}]\label{minimalsetsmall}
	Let $k\in \mathbb{N}$ and $X \subseteq E$.
	If $\lambda(X) \geq k$, then there is $Y \subseteq X$ with $\lambda(Y) \geq k$ and $|Y| \leq k$.
\end{lem}
\begin{proof}
	Assume $\lambda(X) \geq k$.
	By \cref{fintoinfconn} there is a finite $Y' \subseteq X$ with $\lambda(Y')\geq k$.
	Let $Y \subseteq Y'$ be minimal with $\lambda(Y) \geq k$.
	
	Assume for a contradiction that $Y$ has more than $k$ elements, and count these elements as $\{y_1, \ldots, y_n\}$ and let $Y_i := \{y_1, \ldots, y_i\}$.
	As every $\lambda(Y_i)$ is a non-negative integer and $\lambda(Y_{n-1}) \leq k-1$, there is $1 \leq i \leq n-1$ such that $\lambda(Y_{i-1}) \geq \lambda(Y_i)$.
	Then
	\begin{displaymath}
		\lambda(Y) \leq \lambda(Y - y_i) + \lambda(Y_i) - \lambda(Y_{i-1}) \leq \lambda(Y - y_i),
	\end{displaymath}
	contradicting the fact that $Y$ is minimal with respect to $\lambda(Y) \geq k$.
\end{proof}

This lemma has the following useful corollary.

\begin{lem}
	Let $X \subseteq E$ and $k \in \mathbb{N}$.
	If all subsets of $X$ of size at most $k+1$ have connectivity at most $k$, then also $X$ has connectivity at most $k$.
\end{lem}
\begin{proof}
	Assume for a contradiction that $\lambda(X) > k$, that is, $\lambda(X) \geq k+1$.
	Then $X$ has by \cref{fintoinfconn} a finite subset $Y$ that also has connectivity at least $k+1$, and $Y$ can be chosen minimal with respect to this property.
	Then by \cref{minimalsetsmall} $Y$ has at most $k+1$ many elements, a contradiction.
\end{proof}

\subsection{Cyclic orders and cuts}

In this paper, pseudo\-flowers are partitions of $E$ together with a cyclic order, and infinite cyclic orders are more complicated than finite ones.
Basics about cyclic orders and cuts can be found in e.g. \cite{Novak84}.
Open intervals (of linear or cyclic orders) will be denoted as $\lrbracket x,y\rlbracket$, and the half-open intervals accordingly.
An interval is \emph{non-trivial} if it is neither $\emptyset$ nor the whole ground set.
Every non-trivial interval $Y$ naturally has a linear order induced by the cyclic order as follows:
$x<y$ if there is $z\in X \setminus Y$ such that $y\in \lrbracket x,z\rlbracket$.
Given distinct elements $x$ and $y$ of $X$, $y$ is the \emph{successor} of $x$ if there is no $z\in X$ with $z\in \lrbracket x,y\rlbracket$.
Similarly $y$ is the \emph{predecessor} of $x$ if there is no $z\in \lrbracket x,y \rlbracket$.

For cyclically ordered sets $X$ and $Y$, a \emph{monotone} map $f:X\rightarrow Y$ is a map that satisfies $y\in \lrbracket x,z\rlbracket$ for all $x,y,z\in X$ with $f(y)\in \lrbracket f(x), f(z)\rlbracket$.
Note that if the image of $f$ does not have exactly two elements, then for every interval $I$ of $Y$ the set $f^{-1}(I)$ is an interval of $X$.
Every linearly ordered set $X$ induces a cyclic order where $y\in \lrbracket x,z\rlbracket$ if and only if $x<y<z$ or $y<z<x$ or $z<x<y$.
A \emph{cut} of a cyclically ordered set is a linear order on $X$ that induces the cyclic order.
Not only are the cuts of $X$ naturally ordered cyclically themselves (see e.g. \cite{Novak84}), but if $C(X)$ denotes the disjoint union of $X$ and its cuts, then $C(X)$ has a natural cyclic order where, for $x,y\in X$ and a cut $v$, $v\in \lrbracket x,y\rlbracket$ if $y<x$ in the linear order $v$.
Intuitively, the following happens:
Let $X$ be some cyclically ordered set.
If one envisions $X$ as boxes arranged in a circle according to the cyclic order (see for example \cref{fig:cyccomp}), then it is possible to cut up the circle at two places without cutting through boxes, thereby dividing the set of boxes into two intervals.
Cutting in only one place yields a line that depicts a linear order.
If $X$ is finite, then every one of these ``cut points'' is between two adjacent boxes.
So $X$ and the set of possible cut points form together another cyclically ordered set $C(X)$.
If $X$ is infinite, then not every cut point is between two adjacent boxes, but still $C(X)$ is a cyclically ordered set, the \emph{cycle completion} of $X$.
The construction of $C(X)$ from $X$ is closely related to the Dedekind-construction and to the construction of a pseudo-line from a linear order as it is done in \cite{BCC:graphic_matroids}.
It has several properties that are intuitively clear but tedious to prove formally, the most important one being the following:
for every non-trivial interval $I$ of $X$ there are unique $v,w\in C(X)\setminus X$ such that $I=[v,w]\cap X$.
Also, for $v,w\in C(X)\setminus X$, every subset of $[v,w]$ (the interval taken in $C(X)$) has a supremum in the linear order of $[v,w]$.
Furthermore, if $X$ has at least two elements, then every $x\in X$ has a predecessor and a successor in $C(X)$, and those are contained in $C(X)\setminus X$.
More details can be found in \cite{EH:vertexflowers}.

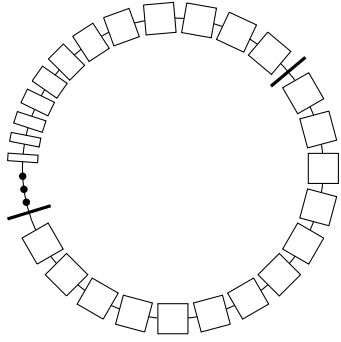
\begin{figure}
	\begin{tikzpicture}
		\draw (0,0) ellipse [radius=2cm];
		\foreach \angle / \width in {
			0/1, 15/1, 30/1,
			50/1, 65/1, 80/1, 95/1,
			110/0.9, 123/0.8, 135/0.7, 145/0.6, 154/0.5, 162/0.4, 169/0.3, 176/0.25,
			345/1, 330/1, 315/1, 300/1, 285/1, 270/1, 255/1, 240/1, 225/1, 210/1
		}
		\path [rotate=\angle,xshift=2cm,yscale=\width, draw=black, fill=white] (0,0) +(0.2,0.2) -- +(0.2,-0.2) -- +(-0.2,-0.2) -- +(-0.2,0.2) -- +(0.2,0.2);
		\foreach \angle in {183, 188, 193}
		\draw (\angle:2cm) node [fill,circle, inner sep=1pt] {};
		\draw [very thick] (40:1.7) -- (40:2.3) (197:1.7) -- (197:2.3);
	\end{tikzpicture}\caption{A set (whose elements are indicated by boxes) which is cyclically ordered (indicated by the arrangement of the boxes on a circle) and two ``cutting points'' dividing the cyclically ordered set into two intervals.}\label{fig:cyccomp}
\end{figure}

\subsection{Separation systems}\label{sec:sepsys}

The definitions for abstract separation systems in this subsection are only needed for \cref{sec:abstractionfromunderlyingsepsys}.
For the rest of the paper (especially \cref{sec:abstract}) it suffices to know the special case of universes of bipartitions which is explained in \cref{sec:prelbips}.
Such a separation system contains subsets of the ground set, and so terminology from separation systems is borrowed for subsets of a ground set (see also \cref{sec:prelbips}).
A detailed introduction to abstract separation systems that also gives an overview of important basic facts can be found in \cite{AbstractSepSys}.

\begin{defn}\cite{AbstractSepSys}
	A \emph{separation system} $(\overrightarrow{S},\leq,\mathord{^*})$ is a set $\overrightarrow{S}$ together with a partial order $\leq$ and an involution $\mathord{^*}$ which is order-reversing, i.e.\ $\overrightarrow{s}\leq \overrightarrow{t}\Leftrightarrow \overrightarrow{s}^*\geq \overrightarrow{t}^*$ for all elements $\overrightarrow{s}$ and $\overrightarrow{t}$ of $\overrightarrow{S}$.
	For an element $\overrightarrow{s}$ of $\overrightarrow{S}$, $\overrightarrow{s}^*$ is also denoted as $\overleftarrow{s}$ and called the \emph{inverse} of $\overrightarrow{s}$.
	The \emph{orientations} of $\overrightarrow{s}$ are $\overrightarrow{s}$ and $\overleftarrow{s}$.
	A \emph{subsystem} of $\overrightarrow{S}$ is a separation system $(\overrightarrow{S}', \leq ', {\mathord{^*}}')$ where $\overrightarrow{S}'$ is a subset of $\overrightarrow{S}$, $\leq'$ is the restriction of $\leq$ to $\overrightarrow{S}'$ and ${\mathord{^*}}'$ is the restriction of $\mathord{^*}$.
	
	A \emph{universe} $(\overrightarrow{S},\leq,\mathord{^*},\vee,\wedge)$ is a separation system in which all elements $\overrightarrow{s}$ and $\overrightarrow{t}$ of $\overrightarrow{S}$ have a join $\overrightarrow{s}\vee \overrightarrow{t}$ and a meet $\overrightarrow{s}\wedge \overrightarrow{t}$.
	(Recall that in partial orders, a join of two elements $\overrightarrow{r}$ and $\overrightarrow{s}$ is the smallest element of $\{\overrightarrow{p} : \overrightarrow{p} \geq \overrightarrow{r} \text{ and } \overrightarrow{p} \geq \overrightarrow{s} \}$, and similarly the meet of two elements is the biggest element of $\{\overrightarrow{p} : \overrightarrow{p} \leq \overrightarrow{r} \text{ and } \overrightarrow{p} \leq \overrightarrow{s} \}$.)
	A universe is \emph{submodular} if it comes with an order function $\left|\mathord{\cdot} \right|$, i.e.\ a symmetric submodular function with values in the non-negative integers together with $\infty$.
	A \emph{subuniverse} of $(\overrightarrow{S},\leq,\mathord{^*},\vee,\wedge)$ is a universe that is a subsystem of the separation system $(\overrightarrow{S},\leq,\mathord{^*})$ such that for all separations $\overrightarrow{r}$ and $\overrightarrow{s}$ of the subsystem, also $\overrightarrow{r} \wedge \overrightarrow{s}$ and $\overrightarrow{r} \vee \overrightarrow{s}$ are contained in the subsystem.
	For any non-negative integer $k$, the set of all separations in $(\overrightarrow{S},\leq,\mathord{^*},\vee,\wedge)$ of order less than $k$ forms, together with the partial order and involution inherited from $(\overrightarrow{S},\leq,\mathord{^*},\vee,\wedge)$, a separation system $S_k$.
\end{defn}

Note that this paper's definition of a submodular universe differs from the definition in \cite{AbstractSepSys} in that the order function is not only allowed to have integer values but additionally can take the value $\infty$.
It is important that arbitrary positive reals are not allowed: this property ensures that if for some $k\in \mathbb{N}$ two separations have order less than $k$ and their meet has order at least $k-1$, then their join has order less than $k$ as well.
The latter fact is used frequently, e.g.\ in \cref{cornersinSexist}.
As opposed to \cite{AbstractSepSys}, this paper' emphasis lies with infinite separation systems, and some come with a natural order function which does take the value $\infty$, for example the connectivity function of an infinite matroid.
Another example is, for some infinite graph, the connectivity function whose ground set is the set $E$ of edges of the graph and that maps a set $X$ of edges to the number (in $\mathbb{N} \cup \{\infty\}$) of vertices incident with both an edge in $X$ and an edge in $E \setminus X$.

\begin{ex}
	Assume that $\mathcal{U}$ is a subuniverse of a universe $\mathcal{U}'$ and that the order function of $\mathcal{U}$ is denoted by $\sigma$.
	Then $\sigma$ can be extended to an order function $\sigma'$ of $\mathcal{U}'$ by letting $\sigma'(\overrightarrow{s})=\sigma(\overrightarrow{s})$ if $\overrightarrow{s}$ is contained in $\mathcal{U}$ and $\sigma'(\overrightarrow{s})=\infty$ otherwise.
	In particular, $\sigma'$ is submodular and symmetric.
	On the other hand, the set of separations in $\mathcal{U}$ which have finite order is closed under joins and meets and thus is the set of separations of a subuniverse $\mathcal{U}''$ of $\mathcal{U}$.
\end{ex}

A separation system can contain elements that behave counter-intuitively:

\begin{defn}\cite{AbstractSepSys}
	Let $\overrightarrow{S}$ be a separation system.
	An element $\overrightarrow{s}$ of $\overrightarrow{S}$ is \emph{degenerate} if $\overrightarrow{s}=\overleftarrow{s}$.
	A separation $\overrightarrow{s}$ is \emph{trivial} if there is a separation $\overrightarrow{t}$ in $\overrightarrow{S}$ such that $\overrightarrow{s}<\overrightarrow{t}$ and $\overrightarrow{s}<\overleftarrow{t}$.
	A separation $\overrightarrow{s}$ is \emph{small} if $\overrightarrow{s}\leq \overleftarrow{s}$.
	The inverse of a small separation is \emph{co-small} and the inverse of a trivial separation is \emph{co-trivial}.
	A separation system is \emph{essential} if it none of its elements are degenerate or trivial; and it is \emph{regular} if none of its elements are small.
\end{defn}

Note that degenerate and trivial elements are small, implying that regular separation systems are essential.
Note also that a separation system $\overrightarrow{S}$ which is a subsystem of some other separation system $\overrightarrow{S}'$ may be essential while containing elements which are trivial in $\overrightarrow{S}'$: whether a separation is trivial or not depends on the existence of a witness of the triviality, and after the deletion of all such witnesses the separation is not trivial any more.
On the other hand, being small or degenerate does not depend on the existence of a witness, and thus if an element is small (or degenerate) in some separation system, then it also is small (or degenerate) in all subsystems which still contain that element.

In this paper, only separation systems are considered that are a subsystem of some universe of separations.
Suprema and infima are always taken in the surrounding universe: So they are always defined, but not always contained in the separation system.

There are two main examples of submodular universes.
The first example are submodular universes arising from graphs, which are not needed in this paper.
The second main example are submodular universes arising from matroids or more general connectivity systems.

\begin{ex}[\cite{SeparationsOfSets}]\label{ex:universebips}
	Let $E$ be a set.
	Let $\mathcal{UB}(E)$ consist of all pairs $(A,B)$ of subsets of $E$ such that $A \cup B = E$ and $A \cap B = \emptyset$, and let $(A,B)^* = (B,A)$ and $(A,B) \leq (C,D)$ if $A \subseteq C$ (which is equivalent to $D \subseteq B)$.
	Then $\mathcal{UB}$ is a universe of separations with join $(A,B) \vee (C,D) = (A \cup C, B \cap D)$ and meet $(A,B) \wedge (C,D) = (A \cap C, B \cup D)$.
	There is only one small separation of $\mathcal{UB}(E)$, namely $(\emptyset, E)$, which is also the only trivial separation, and there is no degenerate separation.
	Just on its own, $\mathcal{UB}(E)$ does not have a natural order function, but if $E$ is e.g.\ the ground set of a matroid, then the connectivity function of that matroid is an example of an order function on $\mathcal{UB}(E)$.
	The universe $\mathcal{UB}(E)$ is isomorphic to the partially ordered set of subsets of $E$ via $(A,B)\mapsto A$.
	Via this isomorphism, all subsets of $E$ are separations of $\mathcal{UB}(E)$.
\end{ex}

\begin{figure}
	\begin{tikzpicture}
		\draw (0,0) rectangle (4,2);
		\draw (2,0) -- (2,2);
		\draw (1,0) node [anchor=north] {$A$} (2,0) node [anchor=north] {$\overrightarrow{s}$} (3,0) node [anchor=north] {$B$};
		\draw (2,1) pic {separrow};
	\end{tikzpicture}
	\caption{A separation $\protect\overrightarrow{s}=(A,B)$ in $\mathcal{UB}(V)$ is depicted by the little arrow from $A$ to $B$ on the line separating $A$ from $B$.}\label{fig:bipartition}
\end{figure}
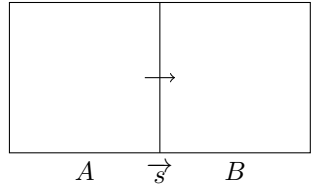

\begin{defn}\label{def:nested}\cite{AbstractSepSys}
	Two elements of a separation system $\overrightarrow{S}$ are \emph{nested} if they have orientations $\overrightarrow{s}$ and $\overrightarrow{t}$ such that $\overrightarrow{s}\leq \overrightarrow{t}$, and they \emph{cross} if they are not nested.
	A separation $\overrightarrow{s}$ \emph{points towards} a separation $\overrightarrow{t}$ if $\overrightarrow{s}\leq \overrightarrow{t}$ or $\overrightarrow{s}\leq \overleftarrow{t}$.
	A separation system is \emph{nested} if its elements are pairwise nested, and a \emph{tree set} if it is additionally essential.
	A set of separations is a \emph{star} if its elements are non-degenerate and all distinct elements $\overrightarrow{s}$ and $\overrightarrow{t}$ satisfy $\overrightarrow{s}\leq \overleftarrow{t}$.
\end{defn}

In \cite{AbstractSepSys}, officially it is only defined when a separations points towards an unoriented separation (not introduced here), but this is the obvious translation to oriented separations.

Throughout this paper, instead of with tangles we will work with the closely related notion of profiles.
The notion of a profile is more suited to the general setting of separation systems and in the context of a connectivity system on ground set $E$ (which is also $\mathcal{UB}$ with an order function) the difference is small and not relevant.

\begin{defn}\cite{ProfilesNew}
	Let $(S,\leq,\mathord{^*})$ be a separation system.
	A subset $O$ of $S$ such that $O\cap {\overrightarrow{s},\overleftarrow{s}}$ has exactly one element for all $\overrightarrow{s}\in S$ is an \emph{orientation} of $S$.
	It is \emph{consistent} if it does not contain two elements $\overleftarrow{s}$ and $\overrightarrow{t}$ which are not orientations of each other such that $\overrightarrow{s}<\overrightarrow{t}$.
	A \emph{profile} is a consistent orientation $P$ with the property that for any two elements $\overrightarrow{s}$ and $\overrightarrow{t}$ the separation $(\overrightarrow{s}\vee \overrightarrow{t})^*$ is not contained in $P$ (possibly because it does not exist in $S$).
	The latter property is also called the \emph{profile property}.
	A profile is \emph{regular} if it does not contain a co-small separation.
	Given a submodular universe $\mathcal{U}$, a $k$-profile $P$ of $\mathcal{U}$ is a profile of the separation system $S_k$.
	Two profiles that contain distinct orientations of a separation $\overrightarrow{s}$ are \emph{distinguished} by that separation, and a set $S'$ of separations distinguishes a set $\mathcal{P}$ of profiles if any two distinct profiles in $\mathcal{P}$ are distinguished by some separation in~$S'$.
\end{defn}

\begin{defn}\cite{GroheSchweitzer}
	Let $\mathcal{U}$ be a submodular universe and $k$ and $l$ elements of $\mathbb{N}$ such that $l\leq k$.
	For a $k$-profile $P$, the \emph{truncation of $P$ to an $l$-profile} is the intersection of $P$ with the separation system $S_l$.
	If $k\neq 0$, then the \emph{truncation of $P$}, without further mention of a second integer $l$, is the truncation of $P$ to a $k-1$-profile.
\end{defn}

If $\mathcal{U}$ is a subuniverse of $\mathcal{UB}(E)$ for some set $E$, then consistency is equivalent to a seemingly stronger property, as can be shown with the next lemma.

\begin{defn}\cite{ProfileDuality}
	A set of separations is called \emph{strongly consistent} if it does not contain elements $\overrightarrow{r}$ and $\overrightarrow{s}$ with $\overleftarrow{r}<\overrightarrow{s}$.
	Thus, an orientation is strongly consistent if and only if for every $\overrightarrow{s}\in O$ and every separation $\overrightarrow{r}$ with $\overrightarrow{r}\leq \overrightarrow{s}$ also $\overrightarrow{r}\in O$.
\end{defn}

\begin{lem}\cite[Lemma 7]{ProfileDuality}\label{stronglyconsistent}
	An orientation of a separation system is strongly consistent if and only if it is consistent and contains all small separations.
\end{lem}

\subsection{Universes of bipartitions}\label{sec:prelbips}

A \emph{universe of bipartitions} is a universe that is, for some ground set $E$, a subuniverse of $\mathcal{UB}(E)$ (see \cref{ex:universebips}).
If $\mathcal{UB}$ is identified with the set $2^E$ of subsets of $E$ via $(A,B) \mapsto A$, then a universe of bipartitions is a subset of $2^E$ that is closed under taking complements, unions and intersections.
Then the join of two separations $A$ and $B$ is their union $A \cup B$ and their meet is their intersection $A \cap B$.
The inverse of a set $A$ is then $E \setminus A$, which is its complement in $E$.
In this sense, a connectivity system is a special case of a submodular universe of bipartitions, and terminology for separation systems can be borrowed for subsets of $E$.
For example, two subsets $A,B$ of the ground set $E$ are nested if $A \subseteq B$ or $A \cap B = \emptyset$ or $B \subseteq A$ or $A \cup B = E$.
Furthermore, $A$ points towards $B$ if $A \subseteq B$ or $A \subseteq E \setminus B$ (equivalently, $A \cup B = E$).
A corner of $A$ and $B$ is one of the sets $A\cup B$, $A \cup (E \setminus B)$, $B \cup (A \setminus B)$ and $E \setminus (A \cap B)$.
Furthermore, when in the context of subsets of $E$ the terminology of separation systems competes with terminology of subsets, then usually the terminology for subsets will be used.

Recall that if a universe of bipartitions $\mathcal{U}$ has a small element, then this element is $\emptyset$.
So if $E \neq \emptyset$ then by \cref{stronglyconsistent} a consistent orientation of a subsystem $S$ of $\mathcal{U}$ is a subset $O$ of $S$ such that
\begin{itemize}
	\item for every $\overrightarrow{s} \in \overrightarrow{S}$, $O$ contains exactly one of $\overrightarrow{s}$ and $\overleftarrow{s}$ and
	\item if $\overrightarrow{r}$ and $\overrightarrow{s}$ are elements of $S$ with $\overrightarrow{r} < \overrightarrow{s} \in O$ then $\overrightarrow{r} \in O$.
\end{itemize}

\section{\texorpdfstring{Definition of $k$-pseudo\-flowers and $k$-flowers}{Definition of pseudoflowers and flowers}}\label{sec:defsbips}

Given the definitions of $k$-flowers in \cite{AikinOxley08} for finite polymatroids and \cite{ClarkWhittle13} for finite connectivity systems with a $k$-tangle, the following is a good provisional definition of finite $k$-flowers for (possibly infinite) connectivity systems:
A finite $k$-flower is an ordered partition $(P_1,\ldots,P_n)$ of the the ground set such that every partition class and the union of any two adjacent partition classes (where $P_1$ is adjacent to $P_n$) has connectivity exactly $k-1$.
Of course, asking that all these sets have connectivity $k$ instead of $k-1$ would work just as well.
As mentioned in the introduction, we want to be able to find maximal flowers.
In connectivity systems with infinite ground set, such maximal flowers might need to display infinitely many separations, and thus have infinitely many partition classes.
At least in the case that the infinite flower is more daisy-like than anemone-like, these petals have to be arranged in a cyclic order.
As opposed to finite cyclic orders, infinite cyclic orders need not be isomorphic just because they have the same size, and furthermore elements do not necessarily have adjacent elements.
To resolve this, the definition of a $k$-flower for infinite connectivity systems has a partition with a cyclically ordered index set (whereas the finite $k$-flowers are partitions on index set $\{1,\ldots,n\}$) and instead of asking for the union of adjacent petals to have connectivity exactly $k-1$, any union of a non-trivial interval of petals must have connectivity exactly $k-1$.
The definition is a special case of the following definition of a $k$-pseudo\-flower.

\begin{defn}\label{def:pseudoflowerbipswithlimits}
	A \emph{$k$-pseudo\-flower} is a partition $(P_i)_{i\in I}$ with a cyclically ordered index set $I$ such that the union of any interval of partition classes has order at most $k-1$.
	The sets $P_i$ are the \emph{petals} of the $k$-pseudo\-flower $(P_i)_{i\in I}$.
	Given an interval $I'$ of $I$, the set/separation $\bigcup_{i\in I'}P_i$ is denoted by $S(I')$.
	The separations $S(I')$ where $\emptyset \subsetneq I'\subsetneq I$ are the separations \emph{displayed} by the $k$-pseudo\-flower.
	
	A \emph{concatenation} of a $k$-pseudo\-flower $(P_i)_{i\in I}$ is $k$-pseudo\-flower $(Q_i)_{i\in I'}$ such that for every $i\in I$ there is an index $f(i)\in I'$ with $P_i\subseteq Q_{f(i)}$ and such that for all intervals $I''$ of $I'$, the set $f^{-1}(I'')$ is an interval of $I$\myfootnote{Recall that for this property of $f$ it is sufficient if $f$ is monotone and its image does not have exactly two elements.}.
	If a $k$-pseudo\-flower $\Phi$ is a concatenation of a $k$-pseudo\-flower $\Psi$, then this is denoted as $\Phi\leq \Psi$, and $\Psi$ is called an \emph{extension} of $\Phi$.
	
	A $k$-pseudo\-flower is a \emph{$k$-flower} if it has at least four petals\myfootnote{See the last paragraph before the next section for a remark about the lower bound on the number of petals.} and the union of any non-trivial interval of petals has order exactly $k-1$.
	It is a \emph{finite $k$-flower} if it is a $k$-flower with finitely many petals.
\end{defn}

Note that $\leq$ is a pre-order and that if two $k$-pseudo\-flowers $\Phi$ and $\Psi$ satisfy $\Phi \leq \Psi \leq \Phi$ then they are the same up to renaming of the index set.
Recall that in a pre-order, a \emph{maximal} element $\Phi$ is one such that, if $\Phi \leq \Psi$  then also $\Psi \leq \Phi$.
One standard method in infinite combinatorics to obtain a $\leq$-maximal $k$-flower is to prove that every $\leq$-chain of $k$-flowers has an upper bound and then apply Zorn's Lemma.
The most straightforward procedure to construct these upper bounds is to take the common refinement of all the partitions of the $k$-flowers in the chain and combine the cyclic orders into a cyclic order of the common refinement.
The resulting partition with cyclic order need not be a $k$-flower, as there might be separations displayed by the common refinement that are not displayed by any $k$-flower in the chain and thus cannot be guaranteed to have order $k$.
But every separation displayed by the resulting partition is a limit of separations that are displayed by $k$-flowers in the chain and, as the connectivity function is limit-closed, thus has order at most $k-1$.
That is why $k$-pseudo\-flowers are defined, and indeed in \cref{limitofpseudoflowers} it is shown that $\leq$-chains of $k$-pseudo\-flowers have upper bounds.

One important observation of connectivity systems is that every partition $\mathcal{P}$ of $E$ naturally induces another connectivity function $\lambda'$ via $\lambda'(\mathcal{Q})=\lambda(\bigcup \mathcal{Q})$ for $\mathcal{Q}\subseteq \mathcal{P}$.
So a finite $k$-flower of an infinite connectivity system is also a finite $k$-flower of some finite connectivity system in which the connectivity of unions of petals is preserved.
In particular the following lemma also holds for finite $k$-flowers in infinite connectivity systems:

\begin{lem}[{\cite[Theorem 1.1]{AikinOxley08}}]\label{daisyoranemone}
	In a finite $k$-flower, either all non-trivial unions of petals have order $k-1$ or the non-trivial unions of petals of order $k-1$ are exactly those whose index set is an interval of $I$.
\end{lem}

In \cite{AikinOxley08}, $k$-flowers of the first type are called $k$-anemones while flowers of the second type are called daisies.
Note that an infinite partition with cyclic order is a $k$-flower if and only if all its finite concatenations are finite $k$-flowers.
The finite concatenations can also be used to determine infinite $k$-daisies and $k$-anemones as follows.
Every finite concatenation of a $k$-flower is either a $k$-daisy, or a $k$-anemone.
For every two finite concatenations $\Psi$ and $\Psi'$ of a $k$-flower $\Phi$, there is a third finite concatenation $\Phi'$ of $\Phi$ that extends both $\Psi$ and $\Psi'$, and so $\Psi$ is a $k$-daisy if and only if $\Psi'$ is a $k$-daisy.
Thus the following definition arises.

\begin{defn}\label{def:anemoneanddaisy}
	A $k$-flower is a \emph{$k$-anemone} if all non-trivial unions of petals of finite concatenations have order $k-1$.
	A $k$-flower is a $k$-daisy if every finite concatenation of it has the property that the non-trivial unions of petals which have order $k-1$ are exactly those where the indices of the petals form an interval of $I$.
\end{defn}

Given that there are two types of finite $k$-flowers, and that the definition of infinite $k$-flowers is closely related to the characterisation of finite $k$-daisies, one might think that maybe there should be a definition of infinite $k$-anemones that is closer to the characterisation of finite $k$-anemones.
The most obvious choice here would be to let a $k$-anemone be a partition such that every non-trivial union of partition classes has order $k-1$.
It turns out (see \cref{infanemonedef}) that that is not a different possible definition of infinite $k$-anemone, but a property of the current definition of $k$-anemones.

Similarly, one might want to give distinct definitions of $k$-pseudo\-anemones and $k$-pseudo\-daisies and give a definition of $k$-pseudo\-anemones that is closer to the characterisation of finite $k$-anemones.
In particular, there are the following possibilities of $k$-pseudo\-flowers that are anemone-like:

\begin{defn}
	A $k$-pseudo\-anemone is a $k$-pseudo\-flower that can be concatenated into a $k$-anemone.
	A strong $k$-pseudo\-anemone is a $k$-pseudo\-flower for which all unions of partition classes have order at most $k-1$.
	For two strong $k$-pseudo\-anemones $\Phi$ and $\Psi$ denote $\Phi\leq_A \Psi$ if the partition of $\Phi$ is coarser than the partition of~$\Psi$.
\end{defn}

For a strong $k$-pseudo\-anemone, replacing the cyclic order of the partition with any other cyclic order yields again a strong $k$-pseudo\-anemone.
Furthermore, given two strong $k$-pseudo\-anemones $\Phi$ and $\Psi$, the relation $\Phi\leq_A\Psi$ holds if and only if there is a $k$-pseudo\-flower $\Psi'$ with the same partition as $\Psi$ such that $\Phi\leq \Psi'$.
In this sense for strong $k$-pseudo\-anemones the cyclic order does not really have a meaning, and therefore strong $k$-pseudo\-anemones should be compared by $\leq_A$ instead of by $\leq$.
Considering two strong $k$-pseudo\-anemones with the same partition to be the same strong $k$-pseudo\-anemone turns $\leq_A$ into a partial order.

One of the main results of the next section is that every $k$-pseudo\-anemone that can be concatenated into a $k$-anemone with $k+1$ many petals is a strong $k$-pseudo\-anemone (see \cref{moststrongkpa}) and thus that, as infinite $k$-flowers are the focus, these two definitions are essentially the same.
As a strong $k$-pseudo\-anemone cannot be concatenated into a $k$-daisy, a $k$-pseudo\-flower that can be concatenated into a $k$-anemone with sufficiently many petals is clearly anemone-like.

In \cite{AikinOxley08} and \cite{ClarkWhittle13} $k$-flowers are allowed to have less than four petals.
That allows for $k$-flowers which are not unambiguously classified as daisies or anemones.
As the focus of this paper is to translate the existing theory of $k$-flowers to infinite $k$-flowers, it seems reasonable to simplify the presentation by restricting the definition of a $k$-flower to partitions with at least four partition classes.

\section{\texorpdfstring{The order of different unions of petals in $k$-pseudo\-anemones}{The order of different unions of petals in pseudoanemones}}\label{sec:ordersinkpa}

As by definition every $k$-pseudo\-anemone can be concatenated into a $k$-anemone, the following lemma implies a statement about the order of unions of petals of the $k$-pseudo\-anemone which nearly form a union of petals of the $k$-anemone.

\begin{lem}\label{subsetofpetalsinfan}
	Let $\Phi$ be a finite $k$-anemone and $Q$ a petal of $\Phi$.
	Let $R$ be a non-empty union of petals of $\Phi$ such that $R\cap Q=\emptyset$ and $R\cup Q\neq E$.
	Then for any subset $S$ of $Q$, the order of $S\cup R$ does not depend on the choice of $R$.
\end{lem}
\begin{proof}
	Let $S$ be a subset of $Q$ and let $R_1$ and $R_2$ be candidates for $R$.
	In order to show $\lambda(R_1\cup S)=\lambda(R_2\cup S)$ it suffices to consider the case that $R_1$ is a subset of $R_2$.
	Every non-trivial union of petals of $\Phi$ has order $k-1$, so
	\begin{align*}
		\lambda(R_2\cup S)&\leq \lambda(R_1\cup S)+\lambda(R_2)-\lambda(R_1)= \lambda(R_1\cup S) =\lambda(E\setminus (R_1\cup S))\\
		&= \lambda((E\setminus (R_1\cup Q))\cup (Q\setminus S)) \\
		&\leq \lambda(E\setminus (R_1\cup Q)) + \lambda((E\setminus (R_2\cup Q))\cup (Q\setminus S)) - \lambda(E\setminus (R_2\cup Q)) \\
		&=\lambda((E\setminus (R_2\cup Q))\cup (Q\setminus S)) = \lambda(R_2\cup S).
	\end{align*}
	Thus $\lambda(R_1\cup S)=\lambda(R_2\cup S)$.
\end{proof}

So within a petal of an anemone, another connectivity function is induced.

\begin{lem}\label{muforpetal}
	Let $\Phi$ be a finite $k$-anemone with distinct petals $Q$ and $R$.
	Then the map $\mu:2^Q\rightarrow \mathbb{N}$ defined by $\mu(S)=\lambda(S\cup R)$ is submodular, symmetric, limit-closed, bounded from below by $k-1$, and does not depend on the choice of $R$.
\end{lem}
\begin{proof}
	As $\lambda$ is submodular and limit-closed, $\mu$ is also submodular and limit-closed.
	By \cref{subsetofpetalsinfan} $\mu$ does not depend on the choice of $R$ and, for all $S\subseteq Q$,
	\begin{align*}
		\mu(S)=\lambda(S\cup R)=\lambda(S \cup (E \setminus (R \cup Q)))=\lambda((Q \setminus S) \cup R)=\mu(Q \setminus S)
	\end{align*}
	so $\mu$ is symmetric.
	Also $\mu(\emptyset)=k-1$, so as $\mu$ is symmetric and submodular it follows that $\mu(S)\geq k-1$ for all $S\subseteq Q$.
\end{proof}

In the special case where $S$ a petal of a $k$-pseudo\-anemone extending the anemone, then $\mu(S)=k-1$, as the next lemma shows.

\begin{lem}\label{muforkpf}
	Let $\Phi'$ be a $k$-pseudo\-anemone which has a concatenation into a $k$-anemone $\Phi$.
	Then the union of any petal $S$ of $\Phi'$ and any petal $R$ of $\Phi$ not containing $S$ has order $k-1$.
\end{lem}
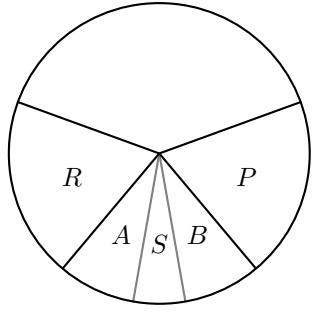
\begin{figure}
	\begin{tikzpicture}[thick]
		\labelofset{230}{160}{$R$}
		\labelofset{-100}{-130}{$A$}
		\labelofset{-80}{-100}{$S$}
		\labelofset{-50}{-80}{$B$}
		\labelofset{20}{-50}{$P$}

		\newedge[gray]{-100}
		\newedge[gray]{-80}		
		\newedge{-50}
		\newedge{230}
		\newedge{160}
		\newedge{20}
		
		\basicflower
	\end{tikzpicture}
	\caption{Notation from the proof of \cref{muforkpf}.}
	\label{fig:muforkpf}
\end{figure}
\begin{proof}
	Denote the petal of $\Phi$ which contains $S$ by $Q$.
	By \cref{subsetofpetalsinfan} it suffices to consider the case that $R$ is adjacent to $Q$.
	Denote the neighbour of $Q$ in $\Phi$ which is not $R$ by $P$.
	Deleting $S$ from $Q$ yields two, possibly empty, unions of intervals of petals of $\Phi'$; denote them by $A$ and $B$ such that $A$ is adjacent to $R$ and $B$ is adjacent to $P$.
	Then
	\begin{align*}
		\lambda(R\cup S)&\leq \lambda(R\cup A\cup S)+\lambda(R\cup S\cup B)-\lambda(R\cup A\cup S\cup B)\\
		&\leq \lambda(R\cup S\cup B)= \lambda(S\cup B\cup P) \leq k-1
	\end{align*}
	where the equality holds by \cref{subsetofpetalsinfan}.
	Thus $\lambda(R\cup S)=k-1$ by \cref{muforpetal}.
\end{proof}

In order to deduce from \cref{muforkpf} that every $k$-pseudo\-anemone extending an anemone with sufficiently many petals is a strong $k$-pseudo\-anemone, the following elementary property of submodular functions is needed.

\begin{lem}\label{witnessforbiggerconnectivity}
	Let $E$ be a finite set and $\lambda:2^E\rightarrow \mathbb{Z}$ submodular. Let $X$ and $Y$ be subsets of $E$ such that $X\subseteq Y$ and $\lambda(X)<\lambda(Y)$. Then there is $e\in Y\setminus X$ such that $\lambda(X)<\lambda(X+e)$.
\end{lem}
\begin{proof}
	Define a function $\mu:2^{Y\setminus X}\rightarrow \mathbb{Z}$ via $\mu(A)=\lambda(A\cup X)-\lambda(X)$.
	Then $\mu$ is submodular and satisfies $\mu(\emptyset)=0$.
	Furthermore $\mu(Y\setminus X)\geq 1$, so there is a minimal set $X'\subseteq Y\setminus X$ such that $\mu(X')\geq 1$.
	Let $e\in X'$.
	Then by submodularity
	\begin{displaymath}
		\mu(X') \leq \mu(X' - e) + \mu(e) - \mu(\emptyset)
	\end{displaymath}
	so minimality of $X'$ implies $\mu(e)\geq \mu(X')$ and thus $X'=\{e\}$.
	Hence $\lambda(X+e)=\mu(e)+\lambda(X)\geq \lambda(X)+1$.
\end{proof}

\begin{lem}\label{nearlystrongkpa}
	Let $\Phi'$ be a $k$-pseudo\-flower which has a concatenation $\Phi$ into a finite $k$-anemone.
	Every union of petals of $\Phi'$ which either contains a petal of $\Phi$ or is disjoint from a petal of $\Phi$ has order at most $k-1$.
\end{lem}
\begin{proof}
	By symmetry of $\lambda$ it suffices to consider unions of petals of $\Phi'$ which contain a petal of $\Phi$.
	
	Let $R$ be a petal of $\Phi$.
	Denote the set of petals of $\Phi'$ that are disjoint from $R$ by $\mathcal{Q}$ and let $\mu'$ be the map defined on $\mathcal{Q}$ via $\mu'(\mathcal{R})=\lambda(\bigcup \mathcal{R} \cup R)$.
	Then $\mu'(\emptyset)=k-1$ and $\mu'(Q)=k-1$ for all $Q\in \mathcal{Q}$ by \cref{muforkpf}, so $\mu'(\mathcal{R})\leq k-1$ for all finite subsets $\mathcal{R}$ of $\mathcal{Q}$ by \cref{witnessforbiggerconnectivity}.
	As $\lambda$ is limit-closed, also $\mu'$ is limit-closed and thus $\mu'(\mathcal{R})\leq k-1$ for all subsets $\mathcal{R}$ of $\mathcal{Q}$ by \cref{fintoinfconn}.
\end{proof}

As a corollary, we obtain the following theorem:

\begin{thm}\label{moststrongkpa}
	Let $\Phi'$ be a $k$-pseudo\-flower which can be concatenated into a finite $k$-anemone $\Phi$ with at least $k+1$ many petals.
	Then $\Phi'$ is a strong $k$-pseudo\-anemone.
\end{thm}
\begin{proof}
	All unions of at most $k$ many petals of $\Phi'$ are disjoint from a petal of $\Phi$ and thus have order at most $k-1$ by \cref{nearlystrongkpa}.
	As $\lambda$ is a connectivity function, this implies that all finite unions of petals of $\Phi'$ have order at most $k-1$.
	So by \cref{fintoinfconn} all unions of petals of $\Phi'$ have order at most $k-1$.
\end{proof}
\begin{rem}
	In particular, every $k$-pseudo\-flower which can be concatenated into a $k$-anemone with at least $k+1$ many petals cannot be concatenated into a $k$-daisy.
\end{rem}

Now we can show that indeed the other definition of infinite $k$-anemone mentioned in \cref{sec:defsbips} is equivalent to the definition currently in use.

\begin{cor}\label{infanemonedef}
	Let $\Phi$ be an infinite $k$-anemone.
	Then every union of petals of $\Phi$ has order $k-1$.
\end{cor}
\begin{proof}
	By the previous theorem, all unions of petals have order at most $k-1$.
	Also every concatenation of $\Phi$ into a finite $k$-flower is a finite $k$-anemone.
	So every non-trivial union of petals that is also a union of petals  of a finite concatenation has order exactly $k-1$.
	Assume for a contradiction that there is a non-trivial union of petals $P$ whose order is less than $k-1$.
	As the connectivity function is assumed to be limit-closed, $P$ can be chosen minimal with that property.
	Then $P$ contains distinct petals $Q$ and $Q'$ such that the interval $[Q,Q']$ has at least $k+1$ elements and such that $]Q',Q'[$ contains a petal that is not contained in $P$.
	Let $R$ be the union of $]Q,Q']$.
	Then there is a finite $k$-anemone $\Phi'$  with at least $k+1$ petals, one of which is $Q$, that is a concatenation of $\Phi$ such that the union of $]Q',Q[$ is contained in a petal of $\Phi'$.
	Now $\mu$ as in \cref{muforpetal} can be defined for $\Phi'$ and its petal containing the union of $]Q',Q[$.
	Then
	\begin{align*}
		\lambda(P)&\geq \lambda(P\cap R) + \lambda(P\cup R) - \lambda(R) = k-1 + \lambda(E \setminus (P \cup R)) - (k-1)\\
		&= \mu(E \setminus (P \cup Q \cup R)) \geq k-1
	\end{align*}
	a contradiction to $\lambda(P)<k-1$.
\end{proof}

\section[Finding maximal pseudoflowers]{\texorpdfstring{Finding maximal $k$-pseudo\-flowers and maximal strong $k$-pseudo\-flowers}{Finding maximal pseudoflowers and maximal strong pseudoanemones}}

In this chapter we show that $k$-pseuodflowers can be extended to $\leq$-maximal $k$-pseudo\-flowers.
For that, we first find an upper bound of a $\leq$-chain of $k$-pseudo\-flowers by taking the common refinement of the partitions and defining a suitable cyclic order.
The statement then follows from an application of Zorn's Lemma.

\begin{lem}\label{limitofpseudoflowers}
	Every $\leq$-chain of $k$-pseudo\-flowers with cyclic orders has an upper bound.
\end{lem}
\begin{proof}
	Let $(\Phi_j)_{j\in J}$ be a $\leq$-chain of $k$-pseudo\-flowers.
	For every $e\in E$ let $P_e$ be the intersection of all petals of the $\Phi_j$ which contain $e$.
	The sets $P_e$ are the petals of $\Psi$.
	In order to define a cyclic order on them, let $P_e$, $P_f$ and $P_g$ be distinct petals of $\Psi$ and let $j$ be an index of $J$ such that $e$, $f$ and $g$ are contained in distinct petals $P'_e$, $P'_f$, and $P'_g$ of $\Phi_j$.
	Then the cyclic order of $P'_e$, $P'_f$ and $P'_g$ does not depend on the choice of $j$, so it is well-defined to put $P_e$, $P_f$ and $P_g$ in the same order as $P'_e$, $P'_f$ and $P'_g$ and this induces a cyclic order on the set of petals of $\Psi$.
	Also clearly every $\Phi_j$ is a concatenation of $\Psi$.
\end{proof}
\begin{cor}\label{existenceleqmaxkpf}
	For every $k$-pseudo\-flower $\Phi$ there is a $\leq$-maximal $k$-pseudo\-flower $\Psi$ such that $\Phi\leq \Psi$.\qed
\end{cor}

Just as there are $\leq$-maximal $k$-pseudo\-flowers, there also are $\leq_A$-maximal strong $k$-pseudo\-anemones.
That fact does not follow immediately from \cref{existenceleqmaxkpf}, as $\Phi\leq \Psi$ for a strong $k$-pseudo\-anemone $\Phi$ does not necessarily imply that $\Psi$ is a strong $k$-pseudo\-anemone as well.
Thus the proof will take a detour via certain subsets of the power set of $E$, of which there are maximal ones by Zorn's Lemma, and then show that the resulting set can be transformed back into a strong $k$-pseudo\-anemone.
The transformation back is done separately in \cref{settoskpa}.

\begin{lem}\label{settoskpa}
	Let $\mathcal{A}$ be a subset of $2^E$ containing $\emptyset$ such that for all elements $A$ and $B$ of $\mathcal{A}$ the sets $E\setminus A$ and $A\cap B$ are contained in $\mathcal{A}$ and $\lambda(A)\leq k-1$.
	Then there is a partition of $E$ such that every union of partition classes has order at most $k-1$ and every element of $\mathcal{A}$ is a union of partition classes.
\end{lem}
\begin{proof}
	Note that $\mathcal{A}$ is closed under finite unions and finite intersections of its elements.
	For a finite subset $F$ of $E$ define $\mathcal{S}_F=\{A\in \mathcal{A}\colon A\cap F=\emptyset\}$.
	Then by Zorn's Lemma the set $\mathcal{S}_F$ has a maximal element $S_F$ in $2^E$ whose order is at most $k-1$.
	As $\mathcal{S}_F$ is closed under taking unions, $S_F=\bigcup \mathcal{S}_F$.
	If $F=\{e\}$, then denote $S_{\{e\}}$ by $S_e$ and let $\mathcal{Q}=\{E\setminus S_e\colon e\in E\}$.
	As $F\subseteq E \setminus S_F$ for all finite subsets $F$ of $E$, in order to show that $\mathcal{Q}$ is a partition of $E$ it suffices to show that for elements $e$ and $f$ of $E$ either $S_e=S_f$ or $E=S_e\cup S_f$.
	If $e\in A \Leftrightarrow f\in A$ holds for all $A$ in $\mathcal{A}$, then $\mathcal{S}_{\{e\}}=\mathcal{S}_{\{f\}}$ and hence $S_e=S_f$.
	Otherwise there is a set $A$ in $\mathcal{A}$ that contains, say, $f$ but not $e$.
	In this case, $A\subseteq S_e$ and $E\setminus A\subseteq S_f$, so $E=A\cup (E\setminus A)\subseteq S_e\cup S_f$.
	
	In order to show that every union of elements of $\mathcal{Q}$ has order at most $k-1$, it suffices by \cref{fintoinfconn} to show that every finite union of elements in $\mathcal{Q}$ has order at most $k-1$.
	By the definition of $\mathcal{Q}$ that is the same as to show for every finite subset $F$ of $E$ that $\bigcup_{e\in F}(E\setminus S_e)$ has order at most $k-1$.
	For this let $X$ and $Y$ be sets whose disjoint union is $F$.
	Then $\mathcal{S}_F\subseteq \mathcal{S}_X\cap \mathcal{S}_Y$, so $S_F\subseteq S_X\cap S_Y$.
	Also for $e\in S_X\cap S_Y$ there are elements $A_1$ and $A_2$ of $\mathcal{A}$ such that $e\in A_1\in \mathcal{S}_X$ and $e\in A_2\in \mathcal{S}_Y$.
	Then $A_1\cap A_2\in \mathcal{A}$ and $e\in A_1\cap A_2\in \mathcal{S}_F$, so $e\in S_F$.
	Thus $S_F=S_X\cap S_Y$.
	By induction this implies $S_F=\bigcap_{e\in F}S_e$.
	So
	\begin{displaymath}
		\lambda(\bigcup_{e\in F}(E\setminus S_e))=\lambda(E\setminus \bigcap_{e\in F}S_e)=\lambda(E\setminus S_F)\leq k-1.\qedhere
	\end{displaymath}
\end{proof}

\begin{lem}\label{limitofstrongpseudoanemones}
	For every strong $k$-pseudo\-anemone $\Phi$ there is a $\leq_A$-maximal strong $k$-pseudo\-anemone $\Psi$ such that $\Phi\leq_A\Psi$.
\end{lem}
\begin{proof}
	The set of separations displayed by $\Phi$ only has elements of order at most $k-1$ and is closed under taking finite unions, finite intersections and complements.
	By Zorn's Lemma there is a maximal set $\mathcal{A}$ of subsets of $E$ which has these properties and contains all separations displayed by $\Phi$.
	By its maximality $\mathcal{A}$ contains both $\emptyset$ and $E$.
	Then by \cref{settoskpa} there is a partition of $E$ such that every union of partition classes has order at most $k-1$ and every element of $\mathcal{A}$ is a union of partition classes.
	Choosing an arbitrary cyclic order turns the partition into a strong $k$-pseudo\-anemone, and by maximality of $\mathcal{A}$ that strong $k$-pseudo\-anemone is $\leq_A$-maximal.
\end{proof}

\section{Combining distinct extensions of an anemone}
Extensions of a $k$-anemone can in general be quite different.
But is has already been shown in \cref{sec:ordersinkpa} that $k$-pseudo\-anemones that can be concatenated into a $k$-anemone with at least $k+1$ many petals have additional properties.
This section shows another property of $k$-anemones with at least $k+1$ many petals:
all their extensions can be combined into one strong $k$-pseudo\-anemone.
The next two lemmas show this for extensions that subdivide only one selected petal of the $k$-anemone.

\begin{lem}\label{partitionformu}
	Let $\Phi$ be a $k$-anemone and $Q$ a petal of $\Phi$.
	There is a partition of $Q$ such that the subsets $S$ of $Q$ with $\mu(S)=k-1$ are exactly the unions of partition classes.
\end{lem}
\begin{proof}
	By \cref{muforpetal}, $\mu(S)\geq k-1$ for all subsets $S$ of $Q$.
	Thus if $S_1$ and $S_2$ are subsets of $Q$ with $\mu(S_1)=\mu(S_2)=k-1$, then by submodularity of $\mu$ also $\mu(S_1\cup S_2)=\mu(S_1\cap S_2)=k-1$.
	
	\Cref{settoskpa} can be applied to the limit-closed\ function $\mu$ and the set $\mathcal{A}$ of all $S\subseteq Q$ with $\mu(S)=k-1$.
	As $\mu$ is bounded from below by $k-1$, all unions of partition classes of the obtained partition $\mathcal{Q}$ have order exactly $k-1$, so $\mathcal{A}$ is the set of unions of partition classes.
\end{proof}

\begin{lem}\label{skpatosetswithmusmall}
	Let $\Phi$ be an anemone, $Q$ a petal and $\mathcal{Q}$ a partition of $Q$.
	Denote the common refinement of $\mathcal{Q}\cup \{E\setminus Q\}$ and the set of petals of $\Phi$ by $\mathcal{Q}'$.
	Then $\mathcal{Q}'$ is the set of petals of a $k$-pseudo\-flower if and only if $\mathcal{Q}'$ is the set of petals of a strong $k$-pseudo\-anemone if and only if every element $S$ of $\mathcal{Q}$ satisfies $\mu(S)=k-1$.
\end{lem}
\begin{proof}
	If $\mathcal{Q}'$ is the set of petals of a $k$-pseudo\-flower, then by \cref{nearlystrongkpa} it is the set of petals of a strong $k$-pseudo\-anemone, and by \cref{muforkpf} also $\mu(S)=k-1$ for all elements $S$ of $\mathcal{Q}$.
	Now consider the case that $\mu(S)=k-1$ for all elements $S$ of $\mathcal{Q}$.
	Pick a cyclic order of $\mathcal{Q}'$ which can be concatenated to $\Phi$.
	In order to show that $\mathcal{Q}'$ together with this cyclic order is a $k$-pseudo\-flower, it suffices by \cref{muforpetal} and the symmetry of $\lambda$ to show that $\lambda(S)=k-1$ for all unions $S$ of elements of $Q$.
	By \cref{partitionformu} $\mu(S)=k-1$.
	Let $R_1$ and $R_2$ be distinct petals of $\Phi$ which are distinct from $Q$.
	Then by \cref{muforpetal}
	\begin{displaymath}
		\lambda(S)\leq \lambda(S\cup R_1)+\lambda(S\cup R_2)-\lambda(S\cup R_1\cup R_2)=\mu(S)=k-1. \qedhere
	\end{displaymath}
\end{proof}

\begin{cor}\label{subdivideonepetal}
	Let $\Phi$ be a $k$-anemone and $Q$ a petal of $\Phi$.
	Let $\mathcal{S}$ be the set of partitions of $k$-pseudo\-flowers $\Psi$ such that $\Phi\leq_A \Psi$ and all petals of $\Phi$ except possibly $Q$ are also petals of $\Psi$.
	Then all elements of $\mathcal{S}$ are partitions of strong $k$-pseudo\-anemones and $\mathcal{S}$ has a $\leq_A$-biggest element.
\end{cor}
\begin{proof}
	By \cref{partitionformu} and \cref{skpatosetswithmusmall}.
\end{proof}

These refinements of the individual petals can be combined.

\begin{lem}\label{commonrefinement}
	Let $\Phi$ be a $k$-anemone with at least $k+1$ many petals and denote its partition by $\mathcal{Q}$.
	For each petal $P$ of $\Phi$ let $\mathcal{Q}_P$ be a partition of $P$ such that the common refinement of $\mathcal{Q}_P\cup \{E\setminus P\}$ and $\mathcal{Q}$ is a strong $k$-pseudo\-anemone.
	Then the common refinement of all partitions $\mathcal{Q}_P\cup \{E\setminus P\}$ is a strong $k$-pseudo\-anemone.
\end{lem}
\begin{proof}
	By \cref{moststrongkpa} it suffices to show that there is a cyclic order which turns the common refinement of all partitions $\mathcal{Q}_P\cup \{E\setminus P\}$ into a $k$-pseudo\-flower which is an extension of $\Phi$.
	For that, it suffices to show that for all distinct petals $P$ and $P'$ of $\Phi$ the common refinement of $\mathcal{Q}_P\cup \{E\setminus P\}$, $\mathcal{Q}_{P'}\cup \{E\setminus P'\}$ and $\mathcal{Q}$ is a strong $k$-pseudo\-anemone.
	In order to show the latter, let $S$ be a union of elements of $\mathcal{Q}_P$, $S'$ a union of elements of $\mathcal{Q}_{P'}$ and $Q$ a non-empty union of petals of $\Phi$ which contains neither $P$ nor $P'$.
	Then
	\begin{align*}
		\lambda(Q\cup S\cup S')&\leq \lambda(Q\cup S)+\lambda(Q\cup S')-\lambda(Q)\leq k-1
	\end{align*}
	where $\lambda(Q)=k-1$ by \cref{infanemonedef}, and 
	\begin{align*}
		\lambda(S\cup S')&\leq \lambda(P\cup S')+\lambda(P'\cup S)-\lambda(P\cup P')\leq k-1.\qedhere
	\end{align*}
\end{proof}

\begin{thm}\label{finestrefinement}
	For every $k$-anemone $\Phi$ with at least $k+1$ many petals there is a strong $k$-pseudo\-anemone $\Psi$ such that $\Phi\leq_A\Psi$ and $\Phi\leq\Psi'\Rightarrow \Psi'\leq_A\Psi$ for all $k$-pseudo\-flowers $\Psi'$.
\end{thm}
\begin{proof}
	For every petal $Q$ there is by \cref{subdivideonepetal} a finest partition into which a $k$-pseudo\-flower can split that petal and by \cref{commonrefinement} all these partitions can be combined into a strong $k$-pseudo\-anemone $\Psi$ which has the required properties.
\end{proof}
\section{Distinguishing profiles}\label{sec:bipslimitsprofiles}

In this section we try to find $k$-pseudoflowers that distinguish as many profiles as possible.
In \cite{ClarkWhittle13}, $k$-flowers are not compared by $\leq$ but by a pre-order that relies on an equivalence relation of separations of order at most $k-1$.
Namely $\Phi$ is less than or equal to $\Psi$ if every equivalence class that is \emph{displayed} by $\Phi$, in the sense that one of its elements is displayed by $\Phi$, is also displayed by $\Psi$.
In this paper, we do not work with the equivalence relation from \cite{ClarkWhittle13} but with a coarser equivalence relation.
For this relation, and throughout this section, let $\mathcal{P}$ be a set of $k$-profiles which have the same truncation $P_0$ to a $k-1$-profile.
Define two separations of order at most $k-1$ to be equivalent if they are contained in the same elements of $\mathcal{P}$\myfootnote{That is, if they have the same image under the map $\phi$ defined in \cref{sec:abstractionfromunderlyingsepsys}.}.
By using this coarser equivalence relation we can avoid several of the technical issues that arise in \cite{ClarkWhittle13} when proving the existence of maximal $k$-flowers, allowing us to concentrate on the problems arising from the infinite setting.
But by using the coarser equivalence relation, also something is lost; in particular, a tree-decomposition with flowers that displays all equivalence classes except two represents the tree-like structure of a connectivity system better if the equivalence relation is finer.

So the goal of this section is to show that there are $k$-pseudo\-flowers distinguishing as many elements of $\mathcal{P}$ as possible, that is to find maximal elements of the following pre-order.

\begin{defn}
	For $k$-pseudo\-flowers $\Phi$ and $\Psi$ let $\Phi\preccurlyeq \Psi$ if every two profiles in $\mathcal{P}$ that are distinguished by the union of an interval of $\Phi$ are also distinguished by the union of an interval of $\Psi$.
\end{defn}

For a $k$-pseudo\-flower $\Phi$ and $n\in \mathbb{N}$ we will use the term \emph{$\Phi$ distinguishes $n$ elements of $\mathcal{P}$} as shorthand for the property that there are at least three elements of $\mathcal{P}$ that are pairwise distinguished by sets displayed by $\Phi$.
We want to extend every $k$-pseudo\-flower to a $\preccurlyeq$-maximal $k$-pseudoflower.
As $\preccurlyeq$ is transitive, it suffices to extend every $k$-pseudo\-flower $\Phi$ with index set $I$ that distinguishes at least three elements of $\mathcal{P}$ to a $\preccurlyeq$-maximal $k$-pseudo\-flower.
Consider the case that $\Phi$ is not $\preccurlyeq$-maximal, then there is a $k$-pseudo\-flower $\Psi$ such that $\Phi\preccurlyeq \Psi$, and such that there are two profiles $P_1$ and $P_2$ in $\mathcal{P}$ that are distinguished by $\Psi$ but not by $\Phi$.
If both $P_1$ and $P_2$ \emph{point to a petal of $\Phi$} in the sense that they contain the inverse of the petal, then they point to the same petal of $\Phi$ and $\Phi$ can be extended to a $k$-pseudo\-flower distinguishing $P_1$ and $P_2$.
This is proven in \cref{divideonepetal}.
This proof is heavily inspired by the corresponding proof for finite connectivity systems (\cite{ClarkWhittle13}) but simpler because the equivalence relation of separations is simpler.
A very similar version of this proof for vertex separations can also be found in \cite{EH:vertexflowers}.

For the proof, we use two additional notions.
First, say two separations $S$ and $T$ \emph{cross properly} if for all orientations $S'$ of $S$ and $T'$ of $T$ there is some element of $\mathcal{P}$ containing both $S'$ and $T'$.
Note that this is equivalent to all four corners of $S$ and $T$ distinguishing elements of $\mathcal{P}$.

Second, if $P \in \mathcal{P}$ points towards a petal $Q$, then $Q$ already determines which unions of intervals of $\Phi$ are contained in $P$.
Even if $P$ does not point towards a petal, it can be pinpointed in the cycle completion of the index set as follows.
For distinct cuts $u$ and $v$ of $I$ denote $S([u,v] \cap I)$ (that is, the union of all $P_i$ with $i \in [u,v] \cap I$) by $S(u,v)$.
A $k$-profile that does not point towards a petal of a $k$-pseudo\-flower $\Phi$ on index set $I$ is \emph{located at} a cut $v\in C(I)\setminus I$ if either $S(x,v)\in P$ for all $x\in C(I)\setminus I - v$ or $S(v,x)\in P$ for all $x\in C(I)\setminus I-v$.
Here, the cut $v$ nearly determines which unions of intervals are contained in $P$.
Indeed, if $I'$ is an interval of $I$ that is of the form $[u,w] \cap I$ with $v\notin [u,w]$ then $S(u,w)$ is a subset of both $S(u,v)$ and $S(v,w)$ and thus is contained in $P$.

\begin{lem}
	Let $P$ be a $k$-profile and $\Phi$ a $k$-pseudo\-flower such that $P$ does not point towards a petal of $\Phi$.
	Then there is a unique $v\in C(I)\setminus I$ such that $P$ is located at $v$.
\end{lem}
\begin{proof}
	Let $w_1\in C(I)\setminus I$.
	Let $I'\subseteq I$ be the set of those $i\in I$ for which there is $w_1\in C(I)$ with $i\in [w_1,w_2]$ and $S(i_1,i_2) \in P$.
	If $I'=\emptyset$ or $I'=I$ then $P$ is located at $w_1$.
	Otherwise $I'$ is a non-trivial interval of $I$ and there is a unique $w_2\in C(I)\setminus I$ such that $I'=[w_1,w_2]\cap I$.
	First consider the case that $S(w_1,w_2)$ is contained in $P$ and let $u\in C(I)\setminus I-w_2$.
	If $u=w_1$ then $S(u,w_2)\in P$.
	If $u\in \lrbracket w_1,w_2\rlbracket \setminus I$ then $S(u,w_2)\subseteq S(w_1,w_2)$ and thus $S(u,w_2)\in P$.
	If $u\in \lrbracket w_2,w_1\rlbracket$, then $S(w_1,u)$ is not contained in $P$, so its complement is.
	Then $S(u,w_2)$ is the union of $S(u,w_1)$ and $S(w_1,w_2)$ and thus is contained in $P$.
	Similarly, if $S(w_2,w_1)$ is contained in $P$, then $S(w_2,u)\in P$ for all $u\in C(I)\setminus I - w_1$.
	So $P$ is located at $w_2$.
	
	Assume for a contradiction that there are two distinct cuts $v$ and $v'$ such that $P$ is located at both $v$ and $v'$.
	Without loss of generality $S(v_1,v_2)\in P$.
	Then for all $u\in C(I)\setminus I - v_1 - v_2$, $S(v_1,u)$ and $S(u,v_2)$ are contained in $P$.
	If there is $u\in \lrbracket v_2, \rlbracket v_1\setminus I$ then both $S(v_1,u)$ and $S(u,v_2)$ are contained in $P$, but the union of these two sets is $E$, a contradiction.
	So $[v_2,v_1]$ contains exactly one $i\in I$ and $P$ points towards $P_i$, a contradiction.
\end{proof}

Now we can show that under some circumstances a $k$-pseudoflower that is not $\preccurlyeq$-maximal is also not $\leq$-maximal.

\begin{lem}\label{divideonepetal}
	Let $\Phi$ be a $k$-pseudo\-flower distinguishing at least three elements of $\mathcal{P}$.
	Also let $S$ be a separation which properly crosses some petal $P_i$ of $\Phi$.
	Then there is an extension $\Psi$ of $\Phi$ which has $P_i\cap S$ and $P_i\cap (E \setminus S)$ as petals and whose other petals are also petals of $\Phi$.
\end{lem}
\begin{proof}
	We start by constructing from $S$ a separation that still properly crosses $P$ but is otherwise more nested with the separations displayed by $\Phi$.
	Let $P_1$ be an element of $\mathcal{P}$ that contains $S$ and $E \setminus P$.
	If $\Phi$ distinguishes two elements of $\mathcal{P}$ that contain $S$ and $P$, then let $P_3$ be an element of $\mathcal{P}$ that contains $E \setminus S$ and $P$.
	In this case, $\Phi$ distinguishes some element $P_2$ of $\mathcal {P}$ that contains $S$ and $P$ from $P_3$.
	Otherwise, let $P_2$ be an element of $\mathcal{P}$ that contains $S$ and $P$.
	As $\Phi$ distinguishes three elements of $\mathcal{P}$, it distinguishes some $P_3 \in \mathcal{P}$ from $P_1$ and $P_2$.
	In this case, $P_3$ has to contain $E \setminus P$ and $E \setminus S$.
	So in both cases there are elements $P_2$ and $P_3$ of $\mathcal{P}$ containing $P$ such that $P_2$ contains $S$, $P_3$ contains $E \setminus S$ and $\Phi$ distinguishes $P_2$ and $P_3$.
	As every element of $\mathcal{P}$ points towards a petal or is located at some cut, there is an interval $I'$ of the index set $I$ of $\Phi$ such that $T:=\bigcup_{i' \in I'} P_{i'}$ distinguishes $P_2$ and $P_3$.
	By several applications of the profile property, and taking the complement if necessary, without loss of generality it can be assumed that $i$ is the smallest element of the linear order of $I \setminus I'$.
	Then also $T \cup P_i$ is a separation displayed by $\Phi$ that distinguishes $P_2$ and $P_3$.
	
	Fist consider the case that $T$ is contained in $P_2$.
	In this case, if $S \cap T$ has order at most $k-1$, then by the profile property it distinguishes $P_2$ and $P_3$ and thus has order exactly $k-1$.
	So by submodularity $S \cup T$ has order at most $k-1$ and is thus contained in $P_2$.
	Similarly $S' = (S \cup T) \cap (T \cup P_i)$ is a separation of order at most, and thus exactly, $k-1$ that is contained in $P_2$ but not in $P_3$.
	Furthermore $S \cap P_i = S' \cap P_i$, which implies that $S'$ properly crosses $P_i$.
	By symmetry, if $T$ is contained in $P_3$ then $S' = ((E \setminus S) \cup T) \cap (T \cup P_i)$ properly crosses $P_i$ and $\{P_i \cap S', P_i \setminus S'\} = \{P_i \cap S, P_i \setminus S\}$.
	
	Let $i_1$ and $i_2$ be two sets that are not elements of $I$ and let $I'' = I - i + i_1 + i_2$.
	Define a cyclic order on $I''$ such that, for elements $j_1$ and $j_2$ in $I - i$, $i_1 \in \ ]j_1,j_2[$ if and only if $i_2 \in \ ]j_1,j_2[$ if and only if $i \in \ ]j_1,j_2[$ and $i_1$ is the predecessor of $i_2$.
	Let $P_{i_1} = P_i \cap S'$ and $P_{i_2} = P_i \setminus S'$.
	In order to show that $\Psi := (P_i)_{i \in I''}$ is a $k$-pseudo\-flower, let $J$ be an interval of $I''$.
	If $J$ contains neither $i_1$ nor $i_2$ or both $i_1$ and $i_2$, then the union of the corresponding partition classes is the union of an interval of partition classes of $\Phi$ and thus has order at most $k-1$.
	So it suffices to consider the case that $J$ contains $i_1$ but not $i_2$.
	
	If $I' \subseteq J$, then $\bigcup_{j \in J} P_j= T \cup S'$.
	As $T$ and $S'$ distinguish elements of $\mathcal{P}$ and some element of $\mathcal{P}$ contains neither $T$ nor $S'$, the set $T \cup S'$ has order at most $k-1$.
	Similarly, if $J - I_1 \subseteq I'$ then $\bigcup_{j \in J} P_j = (T \cup P_i) \cap S'$ which has order at most $k-1$.
	So $\Psi$ is a $k$-pseudoflower, and it extends~$\Phi$.
\end{proof}

In theory, if a $k$-pseudo\-flower contains an infinite $(k-1)$-pseudo\-flower as a concatenation, then the common truncation of the elements in $\mathcal{P}$ need not point towards a petal of the $(k-1)$-pseudo\-flower.
Fortunately, this problem does not occur, which we will show with the next few lemmas.
Assume for a contradiction that there is a $k$-profile $P\in \mathcal{P}$ that does not point towards a petal of some $(k-1)$-pseudo\-flower $\Phi$, and let $\mathcal{S}$ be an inclusion-wise maximal chain of sets displayed by $\Phi$ that are contained in $P$.
Assume for a contradiction that $\bigcup\mathcal{S}$ is contained in $P$.
Then there is a petal $Q$ that is disjoint from $\bigcup \mathcal{S}$ towards which $P$ does not point, so by the profile property $\bigcup \mathcal{S} \cup Q$ is also contained in $P$, a contradiction to the maximality of $\mathcal{S}$.
So there is a chain of sets contained in $P$ whose union is not contained in $P$.
And, remarkably, a $k$-profile with this property is induced by a unique profile of all finite-order separation and cannot be the truncation of two distinct $k+1$-profiles, as will be shown now.
In particular, if $\mathcal{P}$ contains at least two $k$-profiles, then their truncation points, for every $(k-1)$-pseudo\-flower $\Phi$, towards a petal of $\Phi$.

\begin{lem}\label{zigzagsofsmallorder}
	Let $P$ be a $k$-profile and $(Q_{\alpha})_{\alpha<\kappa}$ an increasing chain of sets in $P$ such that its union $Q$ is not contained in $P$.
	Then for every set $S$ of finite order there is a cofinal set $F\subseteq \kappa$ such that $\lambda((S\cap Q)\cup Q_{\alpha})\leq k-1$ and $\lambda((E\setminus S)\cap Q)\cup Q_{\alpha})\leq k-1$ for all $\alpha\in F$.
\end{lem}
\begin{proof}
	As $\lambda(S\cap Q_{\alpha})$ is bounded by $\lambda(S)+k-1$, there is a cofinal set $G\subseteq \kappa$ such that $\lambda(S\cap Q_{\alpha})$ does not depend on $\alpha\in G$.
	Then for all $\alpha,\beta\in G$
	\begin{align*}
		\lambda((S\cap Q_{\alpha})\cup Q_{\beta})\leq \lambda(S\cap Q_{\alpha})+\lambda(Q_{\beta})-\lambda(S\cap Q_{\alpha} \cap Q_{\beta})=\lambda(Q_{\beta})\leq k-1.
	\end{align*}
	Thus for every $\alpha\in G$ the set $(S\cap Q)\cup Q_{\alpha}$ is the union of the sets $(S\cap Q_{\beta})\cup Q_{\alpha}$ where $\beta\in G$, so the order of $(S\cap Q)\cup Q_{\alpha}$ is at most $k-1$. Similarly there is a cofinal subset $F$ of $G$ such that $\lambda(((E\setminus S)\cap Q)\cup Q_{\alpha})\leq k-1$ for all $\alpha\in F$.
\end{proof}

\begin{lem}\label{notlimitclosedinducesaleph0}
	Let $P$ be a $k$-profile and let $(Q_{\alpha})_{\alpha<\kappa}$ be an increasing chain of separations in $P$ such that its supremum $Q$ is not contained in $P$.
	Then the set
	\begin{equation*}
		\{S\in \mathcal{U} \mathbin{|} \lambda(S)\in \mathbb{N} \cap \exists \alpha<\kappa:(S\cap Q)\cup Q_{\alpha}\in P\}
	\end{equation*}
	is a profile of the sets of finite order and induces every $l$-profile which induces $P$.
\end{lem}
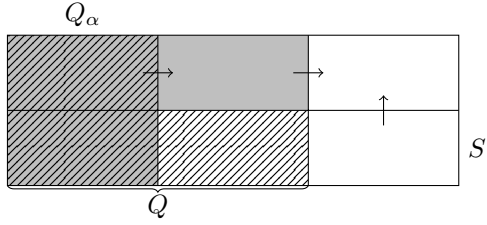
\begin{figure}
	\begin{tikzpicture}[xscale=2,decoration = brace]
		\path [fill=gray!50] (0,0) -- (1,0) -- (1,1) -- (2,1) -- (2,2) -- (0,2) -- (0,0);
		\path [fill,striped] (0,0) -- (2,0) -- (2,1) -- (1,1) -- (1,2) -- (0,2) -- (0,0);
		\draw (0,1) -- (3,1)
		(1,0) -- (1,2)
		(2,0) -- (2,2);
		\draw (1,1.5) pic {separrow} (2,1.5) pic {separrow} (2.5,1) pic[rotate=90] {separrow};
		\draw (3,0.5) node [anchor = west] {$S$}
			(.5, 2) node [anchor = south] {$Q_\alpha$};
		\draw[decorate] (2,0) -- (0,0) node [pos = .5, anchor = north] {$Q$};
		\draw[thin] (0,0) rectangle (3,2);
	\end{tikzpicture}
	\caption{Some notation of the proof of \cref{notlimitclosedinducesaleph0}.
		The rectangle is all of $E$, the set $R$ is depicted with stripes and the set $R'$ is depicted in light grey.}
	\label{fig:notlimitclosedinducesaleph0}
\end{figure}
\begin{proof}
	Let $\mathcal{Q}=\{S\in \mathcal{U} \mathbin{|} \lambda(S)\in \mathbb{N} \cap \exists \alpha<\kappa:(S\cap Q)\cup Q_{\alpha}\in P\}$.
	In order to show that $\mathcal{Q}$ contains every set of finite order or its complement, let $S\in \mathcal{U}$ with $\lambda(S)\in \mathbb{N}$.
	By \cref{zigzagsofsmallorder} there is $\alpha<\kappa$ such that both $R:=(S\cap Q)\cup Q_{\alpha}$ and $R':=((E\setminus S)\cap Q)\cup Q_{\alpha}$ (see \cref{fig:notlimitclosedinducesaleph0}) have order at most $k-1$.
	By submodularity one of $R\cup (E \setminus R')$ and $R\cap (E \setminus R')$ has order at most $k-1$, assume without loss of generality that it is $R\cup (E \setminus R')$.
	If $R\cup (E \setminus R') \in P$, then by consistency also $R \in P$ which implies $S \in Q$.
	So assume otherwise, thus its complement $R' \setminus R$ is contained in $P$.
	Because the union of $R' \setminus R$ and $Q_{\alpha}$ is again $R'$, by the profile property $R'\in P$ and thus $E\setminus S\in \mathcal{Q}$.
	
	The whole ground set is not contained in $\mathcal{Q}$.
	So, in order to show that $\mathcal{Q}$ is a profile, it suffices to show that for every two elements of $\mathcal{Q}$ their union is also contained in $\mathcal{Q}$:
	then $\mathcal{Q}$ cannot contain both $S$ and $E \setminus S$ for any separations $S$ of finite order, and furthermore $\mathcal{Q}$ is consistent and a profile.
	In order to show that the union of any two elements $R$ and $S$ is again contained in $\mathcal{Q}$, apply \cref{zigzagsofsmallorder} several times to obtain a cofinal set $F\subseteq \kappa$ for $S$, a cofinal subset $G\subseteq F$ for $R$ and a cofinal set $H\subseteq G$ for $S\cup R$.
	Let $\alpha < \kappa$ such that $(R\cap Q)\cup Q_{\alpha}\in P$.
	Then for all $\beta \in H$ with $\beta>\alpha$, the separation $(R\cap Q)\cup Q_{\beta}$ has order at most $k-1$ and is the union of $(R\cap Q)\cup Q_{\alpha}$ and $Q_{\beta}$.
	So by the profile property $(R\cap Q)\cup Q_{\beta}\in P$.
	Similarly for all sufficiently large $\beta$ in $H$, $(S\cap Q)\cup Q_{\beta}\in P$.
	As $((R\cup S)\cap Q)\cup Q_{\beta}$ is for all $\beta<\kappa$ the union of $(R\cap Q)\cup Q_{\beta}$ and $(S\cap Q)\cup Q_{\beta}$, it is for all sufficiently large $\beta\in H$ contained in $P$ by the profile property.
	Thus $R\cup S$ is contained in $Q$.
\end{proof}

Cal a $k$-profile \emph{limit-closed} if, for every chain of elements of the $k$-profile, the supremum of the chain is also contained in the $k$-profile.

\begin{cor}\label{truncationlimitclosed}
	The common truncation of any two distinct $k$-profiles to a $k-1$-profile  is limit-closed.\qed
\end{cor}

Now the fact that the common truncation to a $k-1$-profile of distinct $k$-profiles is limit-closed can be used to show that if in a $k$-pseudo\-flower two distinct elements of $\mathcal{P}$ are not distinguished and do not point towards a petal, then the $k$-pseudo\-flower has to be a $k$-pseudo\-anemone.

\begin{lem}\label{twoprofilesatnonpetal}
	Let $\Phi$ be a $k$-pseudo\-flower and let $P_1$ and $P_2$ be two $k$-profiles with truncation $P_0$ which are located at the same cut. Then $\Phi$ can be concatenated into an infinite $k$-anemone and some union of partition classes of the $\leq_A$-maximal partition\myfootnote{Recall that by \cref{finestrefinement} this is a strong $k$-pseudo\-anemone that is unique up to the choice of the cyclic order} extending $\Phi$ distinguishes $P_1$ and $P_2$.
\end{lem}
\begin{proof}
	Let $I$ be the index set of $\Phi$ and $v\in C(I)\setminus I$ be the cut at which the two profiles are located.
	Assume that $S(w,v)\in P_1$ for all $w\in C(I)\setminus (I+v)$, the other case is symmetric.
	Denote the common truncation of $P_1$ and $P_2$ by $P_0$.
	By \cref{truncationlimitclosed} $P_0$ is limit-closed.
	Let $S'$ be the union of all sets $S(u,v)$ with $u\in C(I)\setminus I$ and $\lambda(S(u,v))<k-1$.
	As $P_0$ is limit-closed, $S'\neq E$.
	If $S'$ is not the empty set, then it is of the form $S(v,x_0)$ for some $x_0\in C(I)\setminus I - v$.
	Otherwise let $x_0\in C(I)\setminus I - v$ be arbitrary.
	Then $\lambda(S(v,u))=k-1$ for all $u\in [x_0,v\rlbracket$.	
	
	As an intermediate step of finding the infinite $k$-anemone, let $x\in C(I)\setminus (I+v)$.
	Show that there is $y\in \lrbracket x,v\rlbracket\setminus I$ such that $\lambda(S(x,w))=k-1$ for all $w\in [y,v\rlbracket\setminus I$ as follows:
	Let $C_x=\{w\in \lrbracket x,v\rlbracket\setminus I\colon \lambda(S(x,w))<k-1\}$. If $C_x$ is empty, then $y$ can be chosen arbitrarily from $\lrbracket x,v\rlbracket\setminus I$.
	So assume that $C_x$ is non-empty and denote its supremum in $[x,v]$ by $v'$.
	Then every set of the form $S(x,w)$ with $w\in C_x$ is contained in $P_0$ and the union of all those sets is $S(x,v')$.
	As $P_0$ is limit-closed, $S(x,v')$ is contained in $P_0$ and thus in $P_1$.
	Because $S(x,v)\notin P_1$, this implies that $v\neq v'$.
	Thus $y$ can be chosen arbitrarily from $\lrbracket v',v\rlbracket\setminus I$.
	
	Recall that $x_0$ is already defined.
	Define recursively a (possibly transfinite) sequence $(x_{\alpha})_{\alpha<\nu}$ as follows:
	For a limit $\alpha$, if the supremum $v'$ of $(x_{\beta})_{\beta<\alpha}$ in $[x,v]$ is $v$, then terminate the construction.
	Otherwise let $x_{\alpha}=v'$.
	For a successor ordinal $\alpha+1$ there is by the previous paragraph some $x_{\alpha+1}\in \lrbracket x_{\alpha},v\rlbracket$ such that $\lambda(S(x_{\alpha},w))=k-1$ for all $w\in [x_{\alpha+1},v\rlbracket$.
	As a result, $x_{\beta}\in \lrbracket x_{\alpha},v\rlbracket$ and $\lambda(S(x_{\alpha},x_{\beta}))=k-1$ for all $\alpha<\beta<\nu$ and the supremum in $[x_0,v]$ of all $x_{\alpha}$ is $v$.
	
	Let $S$ be a set of order at most $k$ which distinguishes $P_1$ and $P_2$.
	Applying \cref{zigzagsofsmallorder} to the chain $(S(v,x_{\alpha}))_{0<\alpha<\nu}$ yields a cofinal $F\subseteq \nu$ such that $\lambda(S(v,x_{\alpha})\cup S)\leq k-1$ and $\lambda(S(v,x_{\alpha})\cup (E\setminus S))\leq k-1$ for all $\alpha\in F$.
	As the sets $S(v,x_{\alpha})\cup S$ and $S(v,x_{\alpha})\cup (E\setminus S)$ distinguish $P_1$ and $P_2$, they have order exactly $k-1$.
	Let $\mu<\nu$ be a limit-ordinal such that the supremum of $F\cap \mu$ is $\mu$.
	Then $S\cup S(v,x_\mu)$ is the union of the sets $(S \cup S(v,x_{\alpha}))_{\alpha \in \mu\cap F}$ and thus has order at most $k-1$.
	Similarly the order of $(E\setminus S)\cup S(v,x_{\mu})$ is at most $k-1$.
	
	Let $G$ be the union of $F-0$ and all limit ordinals $\mu<\nu$ for which the supremum of $F\cap \mu$ is $\mu$.
	Let $\alpha_1$ be the smallest element of $G$, $\alpha_2$ the smallest element but one of $G$ and so on.
	Denote $x_{\alpha_j}$ by $y_j$ for all $j\geq 1$, and let $z$ be the supremum of $(y_j)_{j\in \mathbb{N}-0}$.
	The partition of $E$ whose partition classes are $S(z,y_1)$ and the sets $S(y_i,y_{i+1})$ for $i\in \mathbb{N}-0$ inherits a natural linear order from $\mathbb{N}$, with $S(z,y_1)$ as biggest element, and thus a cyclic order.
	Denote this partition with cyclic order by $\Psi$.
	We have shown so far that $\Psi$ is an infinite $k$-flower, we now want to show that it is a $k$-anemone.
	
	Let $l$ be the maximum of $k$ and $3$.
	For every $\alpha\in G$ with $\alpha_l<\alpha$
	\begin{align*}
		\lambda(S(y_{l-1},y_l)&\cup (S(y_l,x_{\alpha})\cap S))\\
		&\leq \lambda(S(v,y_l)\cup (S(y_l,x_{\alpha})\cap S))+\lambda(S(y_{l-1},x_{\alpha}))-\lambda(S(v,x_{\alpha}))\\
		&=\lambda(S(v,y_l)\cup (S(y_l,x_{\alpha})\cap S)) + k-1 - (k-1) \\
		&\leq \lambda(S(v,y_l)\cup (S(y_{l-1},v)\cap S)) + \lambda(S(v,y_l)\cup S(y_{l-1},x_{\alpha}))\\
		&\quad- \lambda(S(v,y_l)\cup S(y_{l-1},x_{\alpha})\cup S)\\
		&= \lambda(S(v,y_l)\cup S) + \lambda(S(v,x_{\alpha})) - \lambda(S(v,x_{\alpha})\cup S) =k-1.
	\end{align*}
	
	Denote $S(y_l,v)\cap S$ by $Q$ and $S(y_l,v)\cap (E\setminus S)$ by $R$.
	Then $S(y_{l-1},y_l)\cup Q$, which is the union of all the sets $S(y_{l-1},y_l)\cup (S(y_l,x_{\alpha})\cap S)$ with $\alpha\in G$ and $\alpha_l<\alpha$, has order at most $k-1$.
	Symmetrically $S(y_{l-1},y_l)\cup R$ has order at most $k-1$.
	So
	\begin{align*}
		\lambda(S(y_{l-2},y_{l-1})\cup Q)&\leq \lambda(S(y_{l-2},y_l)\cup Q) + \lambda(S(v,y_{l-1})\cup Q) - \lambda(S(v,y_l) \cup Q)\\
		&=\lambda(S(y_{l-2},y_l)\cup Q) + \lambda(S(y_{l-1},y_l)\cup R) -\lambda(S(v,y_l)\cup S)\\
		&\leq \lambda(S(y_{l-2},y_l)) + \lambda(S(y_{l-1},y_l)\cup Q)- \lambda(S(y_{l-1},y_l))\\
		&\quad  + (k-1) -\lambda(S(v,y_l)\cup S)\\
		&=\lambda(S(y_{l-1},y_l)\cup Q)\leq k-1
	\end{align*}
	and symmetrically $\lambda(S(y_{l-2},y_{l-1})\cup R)\leq k-1$.
	Thus by submodularity also $\lambda(S(y_{l-2},y_{l-1})\cup S(y_l,v))\leq k-1$ and $\lambda(S(y_{l-2},y_{l-1})\cup S(y_l,y_{l+1}))\leq k$ and hence $\Psi$ is a $k$-anemone.
	
	Let $\Psi_f$ be the concatenation of $\Psi$ whose petals are $S(y_i,y_{i+1})$ with $1\leq i \leq l-1$ and $S(y_l,y_1)$.
	By \cref{muforpetal}
	\begin{displaymath}
		\lambda(S(y_{l-1},y_l)\cup Q)=\lambda(S(y_{l-1},y_l)\cup R)=k-1
	\end{displaymath}
	and thus by \cref{skpatosetswithmusmall} there is an extension of $\Psi_f$ which is a strong $k$-pseudo\-anemone, contains $Q$ as a petal and thus distinguishes $P_1$ and $P_2$.
	So the maximal extension of $\Phi$, which exists by \cref{finestrefinement} and is also the maximal extension of $\Psi$, contains a union of petals that equals $Q$ and thus distinguishes $P_1$ from $P_2$.
\end{proof}

So most $\leq$-maximal $k$-pseudo\-flowers that extend a daisy are also $\preccurlyeq$-maximal.

\begin{lem}\label{maxpreccurlyeqkpf}
	Let $\Phi$ be a $\leq$-maximal $k$-pseudo\-flower which has a concatenation into a $k$-daisy and distinguishes at least three profiles.
	Then $\Phi$ is $\preccurlyeq$-maximal.
\end{lem}
\begin{proof}
	Assume for a contradiction that there is a $k$-pseudo\-flower $\Phi'$ such that $\Phi\preccurlyeq \Phi'$ and such that $\Phi'$ distinguishes two profiles $P_1$ and $P_2$ from $\mathcal{P}$ which are not distinguished by $\Phi$.
	Then by \cref{twoprofilesatnonpetal} the two profiles cannot be located at the same cut, so they have to point to the same petal $i$.
	As $\Phi$ distinguishes sufficiently many profiles, there is a separation $S(v,w)$ of $\Phi'$ which not only distinguishes $P_1$ from $P_2$, but also distinguishes two profiles $P_3$ and $P_4$ which are distinguished in $\Phi$ from $P_1$ as well as from each other.
	Then $S(v,w)$ properly crosses $S(i)$, which is by \cref{divideonepetal} a contradiction to the fact that $\Phi$ is $\leq$-maximal.
\end{proof}

Unfortunately, it is not true that a $\leq$-maximal $k$-pseudo\-anemone (with sufficiently many petals and distinguishing sufficiently many $k$-profiles) is also necessarily $\preccurlyeq$-maximal.
This is illustrated by \cref{ex:anemoneultrafilter}.
In this example there is a $k$-anemone whose partition is finest among all partitions of $k$-pseudo\-flowers.
But there are many possible choices for the cyclic order on the set of petals, and which profiles are distinguished depends on the cyclic order.

The example makes use of the notion of ultrafilters, an important notion in topology, to be found for example in~\cite{Manetti}.
\begin{defn}\label{defn:ultrafilter}
	An \emph{ultrafilter} of a set $X$ is a non-empty set $\mathcal{F}$ of subsets of $X$ with the following properties:
	\begin{itemize}
		\item The empty set is not contained in $\mathcal{F}$.
		\item The intersection of any two elements of $\mathcal{F}$ is again contained in $\mathcal{F}$.
		\item Given subsets $Y$ and $Z$ of $X$ such that $Y\subseteq Z\subseteq X$ and $Y\in \mathcal{F}$, then also $Z\in \mathcal{F}$.
		\item If a subset $Y$ of $X$ has a non-empty intersection with all elements of $\mathcal{F}$, then it is contained in $\mathcal{F}$.
	\end{itemize}
	An ultrafilter is \emph{free} if it does not contain finite sets.
\end{defn}

One of the most important properties of free ultrafilters is that they exist for all infinite sets $X$.

\begin{lem}\label{existencefreeultrafilter}
	Let $X$ be a set and $Y$ an infinite subset of $X$.
	Then there is a free ultrafilter of $X$ which contains $Y$.
\end{lem}
\begin{proof}
	This well-known result from topology can for example be shown by applying \cite[Theorem 8.17]{Manetti} to the set of subsets $Z$ of $X$ for which $Y\setminus Z$ is finite.
\end{proof}

\begin{ex}\label{ex:anemoneultrafilter}
	Let $E$ be an infinite set and $k$ an integer bigger than $1$.
	Define an order function $\lambda$ on the set of subsets of $E$ via
	\begin{equation*}
		\lambda(X)=
		\begin{cases}
			0 & X=\emptyset \text{ or } X=E\\
			k-1 & \text{otherwise.}
		\end{cases}
	\end{equation*}
	Let $\mathcal{P}$ be the set of $k$-profiles of $E$ and $\lambda$.
	Then $\mathcal{P}$ is the set of ultrafilters of $E$ and every cyclic order turns $E$ into a $\leq$-maximal $k$-anemone in which every petal has exactly one element.
	
	\begin{clm}
		For $E$, $\lambda$ and $\mathcal{P}$ there is no $\preccurlyeq$-maximal $k$-anemone.
	\end{clm}
	\begin{proof}
		It suffices to show that if $C$ is a cyclic order of $E$, then the $k$-anemone $\Phi$ whose petals contain only one element and whose petals are cyclically ordered according to $C$ is not $\preccurlyeq$-maximal.
		
		Let $L$ be a linear order such that closing it to a cyclic order yields $C$.
		As $E$ is infinite, there is a sequence $e_1,e_2,\ldots$ of elements of $E$ such that in the linear order $L$ either $e_i<e_{i+1}$ for all indices $i$ or $e_i>e_{i+1}$ for all indices $i$.
		Assume that there is such a sequence such that $e_i<e_{i+1}$ for all indices $i$, the other case is symmetric.
		Denote the set of elements of $E$ which are of the form $e_i$ with an odd index $i$ by $R$.
		Define a new linear order $L'$ on $E$ where $e<f$ if one of the following happens:
		\begin{itemize}
			\item both $e$ and $f$ are contained in $R$ and $e<f$ in $L$;
			\item only $e$ is contained in $R$; or
			\item neither $e$ nor $f$ is contained in $R$ and $e<f$ in $L$.
		\end{itemize}
		Let $C'$ be the cyclic order obtained from closing $L'$ to a cyclic order, and denote the $\leq$-maximal $k$-anemone arising from $C'$ by $\Psi$.
		Then $\Psi$ distinguishes all elements of $\mathcal{P}$ which are distinguished by $\Phi$, so $\Phi\preccurlyeq\Psi$.
		But there are free ultrafilters $P_1$ and $P_2$ such that $P_1$ contains $R$ and $P_2$ contains the set of all $e_i$ with even index.
		Then $\Psi$ distinguishes $P_1$ from $P_2$, but $\Phi$ does not.
		Hence $\Phi$ is not $\preccurlyeq$-maximal.
	\end{proof}
	So in this setting there are many $\leq$-maximal $k$-anemones, each of them distinguishing infinitely many profiles, but no $\preccurlyeq$-maximal $k$-pseudo\-flowers.
\end{ex}

Because of the previous example, strong $k$-pseudo\-anemones are compared by the following pre-order, instead of by $\preccurlyeq$.

\begin{defn}
	Define a relation $\preccurlyeq_A$ on the set of strong $k$-pseudo\-anemones where $\Phi\preccurlyeq_A \Psi$ if all profiles in $\mathcal{P}$ which can be distinguished by a union of petals of $\Phi$ can be distinguished by a union of petals of $\Psi$.
\end{defn}

So essentially, $\preccurlyeq$ and $\preccurlyeq_A$ mean a $k$-pseudo\-flower is less than another if all profiles distinguished by a separation displayed by the first $k$-pseudo\-flower are also distinguished by a separation displayed by the second $k$-pseudo\-flower.
The two pre-orders just disagree on which separations count as displayed by a $k$-pseudo\-flower.
The separate definition of $\preccurlyeq_A$ is also in line with the separate definition of $\leq_A$ for strong $k$-pseudo\-anemones and the observation that for strong $k$-pseudo\-anemones the cyclic order of the partition is unimportant.
Now $\leq_A$-maximal strong $k$-pseudo\-anemones can be shown to be $\preccurlyeq_A$-maximal.

\begin{lem}\label{exmaxpreccurlyeqA}
	Let $\Phi$ be a $\leq_A$-maximal strong $k$-pseudo\-anemone which has a concatenation into a $k$-anemone and distinguishes at least three profiles.
	Then $\Phi$ is $\preccurlyeq_A$-maximal.
\end{lem}
\begin{proof}
	Assume for a contradiction that $\Phi$ is not $\preccurlyeq_A$-maximal.
	So there is a strong $k$-pseudo\-anemone $\Psi$ such that $\Phi\preccurlyeq_A\Psi$ but not $\Psi\preccurlyeq_A\Phi$.
	Let $P_1$ and $P_2$ be two profiles which are distinguished by a union of petals of $\Psi$ but not by a union of petals of $\Phi$.
	As no union of petals of $\Phi$ distinguishes $P_1$ from $P_2$, they are in particular located at the same cut or petal $\Phi$.
	If they are located at the same cut, then by \cref{twoprofilesatnonpetal} there is a strong $k$-pseudo\-anemone $\Phi'$ which can be concatenated into an infinite $k$-anemone such that $\Phi\leq_A\Phi'$ and such that some union of petals of $\Phi$ distinguishes $P_1$ and $P_2$.
	Thus $\Phi'$ is a strong $k$-pseudo\-anemone and some cyclic order on it gives an extension of $\Phi$.
	So there is an extension of $\Phi$ of which some union of petals distinguishes $P_1$ and $P_2$, contradicting the fact that $\Phi$ is $\leq_A$-maximal and that none of its unions of petals distinguishes $P_1$ from $P_2$.
	
	If $P_1$ and $P_2$ are located at the same petal $i$ of $\Phi$, let $P_3$ and $P_4$ be profiles which are distinguished from each other and from $P_1$ in $\Phi$.
	Let $S$ be a separation displayed by $\Psi$ which distinguishes $P_1$ from $P_2$ and $P_3$ from $P_4$.
	Then $S$ properly crosses $S(i)$, so by \cref{divideonepetal} there is an extension of $\Phi$ which has $S(i)\cap S$ and $S(i)\cap (E \setminus S)$ as petals and thus distinguishes $P_1$ from $P_2$.
	This is a contradiction to the fact that $\Phi$ is $\leq_A$-maximal.
\end{proof}

\section{The abstract structure of the equivalence classes of separations}\label{sec:abstractionfromunderlyingsepsys}

Let $\mathcal{U}$ be a submodular universe, $k\in \mathbb{N}$ and $S_k$ the set of separations of $\mathcal{U}$ of order less than $k$.
Let $\mathcal{P}$ be a non-empty set of regular $k$-profiles which all have the same truncation $Q$.
Declaring two separations in $S$ to be equivalent if and only if they are contained in the same elements of $\mathcal{P}$ induces a natural equivalence relation on $S_k$.
This section shows that the set of equivalence classes has the structure of a separation system of bipartitions that is closed under finite unions.
The map $\phi$ is taken from \cite{EH:treesets}, even though that paper will be published later than this one.
The properties of $\phi$ describe phenomena known and used frequently in proof of tangle-tree theorems.
In particular the fact that certain suprema of elements of $S$ are also contained in $S$ has led to the study of structurally submodular separation systems, a generalisation of subsystems of submodular universes.
Also, the nested set obtained in \cite[Theorem 3.6]{ProfilesNew} applied to $S$ and $\mathcal{P}$ is very close to the set of equivalence classes that do not cross other equivalence classes, as will be explained in \cref{ex:treeoftangles}.

The first lemma of this section shows that the involution, partial order and join of $S$ induce natural maps on the set of equivalence classes which turn the latter into a separation system naturally isomorphic to a separation system of bipartitions of $\mathcal{P}$.
Intuitively, that separation system of bipartitions condenses from $S$ the information of how it distinguishes the elements of $\mathcal{P}$.
As the image of $\phi$ consists of subsets of $\mathcal{P}$, notation and terminology for subsets instead of for separations will be used for its elements.

\begin{defn}\label{def:abstractionofseps}
	Let $\phi:S\rightarrow \mathcal{UB}(\mathcal{P})$ map every separation $\overrightarrow{p}$ to the set of elements of $\mathcal{P}$ which contain $\overleftarrow{p}$.
\end{defn}

\begin{lem}\label{abstractionishom}
	The map $\phi$ respects $*$, $\leq$ and $\vee$.
\end{lem}
\begin{proof}
	For all $\overrightarrow{s}\in S$ the complement of $\phi(\overrightarrow{p})$ in $E$ equals $\phi(\overleftarrow{p})$.
	If $\overrightarrow{p}$ and $\overrightarrow{q}$ are elements of $S$ such that $\overrightarrow{p}\leq \overrightarrow{q}$ then, as the elements of $\mathcal{P}$ are regular consistent orientations, all elements of $\mathcal{P}$ which contain $\overleftarrow{p}$ also contain $\overleftarrow{q}$, so $\phi(\overrightarrow{p}) \subseteq \phi(\overrightarrow{q})$.
	If $\overrightarrow{p}$ and $\overrightarrow{q}$ are elements of $S$ such that $\overrightarrow{p}\vee \overrightarrow{q}$ is also contained in $S$, then $\phi(\overrightarrow{p}) \cup \phi(\overrightarrow{q}) \subseteq \phi(\overrightarrow{p}\vee \overrightarrow{q})$ as $\phi$ respects the partial order.
	Let $P$ be a profile in $\mathcal{P}$ which is contained in $\phi(\overrightarrow{p}\vee \overrightarrow{q})$.
	Then $P$ contains $(\overrightarrow{p}\vee \overrightarrow{q})^*$ and thus by the profile property $P$ also contains $\overleftarrow{p}$ or $\overleftarrow{q}$.
	Thus $P$ is also contained in $\phi(\overrightarrow{p}) \cup \phi(\overrightarrow{q})$.
	So $\phi(\overrightarrow{p}\vee \overrightarrow{q})$ is contained in $\phi(\overrightarrow{p}) \cup \phi(\overrightarrow{q})$ and thus the two sets are equal.
\end{proof}

Part of \cref{abstractionishom} is a statement about elements of $S$ whose join is contained in $S$ as well.
\Cref{cornersinSexist} shows that in many cases that join does exist, a fact that will later on also reveal more structure of the image of $\phi$.
Both the following and the previous lemma of course imply similar statements for the meet operation.

\begin{lem}\label{cornersinSexist}
	Let $\overrightarrow{p}$ and $\overrightarrow{q}$ be elements of $S$ such that $\phi(\overrightarrow{p}) \cap \phi(\overrightarrow{q})\notin \{\emptyset,\mathcal{P}\}$.
	Then $\overrightarrow{p}\vee \overrightarrow{q}\in S$.
\end{lem}
\begin{proof}
	As $\phi(\overrightarrow{p}) \cap \phi(\overrightarrow{q})\neq \emptyset$, there is a profile in $\mathcal{P}$ which contains both $\overleftarrow{p}$ and $\overleftarrow{q}$.
	If $\overrightarrow{p}\wedge \overrightarrow{q}$ is contained in $S$, then the profile containing both $\overleftarrow{p}$ and $\overleftarrow{q}$ also contains $\overleftarrow{p}\vee \overleftarrow{q}$.
	In this case $\overrightarrow{p}\wedge \overrightarrow{q}$ distinguishes two profiles in $\mathcal{P}$ and thus has order $k-1$.
	So in each case $\overrightarrow{p}\wedge \overrightarrow{q}$ has order at least $k-1$ and by submodularity $\overrightarrow{p}\vee \overrightarrow{q}$ has order at most $k-1$ and is thus contained in $S$.
\end{proof}

If two elements of $S$ satisfy $\phi(\overrightarrow{p}) \subseteq \phi(\overrightarrow{q})$ then $\overrightarrow{p}\leq \overrightarrow{q}$ does not necessarily hold.
But \cref{nestedpreimagesofphi} shows for nested elements of the image of $\phi$ that there do exist pre-images that are nested.

\begin{lem}\label{nestedpreimagesofphi}
	Let $R$ and $T$ be elements of the image of $\phi$ such that $R \subseteq T$.
	Then there are elements $\overrightarrow{p}$ and $\overrightarrow{q}$ of $S$ such that $\phi(\overrightarrow{p})=R$, $\phi(\overrightarrow{q})=T$ and $\overrightarrow{p}\leq \overrightarrow{q}$.
\end{lem}
\begin{proof}
	If $R = T$, then it suffices to pick $\overrightarrow{p}=\overrightarrow{q}$, so assume otherwise.
	Let $\overrightarrow{p}$ and $\overrightarrow{q}$ be elements of $S$ such that $\phi(\overrightarrow{p}) = R$ and $\phi(\overrightarrow{q}) = T$.
	Also if $R=\emptyset$ and $T=\mathcal{P}$, then at least one of $\overrightarrow{p} \wedge \overrightarrow{q}$ and $\overrightarrow{p}\vee \overrightarrow{q}$ is contained in $S$.
	In the first case, replace $\overrightarrow{p}$ by $\overrightarrow{p} \wedge \overrightarrow{q}$ and in the second case replace $\overrightarrow{q}$ with $\overrightarrow{p} \vee \overrightarrow{q}$ and the lemma holds.
	So assume also that $R \neq \emptyset$ or $T \neq \mathcal{P}$.
	
	So one of $R \cap T$ and $R \cup T$ is not contained in $\{\emptyset, \mathcal{P}\}$.
	Thus by \cref{cornersinSexist}, one of $\overrightarrow{p}\wedge \overrightarrow{q}$ and $\overrightarrow{p} \vee \overrightarrow{q}$ is contained in $S$, and replacing one of $\overrightarrow{p}$ and $\overrightarrow{q}$ with the existing corner as above shows that the lemma holds.
\end{proof}

Denote the image of $\phi$ after deleting $\emptyset$ and $\mathcal{P}$ by $\mathcal{B}$.
Note that the empty set is contained in the image of $\phi$ if and only if $S$ contains a separation which is contained in all profiles in $\mathcal{P}$, and similarly for $\mathcal{P}$.
Also $\mathcal{B}$ is a regular separation system which is a sub-system of the universe of bipartitions of $\mathcal{P}$.
Furthermore, by \cref{cornersinSexist} in $\mathcal{B}$ many unions of elements of $\mathcal{B}$ are again contained in $\mathcal{B}$:

\begin{cor}[of \cref{cornersinSexist}]\label{cornersinBexist}
	Let $R$ and $T$ be elements of $\mathcal{B}$ with $R \cap T \neq \emptyset$ and $R \cup T \neq \mathcal{P}$.
	Then both $R \cap T$ and $R \cup T$ are elements of $\mathcal{B}$.
\end{cor}
\begin{proof}
	Let $R = \phi(\overrightarrow{p})$ and $T = \phi(\overrightarrow{q})$.
	By \cref{cornersinSexist} both $\overrightarrow{p}\wedge \overrightarrow{q}$ and $\overrightarrow{p}\vee \overrightarrow{q}$ are contained in $S$, and so by \cref{abstractionishom} both $R \cap T$ and $R \cup T$ are contained in $\mathcal{B}$.
\end{proof}

Hence $\mathcal{B}$ is closed under unions of crossing elements.
Such separation systems will be further investigated in \cref{sec:abstract}.
Also, if $S$ is finite, then most equivalence classes have a biggest and a smallest element.

\begin{cor}[of \cref{cornersinSexist}]\label{biggestpreimage}
	If $S$ is finite, then for every $R\in \mathcal{B}$ the set $\phi^{-1}(R)$ has a biggest and smallest element.
\end{cor}

The $k$-profiles of $\mathcal{U}$ whose truncation is $Q$ are closely related to the profiles of the separation system $\mathcal{B}$, as will be explained now.
First, let $P$ be an element of $\mathcal{P}$.
Note that the set of all $k$-profiles of $\mathcal{U}$ whose truncation is $Q$ is a candidate for $\mathcal{P}$, so if $\mathcal{P}$ is chosen accordingly, $P$ can be any regular $k$-profile of $\mathcal{U}$ whose truncation is $Q$.
Let $P'$ be the set of all elements of $\mathcal{B}$ that do not contain $P$.
It is easily checked that $P'$ is a profile of $\mathcal{B}$, and $\overrightarrow{p}\in S$ is contained in $P'$ if and only if $\phi(\overrightarrow{p})$ is contained in $P$.
In the other direction, every profile of $\mathcal{B}$ also induces a regular $k$-profile of $\mathcal{U}$ whose truncation is $Q$.

\begin{lem}
	Let $P$ be a profile of the separation system $\mathcal{B}$ and let $P'$ be the set of all those elements of $S$ whose image under $\phi$ is contained in $P \cup \{\emptyset\}$.
	Then $P'$ is a regular $k$-profile of $\mathcal{U}$ whose truncation is $Q$.
\end{lem}
\begin{proof}
	By the definition of $\phi$, $P'$ is an orientation of $S$.
	If some element $\overrightarrow{p}$ of $\mathcal{B}\cup \{\emptyset, \mathcal{P}\}$ is a subset of an element of $P \cup \{\emptyset\}$, then also $\overrightarrow{p} \in P \cup \{\emptyset\}$.
	As $\phi$ respects the partial order $\leq$, this property translates to the fact that if $\overrightarrow{p}$ and $\overrightarrow{q}$ are elements of $S$ with $\overrightarrow{p} \leq \overrightarrow{q} \in P'$ then also $\overrightarrow{p} \in P'$.
	So $P'$ is a consistent orientation.
	Also, the fact that $\phi$ respects $\vee$ implies that $P'$ has the profile property because $P$ has it.
	
	In order to show that $Q$ is the truncation of $P'$, let $\overrightarrow{r}$ be an element of $Q$.
	Then $\overrightarrow{r}$ is contained in $S$ and in all profiles in $\mathcal{P}$.
	So $\phi(\overrightarrow{r}) = \emptyset$ and hence $\overrightarrow{r} \in P'$.
	Thus $Q$ is a subset of $P'$ and thus the truncation of $P'$.
	
	Let $\overrightarrow{s}$ be a small element of $S$.
	As the elements of $\mathcal{P}$ are all regular, they contain $\overrightarrow{s}$ and thus $\phi(\overrightarrow{s})= \emptyset$.
	So $\overrightarrow{s} \in P'$, hence $P'$ contains all small elements of $S$ and is thus regular.
\end{proof}

The more elements $\mathcal{P}$ has, the better does the image of $\phi$ represent the structure of $S$.
As will be seen in \cref{sec:abstractprofiles}, it is quite easy to construct profiles of the image of $\phi$, which then by the previous lemma correspond to profiles of $S$.

See \cref{ex:treeoftangles} for a remark on how the abstract separation system $\mathcal{B}$ fits into the context of already existing strategies for the construction of trees of tangles.

\section{Abstract separation systems}\label{sec:abstract}
For this section fix a ground set $E$ and a set $\mathcal{B}$ of non-trivial subsets of $E$ such that $\mathcal{B}$ is a separation system with the following property:
\begin{align*}
	\text{If $S$ and $T$ are crossing elements of $\mathcal{B}$ , then $S\cup T$ is also contained in $\mathcal{B}$.}
\end{align*}
Phrased differently, $\mathcal{B}$ should be a separation system of bipartitions that is closed under unions of crossing elements.
We denote the set of all elements of $\mathcal{B}$ that do not cross any element of $\mathcal{B}$ by $\mathcal{E}$.
Note that $\mathcal{E}$ is closed under taking inverses and any two elements of $\mathcal{E}$ are nested.
Also, let $\mathcal{V}$ be the finest partition of $\mathcal{B} \setminus \mathcal{E}$ such that elements of distinct partition classes are nested.

\subsection{The finite tree structure}\label{sec:abstractfinite}

If $\mathcal{B}$ is finite, then it has a natural tree structure, as is described by \cite{CunninghamEdmonds80}.
Their terminology may be very different, but it can be applied to finite systems of bipartitions and then yields a tree decomposition.
In particular Theorems 3 and 4 of \cite{CunninghamEdmonds80} can be translated to this paper's terminology as follows:

\begin{thm}[reformulation of results of {\cite{CunninghamEdmonds80}}]\label{thm:CEtdc}
Let $E$ be a set and $\mathcal{B}'$ a finite set of non-oriented bipartitions such that every element of $\mathcal{B}'$ only contains partition classes of size at least two.
Assume that whenever $|A \cap B| \geq 2$ and $|E \setminus (A \cup B)|\geq 1$ for some $\{A, E \setminus A\}$ and $\{B, E \setminus B\}$ in $\mathcal{B}'$, then also $\{A \cap B, E \setminus (A\cap B)\}$ is contained in $\mathcal{B}'$.
Then there is a tree decomposition of $E$ with the following properties:
\begin{itemize}
	\item The unoriented separations induced by the edges of the tree decomposition are exactly the elements of $\mathcal{B}'$ that do not cross any element of $\mathcal{B}'$.
	\item Let $v$ be a node of the decomposition tree with part $B_v$, then deleting $v$ yields a partition of $E\setminus B_v$ that can be extended to a partition $P_v$ of $E$ by adding the elements of $B_v$ as singletons.
	If $S_v$ is the set of all biparititions $\{A,B\}$ of $E$ where both $A$ and $B$ are unions of at least two partition classes of $P_v$ then:
	\begin{enumerate}
		\item\label{CEtdc:simple} Either $S_v$ and $\mathcal{B}'$ are disjoint; or
		\item\label{CEtdc:brittle} $S_v$ is a subset of $\mathcal{B}'$; or
		\item \label{CEtdc:semibrittle} There is a cyclic order on $P_v$ such that an element $\{A,B\}$ of $S_v$ is contained in $\mathcal{B}'$ if and only if $A$ is the union of an interval of elements of $P_v$.
	\end{enumerate}
	\item Every element of $\mathcal{B}'$ is either induced by an edge of the decomposition tree or is contained in some $S_v$ for a node $v$ of the decomposition tree.
\end{itemize}
\end{thm}

\begin{rem}
	This version of \cite[Theorem 4]{CunninghamEdmonds80} relates to its original version as follows:
	As the elements of $\mathcal{B}'$ only contain partition classes with at least two elements each, $E$ and $\mathcal{B}'$ together form a split system as defined in \cite{CunninghamEdmonds80} after the proof of Theorem 3.
	There it is also explained how the split system induces a decomposition frame.
	The additional property of $\mathcal{B}'$ in \cref{thm:CEtdc} translates to $\mathcal{B}'$ having the intersection property of \cite{CunninghamEdmonds80}.
	So \cite[Theorem 4]{CunninghamEdmonds80} can be applied to the decomposition frame induced by $\mathcal{B}'$.
	On the first page of \cite{CunninghamEdmonds80} it is described how the resulting minimal decomposition can be turned into a (decomposition) tree, and the parts of the decomposition correspond to the vertices of the tree.
	This decomposition tree together with its parts can be turned into a tree decomposition by deleting, from all parts, all edges that are not contained in $E$.
	Parts of the decomposition that are called simple in \cite{CunninghamEdmonds80} are those that correspond to vertices $v$ of the decomposition tree for which (\ref{CEtdc:simple}) from \cref{thm:CEtdc} holds, namely that $S_v$ and $\mathcal{B}'$ are disjoint.
	Similarly, the brittle parts of the decomposition correspond to vertices for which (\ref{CEtdc:brittle}) holds, and the semi-brittle parts correspond to vertices for which (\ref{CEtdc:semibrittle}) holds.
\end{rem}

In order to apply this theorem to $\mathcal{B}$, it is necessary to first turn $\mathcal{B}$ into a suitable set $\mathcal{B}'$ of unoriented separations.
This is achieved by taking the unoriented versions of the elements of $\mathcal{B}$ and then deleting all elements that contain a partition class with only one element.
The tree decomposition that results from applying \cref{thm:CEtdc} can be turned into a tree decomposition that describes the tree structure of $\mathcal{B}$ by adding, for all $e\in E$ for which the orientations of $\{e,E\setminus e\}$ are contained in $\mathcal{B}$, a new vertex $v$ to $T$ that is adjacent to the unique vertex $w$ of $T$ whose part already contains $e$, deleting $e$ from the part of $w$ and making $\{e\}$ the part of $v$.
So the following version of \cref{thm:CEtdc} holds for $\mathcal{B}$:

\begin{thm}\label{thm:structurefiniteBsimple}
	Let $E$ be a set and $\mathcal{B}$ a finite subsystem of the separation system of bipartitions on ground set $E$ such that neither orientation of $\{\emptyset, E\}$ is contained in $\mathcal{B}$ and such that the union of any two non-nested elements of $\mathcal{B}$ is also contained in $\mathcal{B}$.
	Then there is a tree decomposition of $E$ with the following properties:
\begin{itemize}
	\item Every separation induced by the tree decomposition is contained in $\mathcal{B}$.
	\item Let $v$ be a node of the decomposition tree, then:
	\begin{enumerate}
		\item\label{structurefiniteBsimple:nonempty} Either the part of $v$ is non-empty; or
		\item\label{structurefiniteBsimple:simple} The part of $v$ is empty, and the partition $P_v$ that arises from the deletion of $v$ from $T$ has the property that every element of $\mathcal{B}$ contains at most one partition class and is disjoint from at most one partition class; or
		\item\label{structurefiniteBsimple:anemone} The part of $v$ is empty and all non-trivial unions of partition classes of $P_v$ are contained in $\mathcal{B}$; or
		\item\label{structurefiniteBsimple:daisy} The part of $v$ is empty, and there is a cyclic order on $P_v$ (unique up to mirroring) such that a non-trivial union of partition classes is contained in $\mathcal{B}$ if and only if it is a union of a non-trivial interval.
	\end{enumerate}
	\item Every element of $\mathcal{B}$ is either induced by an edge of the decomposition tree or is a union of partition classes of $P_v$ for a vertex $v$ of the decomposition tree.
\end{itemize}
\end{thm}

\begin{rem}
	\Cref{thm:structurefiniteBsimple} is phrased deliberately to closely resemble maximal partial $(k, \mathcal{S})$-trees from \cite{ClarkWhittle13}.
	Essentially, \cref{thm:structurefiniteBsimple} solves a problem for $\mathcal{B}$ that is solved in \cite{ClarkWhittle13} for equivalence classes of separations of the same order: trying to display oriented bipartitions of a ground set in a tree decomposition even though the displayed (equivalence classes of) bipartitions need not be nested.
	The partitions in \cref{structurefiniteBsimple:anemone} and \cref{structurefiniteBsimple:daisy} play the same role as anemones and daisies: displaying many bipartitions at once that mostly cross and thus cannot be displayed together by the tree decomposition alone.
	Indeed, in the special case where there is a connectivity function and $k\in \mathbb{N}$ such that the elements of $\mathcal{B}$ are exactly the subsets of $E$ of connectivity $k-1$, the partitions in \cref{structurefiniteBsimple:anemone} and \cref{structurefiniteBsimple:daisy} are $k$-anemones and $k$-daisies.
\end{rem}

The main goal of this section is to derive a similar decomposition theorem for the case that $\mathcal{B}$ is infinite.
For that, denote by $\mathcal{E}$ the set of separations in $\mathcal{B}$ that do not cross any other separations in $\mathcal{B}$, and let $\mathcal{V}$ be the finest partition of $\mathcal{B}\setminus \mathcal{E}$ in which crossing separations are contained in the same partition class.
By definition, $\mathcal{E}$ does not contain any crossing elements nor the empty set, and it is closed under taking inverses.

With this terminology, the tree decomposition obtained in \cref{thm:structurefiniteBsimple} has the following properties, which we try to also obtain for a tree decomposition if $\mathcal{B}$ is infinite:
\begin{itemize}
	\item The separations induced by the edges of the decomposition tree are the elements of $\mathcal{E}$.
	\item Every $V\in \mathcal{V}$ is associated with a vertex $v_V$ of the decomposition tree whose part is empty, and if $V\neq W \in \mathcal{V}$ then $v_V$ and $v_W$ are distinct.
	\item For $V\in \mathcal{V}$, $\partial(V)$ consists of those (oriented) separations in $\mathcal{B}$ that are induced by the edges of the decomposition tree that are incident with $v_V$ and oriented towards $v_V$.
	In particular $\partial(V)\subseteq \mathcal{E}$.
\end{itemize}
The proof of \cref{thm:CEtdc} in \cite{CunninghamEdmonds80} finds a decomposition tree whose edges induce the separations in $\mathcal{E}'$, and then find that the sets $S_v$ as in \cref{thm:CEtdc} are as described in that reformulated version of the theorem.
Furthermore, if $\mathcal{E}$ is finite then there are tools (described later) to obtain a tree decomposition whose edges induce the elements of $\mathcal{E}$.
These tools also work if $\mathcal{E}$ is infinite, though in that case the result might be a tree-like space, an infinite generalisation of a tree.
So it is a valid proof strategy to construct a tree-decomposition from $\mathcal{E}$ and then hope that the elements of $\mathcal{V}$ are associated with vertices of the decomposition tree(-like space) just as nicely as in the finite case.
There are two main reasons why this proof strategy does not work in the infinite case: First, the proof in \cite{CunninghamEdmonds80} that $S_v\cap \mathcal{B}'$ has a certain structure involves induction, but in the infinite case $S_v\cap \mathcal{B}'$ can also be infinite.
Second, in the decomposition tree(-like space) which is obtained from $\mathcal{E}$, several elements of $\mathcal{V}$ may be associated with the same vertex of the decomposition tree (-like space); essentially there may be elements ``missing'' from $\mathcal{E}$.
A solution to both problems is to organise the elements of $\mathcal{V}$ into flower-like structures first by applying \cref{thm:CEtdc} to suitable finite subsets of $V$ and combining the results, and then use the flower-like structures of the elements of $V$ to find a more appropriate tree-like structure.

\subsection{The flower-like structure of infinite sets of mostly crossing separations}\label{sec:abstractflower}

What is a suitable finite subset is captured in the notion of pre-flowers:

\begin{defn}\label{def:preflower}
	A \emph{pre-flower} is a finite subset $\mathcal{F}$ of $\mathcal{B}$ with at least two elements such that for all elements $S$ and $T$ of $\mathcal{F}$ there are $R_0,\ldots,R_n \in \mathcal{F}$ with $S = R_0$, $T = R_n$ such that, for $0\leq i\leq n-1$, $R_i$ and $R_{i+1}$ cross.
\end{defn}

Note that every pre-flower is a subset of some element of $\mathcal{V}$, and that if $\mathcal{F}$ and $\mathcal{G}$ are two subsets of $V \in \mathcal{V}$ that are pre-flowers then there is a pre-flower $\mathcal{H} \subseteq V$ that contains $\mathcal{F}$ and $\mathcal{G}$ as subsets.
The following notation will be used for both elements of $\mathcal{V}$ and pre-flowers:

\begin{defn}
	Given a set $S$ of bipartitions of $E$ let $\sim_{S}$ be the equivalence relation on $E$ where $e\sim_{S} f$ if and only if no element of $S$ distinguishes $e$ and $f$.
	Denote the set of equivalence classes of $\sim_{S}$ by $\partial(S)$, the set of elements of $\mathcal{B}$ which are unions of elements $\partial(S)$ by $\sepclos{S}$, and the set of elements of $\sepclos{S}$ which are not orientations of elements in $\partial(S)$ by $\sepinn{S}$.
\end{defn}

Every pre-flower is a subset of some element of $\mathcal{V}$, and if some $V\in \mathcal{V}$ is finite, then it is itself a pre-flower.
Also every pre-flower $\mathcal{F}$ contains two crossing separations, so $\partial(\mathcal{F})$ has at least four elements.
Furthermore, the bipartitions to which the elements of $\partial{\mathcal{F}}$ correspond form a star; and the separations of $\mathcal{B}$ towards which all elements of the star point are exactly the separations contained in $\sepclosf$.
Also $\mathcal{F}\subseteq \sepinnf\subseteq \sepclosf$, and every element of $\sepinnf$ contains at least two elements of $\partial(\mathcal{F})$ and is disjoint from at least two elements of $\partial(\mathcal{F})$.
As every corner of elements of $\sepclosf$ corresponds again to a union of elements of $\partial(\mathcal{F})$, also $\sepclosf$ is closed under unions of crossing elements.

\begin{lem}\label{preflowerisdaisyoranemone}
	Let $\mathcal{F}$ be a pre-flower.
	Then one of the following happens:
	\begin{itemize}
		\item The elements of $\sepclosf$ are exactly the unions of elements of $\partial(\mathcal{F})$ that are neither $\emptyset$ nor $E$.
		\item There is a cyclic order on $\partial(\mathcal{F})$, unique up to mirroring, such that the elements of $\sepclosf$ are exactly the unions of non-trivial intervals of $\partial(\mathcal{F})$.
	\end{itemize}
\end{lem}
\begin{proof}
	Assume for a contradiction that there is $S\in \sepinnf$ that is nested with all elements of $\sepinnf$.
	If $Q$ and $R$ are two elements of $\mathcal{F}$ that cross, then some orientation of $S$ points towards both $Q$ and $R$.
	So by the definition of $\mathcal{F}$, some orientation $S'$ of $S$ points towards all elements of $\mathcal{F}$.
	But then all elements of $S'$ are contained in the same element of $\partial(\mathcal{F})$, a contradiction.
	
	So applying \cref{thm:CEtdc} to $\sepinnf$ yields a decomposition with only one part.
\end{proof}

So for a pre-flower $\mathcal{F}$, \cref{thm:CEtdc} can be applied to $\sepinnf$ and yields a tree-decomposition with only one part that is anemone-like, in which case $\mathcal{F}$ is a \emph{pre-anemone}, or daisy-like, in which case $\mathcal{F}$ is a \emph{pre-daisy}.
Furthermore, given an element $V$ of $\mathcal{V}$, it is not possible that $V$ contains both pre-daisies and pre-anemones.

\begin{lem}\label{infdaisyoranemone}
	Let $V\in \mathcal{V}$ and let $\mathcal{F}$ and $\mathcal{G}$ be pre-flowers contained in $V$ such that $\mathcal{F}\subseteq \mathcal{G}$.
	Then $\mathcal{F}$ is a pre-anemone if and only if $\mathcal{G}$ is a pre-anemone.
\end{lem}
\begin{proof}
	As $\mathcal{F}\subseteq \mathcal{G}$, also $\sepinnf \subseteq \sepinn{\mathcal{G}}$.
	
	First consider the case that $\mathcal{F}$ is a pre-anemone, and let $Q$, $R$, $R'$ and $S$ be elements of $\partial(\mathcal{F})$ such that the union of $Q$ with any of the other three elements of $\partial(\mathcal{F})$ is an element of $\sepinnf$.
	As $\partial(\mathcal{G})$ is a refinement of $\partial(\mathcal{F})$, the sets $Q$, $R$, $R'$ and $S$ are unions of elements of $\partial(\mathcal{G})$.
	Then there is no cyclic order on $\partial(\mathcal{G})$ which turns all of $Q\cup R$, $Q\cup R'$ and $Q\cup S$ into intervals of $\partial(\mathcal{G})$.
	So $\mathcal{G}$ is not a pre-daisy, thus it is a pre-anemone.
	
	Now consider the case that $\mathcal{G}$ is a pre-anemone, and let $Q$, $R$, $R'$ and $S$ be distinct elements of $\partial(\mathcal{F})$.
	Then $Q\cup R$, $Q\cup R'$ and $Q\cup S$ are all elements of $\sepinn{\mathcal{G}}$ and unions of elements of $\partial(\mathcal{F})$, so they are elements of $\sepclosf$.
	As there is no cyclic order on $\partial(\mathcal{F})$ which turns these three sets into unions of intervals of $\partial(\mathcal{F})$, $\mathcal{F}$ is not a pre-daisy and thus is a pre-anemone.
\end{proof}

\begin{cor}
	Let $V\in \mathcal{V}$.
	Then either all pre-flowers contained in $V$ are pre-anemones or all pre-flowers contained in $V$ are pre-daisies.
\end{cor}
\begin{proof}
	Let $\mathcal{F}$ and $\mathcal{G}$ be pre-flowers contained in $V$ and let $Q\in \mathcal{F}$ and $R\in \mathcal{G}$.
	Then there is a pre-flower $\mathcal{H} \subseteq V$ with $\mathcal{F} \cup \mathcal{G} \subseteq \mathcal{H}$, and $\mathcal{F}$ is a pre-anemone if and only if $\mathcal{H}$ is a pre-anemone, if and only if $\mathcal{G}$ is a pre-anemone.
\end{proof}

So now the structure of the pre-flowers contained in some $V\in \mathcal{V}$ can be combined to a structure on $V$.
If the pre-flowers are pre-anemones, then not much can be gained in addition:

\begin{lem}
	If the pre-flowers contained in $\Phi(V)$ are pre-anemones, then for any two elements $R$ and $S$ of $V$, the intersection is contained in $\sepclos{V}\cup \{\emptyset\}$ and the union is contained in $\sepclos{V}\cup \{E\}$. 
\end{lem}
\begin{proof}
	By symmetry it suffices to show that the union of $R$ and $S$ is contained in $\sepclos{V}\cup\{E\}$.
	By the definition of $\mathcal{V}$, there is a pre-flower $\mathcal{F}\subseteq V$ that contains both $R$ and $S$.
	Hence $R$ and $S$ are unions of elements of $\partial(\mathcal{F})$, and so is their union.
	As $\mathcal{F}$ is a pre-anemone, either $R\cup S = E$ or $R \cup S \in \sepclosf$, and as $\sepclosf\subseteq \sepclos{V}$ the lemma follows.	
\end{proof}

If the pre-flowers contained in $V$ are pre-daisies, the cyclic orders are compatible with each other in the following sense:
For a pre-daisy $\mathcal{F}\subseteq V$ let a \emph{cyclic order of $\mathcal{F}$} be a cyclic order of $\partial(\mathcal{F})$ such that the unions of non-trivial intervals are exactly the elements of $\sepinnf$.
Also, for pre-daisies $\mathcal{F} \subseteq \mathcal{G} \subseteq V$ let $\pi_{\mathcal{G}\mathcal{F}}$ be the inclusion $\partial(\mathcal{F}) \rightarrow \partial(\mathcal{F})$.
Then every cyclic order of a pre-daisy induces unique cyclic orders of all comparable pre-daisies such that the inclusions are monotone maps of cyclically ordered sets.

\begin{lem}\label{daisiesconsistentorders}
	Let $\mathcal{F} \subseteq \mathcal{G} \subseteq \mathcal{H} \subseteq V$ be pre-daisies.
	For every cyclic order of $\mathcal{F}$ there is a unique cyclic order of $\mathcal{G}$ such that $\pi_{\mathcal{G}\mathcal{F}}$ is monotone, and for every cyclic order of $\mathcal{G}$ there is a unique cyclic order of $\mathcal{F}$ such that $\pi_{\mathcal{G}\mathcal{F}}$ is monotone.
	Also, if $T_{\mathcal{F}}$, $T_{\mathcal{G}}$ and $T_{\mathcal{H}}$ are cyclic orders of their respective pre-daisies, and two of the maps $\pi_{\mathcal{H}\mathcal{F}}$, $\pi_{\mathcal{H}\mathcal{G}}$ and $\pi_{\mathcal{G}\mathcal{F}}$ are monotone, then all three are monotone.
\end{lem}
\begin{proof}
	Every cyclic order of $\mathcal{G}$ induces a unique cyclic order of $\partial(\mathcal{F})$ such that $\pi_{\mathcal{G}\mathcal{F}}$ is monotone.
	In this cyclic order of $\partial(\mathcal{F})$, every union of a non-trivial interval is also a union of a non-trivial interval of $\partial(\mathcal{G})$ and thus is contained in $\sepclos{\mathcal{G}}\supseteq \sepclosf$.
	So the cyclic order of $\partial(\mathcal{F})$ is also a cyclic order of $\mathcal{F}$.
	As there are two cyclic orders of $\mathcal{G}$, which induce different cyclic orders of $\mathcal{F}$, the two cyclic orders of $\mathcal{F}$ are necessarily induced by the two cyclic orders of $\mathcal{G}$.
	
	For the second statement, if both $\pi_{\mathcal{H}\mathcal{G}}$ and $\pi_{\mathcal{G}\mathcal{F}}$ are monotone then so is their concatenation $\pi_{\mathcal{H}\mathcal{F}}$.
	If both $\pi_{\mathcal{H}\mathcal{G}}$ and $\pi_{\mathcal{H}\mathcal{F}}$ are monotone, then there is a unique cyclic order $T'_{\mathcal{F}}$ of $\mathcal{F}$ such that $\pi_{\mathcal{G}\mathcal{F}}$ is monotone with respect to $T_{\mathcal{G}}$ and $T'_{\mathcal{F}}$.
	Then $\pi_{\mathcal{H}\mathcal{F}}$ is the concatenation of monotone maps and thus monotone with respect to $T_{\mathcal{H}}$ and $T'_{\mathcal{F}}$.
	Thus $T_{\mathcal{F}} = T'_{\mathcal{F}}$.
	Similarly, if $\pi_{\mathcal{H}\mathcal{F}}$ and $\pi_{\mathcal{G}\mathcal{F}}$ are monotone, then there is a unique cyclic order $T'_{\mathcal{H}}$ such that $\pi_{\mathcal{H}\mathcal{G}}$ is monotone with respect to $T'_{\mathcal{H}}$ and $T_{\mathcal{G}}$, and $T'_{\mathcal{H}} = T_{\mathcal{H}}$.
\end{proof}

In particular, by \cref{daisiesconsistentorders}, if $\mathcal{F} \subseteq \mathcal{G} \subseteq \mathcal{H}$ are pre-daisies contained in $V$, then a cyclic order of one of them induces cyclic orders of the other two that in turn induce each other.
With the help of this fact, the cyclic orders of the pre-daisies contained in $V$ can be combined into one cyclic order of $\partial(V)$.

\begin{lem}\label{infpredaisy}
	If the pre-flowers contained in $V$ are pre-daisies, then there is a cyclic order on $\partial(V)$ such that all elements of $\sepclos{V}$ are unions of intervals of $\partial(V)$, and that cyclic order is unique up to mirroring.
\end{lem}
\begin{proof}
	Let $T_V$ be a cyclic order on $\partial(V)$ such that all elements of $\sepclos{V}$ are unions of intervals.
	Then for every pre-flower $\mathcal{F}\in \Phi(V)$ there is a unique cyclic order $T_{\mathcal{F}}$ on $\partial(\mathcal{F})$ such that the inclusion $\pi_{\mathcal{F}}:\partial(V)\rightarrow \partial(\mathcal{F})$ is monotone.
	Every union of elements of $\partial(\mathcal{F})$ that do not form an interval of $T_{\mathcal{F}}$ is also a union of elements of $\partial(V)$ that do not form an interval of $T_V$ and is thus not contained in $\sepclos{V}\supseteq \sepclosf$.
	Hence only unions of intervals of $T_{\mathcal{F}}$ can be contained in $\sepclosf$, so all unions of non-trivial intervals of $T_{\mathcal{F}}$ are indeed contained in $\sepclosf$.
	Thus, if the set of all pre-daisies contained in $V$ is denoted by $\Phi(V)$, then $T_V$ induces a family $(T_{\mathcal{F}})_{\mathcal{F}\in \Phi(V)}$ where every $T_{\mathcal{F}}$ is a cyclic order of $\mathcal{F}$ and such that the inclusions $\pi_{\mathcal{G}\mathcal{F}}: \partial(\mathcal{G})\rightarrow\partial(\mathcal{F})$ for pre-daisies $\mathcal{F}\subseteq \mathcal{G}$ are monotone.
	For this proof, such a family is called a consistent family of cyclic orders.
	
	In the other direction, every consistent family of cyclic orders $(T_{\mathcal{F}})_{\mathcal{F}\in \Phi(V)}$ induces a set of triples $T_V$ of $\partial(V)$ containing all those triples $(Q, R, R')$ for which there is $\mathcal{F}\in \Phi(V)$ such that $(\pi_{\mathcal{F}}(Q), \pi_{\mathcal{F}}(R), \pi_{\mathcal{F}}(R')) \in T_{\mathcal{F}}$.
	Then $T_V$ is cyclic as every $T_{\mathcal{F}}$ is cyclic, and it is linear because every element of $V$ is contained in some pre-flower.
	In order to show that $T_V$ is antisymmetric and transitive it suffices to show that for any two triples $(Q, R, R')$ and $(Q', R', R'')$ in $T_V$ there is some pre-flower $\mathcal{G}\subseteq V$ such that both $(\pi_{\mathcal{F}}(Q), \pi_{\mathcal{F}}(R), \pi_{\mathcal{F}}(R'))$ and $(\pi_{\mathcal{F}}(Q'), \pi_{\mathcal{F}}(R'), \pi_{\mathcal{F}}(R''))$ are contained in $T_{\mathcal{F}}$.
	But that is true as the union of any two pre-flowers contained in $V$ is a subset of another pre-flower contained in $V$ and the maps $\pi_{\mathcal{F}\mathcal{G}}$ are monotone.
	
	So the cyclic orders of $\partial(V)$ such that all elements of $\sepclos{V}$ are unions of intervals correspond to the consistent families of cyclic orders.
	Thus, in order to show this lemma, it suffices to show that every consistent family of cyclic orders $(T_{\mathcal{F}})_{\mathcal{F}\in \Phi(V)}$ can be constructed from, and is thus determined by, any of its cyclic orders $T_{\mathcal{F}}$.
	Indeed, for every pre-daisy $\mathcal{G} \subseteq V$ there is a pre-daisy $\mathcal{H} \subseteq V$ that contains both $\mathcal{F}$ and $\mathcal{G}$, and $T_{\mathcal{F}}$ induces a cyclic order $T_{\mathcal{H}}$ of $\mathcal{H}$ which in turn induces a cyclic order $T_{\mathcal{G}}$ of $\mathcal{G}$.
	If $\mathcal{H}' \subseteq V$ is another pre-daisy that contains both $\mathcal{F}$ and $\mathcal{G}$, and $T_{\mathcal{F}}$ induces $T_{\mathcal{H}'}$, then there is a pre-daisy $\mathcal{H}'' \subseteq V$ that contains $\mathcal{H}$ and $\mathcal{H}''$.
	In this case, by \cref{daisiesconsistentorders}, both $T_{\mathcal{H}}$ and $T_{\mathcal{H}'}$ induce the same cyclic order $T_{\mathcal{H}''}$ of $\mathcal{H}''$, namely the one that is also induced by $T_{\mathcal{F}}$.
	By the same lemma, $T_{\mathcal{H}}$, $T_{\mathcal{H}'}$ and $T_{\mathcal{H}''}$ all induce the same cyclic order of $\mathcal{G}$.
	Thus $T_{\mathcal{G}}$ does not depend on the choice of $\mathcal{H}$ and the family $(T_{\mathcal{F}})_{\mathcal{F}\in \Phi(V)}$ is well-defined.
	That it is a consistent family also follows from \cref{daisiesconsistentorders}.
	
	So in summary, the cyclic orders of $\partial(V)$ such that every element of $\sepclos{V}$ is the union of an interval correspond to the consistent families of cyclic orders, which in turn correspond to the cyclic orders of $\mathcal{F}$ for any pre-daisy $\mathcal{F} \subseteq V$.
	As there are exactly two of the latter, and the mirror of a suitable cyclic order of $\partial(V)$ is again suitable, there is a suitable cyclic order of $\partial(V)$ and that is unique up to mirroring.
\end{proof}

Note that the cyclic order obtained in \cref{infpredaisy} for $\partial(V)$ has the property that all elements of $\sepclos{V}$ are unions of intervals, but not all non-trivial unions of intervals need be contained in $\sepclos{V}$.
This is in contrast to the fact that \cref{thm:structurefiniteBsimple}, applied to a finite $V\in \mathcal{V}$, yields that the elements of $\partial(V)$ are contained in $\mathcal{B}$ and thus also in $\sepclos{V}$ and in $\mathcal{E}$.
Similarly, if $V\in \mathcal{V}$ is infinite and contains pre-anemones, then it is not necessarily the case that all unions of elements of $\partial(V)$ are contained in $\sepclos{V}$.
This problem could be avoided by asking that $\mathcal{B}\cup \{E\}$ should be closed under taking suprema of chains of separations, which for example could be achieved by asking that the elements of $\mathcal{P}$ should be closed under taking suprema of chains.
This approach is taken in \cite{EH:treesets}.
Even without extra assumptions, at least in the case where the pre-flowers contained in $V$ are pre-daisies, it is possible to describe quite precisely which unions of intervals of $\partial(V)$ are contained in $\mathcal{B}$.
In the following lemma, intervals are denoted with the help of cuts, as is described in the preliminaries.
For that, assume that some cyclic order of $\partial(V)$ has been fixed such that all elements of $V$ are unions of intervals of $\partial(V)$.
For cuts $v$ and $w$ of $\partial(V)$, $[v,w] \cap \partial(V)$ is a subset of $\partial(V)$ and its union is a subset of $E$ that is either contained in $V$ or is not contained in $V$.
Let $Z$ be the set of cuts $v$ such that there is a cut $w$ such that the union of $[v,w] \cap \partial(V)$ is contained in $V$ (equivalently, such that the union of $[w,v] \cap \partial(V)$ is contained in $V$).

\begin{lem}
	Let $V\in \mathcal{V}$ such that the pre-flowers contained in $V$ are pre-daisies.
	Then for all distinct elements $v$ and $w$ of $Z$, the union of $[v,w] \cap \partial(V)$ is contained in $\sepclos{V}$.
\end{lem}
\begin{proof}
	For this proof, for distinct cuts $x$ and $y$ of $\partial(V)$ denote the union of $[x,y] \cap \partial(V)$ by $S(x,y)$.
	Let $v'$ and $w'$ be elements of $Z$ such that $S(v,v')$ and $S(w,w')$ are contained in $V$.
	If $S(v,v')$ and $S(w,w')$ cross, then $S(v,w)$ is of the form $S(v,v') \setminus S(w,w')$ or $S(v,v') \cup (E \setminus S(w,w'))$ and thus contained in $\sepclos{V}$.
	As $v\neq w$, $S(v,v')$ and $S(w,w')$ cannot be equal, and if they are complements of each other then $S(v,w) = S(v,v') \in V$.
	So assume that $S(v,v')$ and $S(w,w')$ cross and are neither equal nor complements of each other.
	Let $\mathcal{F} \subseteq V$ be a pre-flower that contains $S(v,v')$ and $S(w,w')$.
	Then $\sepclosf$ contains a separation that crosses both $S(v,v')$ and $S(w,w')$.
	This separation is contained in $V$ and thus is of the form $S(u,u')$ for some cuts $u$ and $u'$ of $\partial(V)$.
	So some corner of $S(v,v')$ and $S(u,u')$ is of the form $S(v,u'')$ or $S(u'',v)$ for some $u'' \in \{u,u'\}$ and crosses $S(w,w')$.
	This corner $S$ is contained in $V$ and some corner of $S$ and $S(w,w')$ is of the form $S(v,w)$ or $S(w,v)$ and is also contained in $V$.
	Thus $S(v,w)$ is contained in $V$.	
\end{proof}

\subsection{The infinite tree structure}\label{sec:abstractinfinite}

Back to the construction of a decomposition tree(-like space):
The set $\mathcal{E}$ consists only of nested separations, it contains for every separation also its complement, and it contains neither $\emptyset$ nor $E$.
In the context of separations, this means that $\mathcal{E}$ is regular tree set.
The following remark is about the fact that a tree set can be turned into a tree-like space (that happens to be a finite tree if the tree set is finite) from which a tree-decomposition of $E$ can be constructed such that the separations induced by the edges are the elements of the tree set.
In this construction, the vertices $v$ of the tree-like space correspond to the consistent orientations $O_v$ of $E$, and the part of a vertex $v$ is $E \setminus \bigcup O_v$.
In order to avoid the technicalities that come with working with tree-like spaces, we will not work with them directly but work with tree sets and their consistent orientations instead.

\begin{rem}\label{treefortreeset}
	For regular tree sets $\tau$, there is in \cite{TreelikeSpaces} a construction of a tree-like space $T(\tau)$\footnote{This is a generalisation of the construction in \cite{TreeSets} for finite tree sets.}, together with a bijection from the tree set to the oriented edges of the tree-like space, such that for oriented edges $xy$ and $vw$ the following property holds:
	\begin{itemize}
		\item[]If $x$ and $w$ are contained in distinct components of $T - v$ and of $T - y$, then $S < S'$ for the separations $S$ and $S'$ mapped to the oriented edges $xy$ and $vw$ respectively.
	\end{itemize}
	Here, the vertices of $T(\tau)$ are the consistent orientations of the regular tree set.
	Then a decomposition of $E$ with decomposition ``tree'' $T(\tau)$ can be defined by letting the part of a vertex $O$ be $E \setminus \bigcup O$.
	It is then routine to show that the above extra property of $T(\tau)$ implies that this yields indeed a decomposition of $E$ such that the separations induced by the oriented edges of $T(\tau)$ are the elements of $\tau$.
\end{rem}

One important property of the tree-decomposition in \cref{thm:structurefiniteBsimple} was that every element of $\mathcal{V}$ belongs to a vertex of the decomposition tree, and that these vertices are distinct for distinct elements of $\mathcal{V}$.
Translated to tree sets, this means that for every $V\in \mathcal{V}$ there is a consistent orientation $O_V$ of $\mathcal{E}$ such that all elements of $O_V$ point towards $V$, and that distinct elements of $\mathcal{V}$ induce distinct consistent orientations.
In the infinite case, it is possible that this property does not hold.
Indeed, there is a candidate for $\mathcal{B}$ for which $\mathcal{E}$ is empty (implying that there is exactly one consistent orientation) while $\mathcal{V}$ has two elements.

\begin{ex}\label{EdoesnotdistinguishV}
	Let $E=\{0,1\} \times \mathbb{N}$ and let
	\begin{align*}
		\mathcal{B}=& \{ X \subseteq E \colon \emptyset \subsetneq X \subsetneq \{0\} \times \mathbb{N}, \text{ $X$ is finite}\}
		\cup \{ X \subseteq E \colon \{1\} \times \mathbb{N} \subsetneq X \subsetneq E, \text{ $E \setminus X$ is finite}\} \\
		&\cup \{ X \subseteq E \colon \emptyset \subsetneq X \subsetneq \{1\} \times \mathbb{N}, \text{ $X$ is finite}\}
		\cup \{ X \subseteq E \colon \{0\}\times\mathbb{N} \subsetneq X \subsetneq E, \text{ $E \setminus X$ is finite}\}.
	\end{align*}
	By case-distinction it is easy to see that the set $\mathcal{B}$ is closed under taking complements and under unions of two crossing elements.
	But $\mathcal{E}$ is empty, and $\mathcal{V}$ has two elements.
	In particular, taking the graph with exactly one vertex as decomposition tree is the only possibility to obtain a tree decomposition of $E$ such that the set of separations induced by the edges is $\mathcal{E}$.
	But there is no injective map from $\mathcal{V}$ into a set of vertices containing only one element.
\end{ex}

So in order to ensure that the elements of $\mathcal{V}$ belong to distinct vertices, we define

\begin{align*}
	\mathcal{E}':=\mathcal{E}\cup \bigcup_{V\in \mathcal{V}}\{A, E \setminus A \colon A \in \partial(V)\} && \text{and} && \mathcal{B}':=\mathcal{B}\cup \mathcal{E}'.
\end{align*}

Note that if every element of $\mathcal{V}$ is finite, in particular if $\mathcal{B}$ is finite, then $\mathcal{E}=\mathcal{E}'$ and $\mathcal{B}=\mathcal{B}'$.
The following lemmas establish basic facts about the interaction of elements of some $V\in \mathcal{V}$ and elements of $\mathcal{B}'\setminus V$.
Together they imply, among other things, that $\mathcal{E}'$ is the set of elements of $\mathcal{B}'$ that do not cross any other element of $\mathcal{B}'$ and that every $V\in \mathcal{V}$ induces a consistent orientation of $\mathcal{E}'$.

\begin{lem}\label{sepsinVnestedwithsepsnotinV}
	Let $V$ be an element of $\mathcal{V}$ and let $S$ be an element of $\mathcal{B}\setminus V$.
	Then some orientation of $S$ is contained in some element of $\partial(V)$.
\end{lem}
\begin{proof}
	First we show that some orientation of $S$ points towards all elements of $V$.
	Let $P$ be the set of elements of $V$ towards which $S$ points and let $Q$ be the set of elements of $V$ towards which $E\setminus S$ points.
	As $S$ is nested with every element of $V$, $P\cup Q=V$.
	Assume for a contradiction that there is an element $T\in P\cap Q$.
	Then both orientations of $S$ point towards $T$, so either $S$ is an orientation of $T$ or $T\in \{\emptyset, E\}$.
	But $T\notin \{\emptyset, E\}$ because it is an element of $\mathcal{B}$.
	Furthermore every element of $V$ crosses some element of $V$, so $V$ is closed under taking inverses and in particular $S$ is not an orientation of $T$.
	So $P\cap Q=\emptyset$.
	By the definition of $\mathcal{V}$, this implies that one of $P$ and $Q$ is empty.
\end{proof}

Note that this lemma has two important consequences.
First, as every element of $\partial(V)$ points towards all elements of $V \cup \partial(V)$, also $S$ has some orientation that points towards all elements of $V \cup \partial(V)$.
And second, if $A$ is an element of $\partial(V)$ that contains some orientation of $S$, then $E \setminus A$ points towards $S$ and thus every element of $V$ that is disjoint from $A$ also points towards $S$.
With these implications in mind we can extend \cref{sepsinVnestedwithsepsnotinV} to the case where $S\in \mathcal{B}' \setminus V$:

\begin{lem}\label{sepsinVnestedwithsepsnotinV2}
	Let $V$ be an element of $\mathcal{V}$ and let $S$ be an element of $\mathcal{B}' \setminus V$.
	Then some orientation of $S$ is contained in some element of $\partial(V)$.
\end{lem}
\begin{proof}
	We are left with the case $S\notin \mathcal{B}$, that is, $T\in \partial(W)$ for some orientation $T$ of $S$ and some $W\in \mathcal{V}$.
	If $W=V$ then the lemma holds, so assume otherwise.
	Let $A \in \partial(V)$ such that $E\setminus A$ points towards some element $S'$ of $W$.
	Then $E \setminus A$ is the union of all elements $Q$ of $V$ that point towards $S'$, and these $Q$ point by \cref{sepsinVnestedwithsepsnotinV} towards all elements of $W$.
	So all elements of $E\setminus A$ have to be contained in the same element $B$ of $\partial(W)$, implying that $E \setminus A$ is contained in $B$ and hence that $E \setminus B$ is contained in $A$.
	As both $T$ and $B$ are contained in $\partial(W)$, some orientation of $S$ is contained in $E \setminus B$ and thus contained in $A$.
\end{proof}

From this more general lemma we have the same (more general) consequences:
\begin{cor}\label{lem:consistenorientationfromflower}
	For all $V \in \mathcal{V}$ and all $S \in \mathcal{B}' \setminus V$ there is an orientation of $S$ that points towards all elements of $V \cup \partial(V)$.\qed
\end{cor}
\begin{cor}\label{lem:fromseptopetal}
	For all $V \in \mathcal{V}$ and all $S \in \mathcal{B}' \setminus V$ there is $A \in \partial(V)$ such that $E\setminus A$ points towards $S$.\qed
\end{cor}

So now we can show that $\mathcal{E}'$ is a good tree set from which to build a tree-like space for a tree-decomposition:

\begin{lem}
	The set $\mathcal{E}'$ is nested and consists of those elements of $\mathcal{B}'$ that are nested with all elements of $\mathcal{B}'$.
\end{lem}
\begin{proof}
	As every element of $\mathcal{B}'\setminus \mathcal{E}'$ crosses an element of $\mathcal{B}$, it suffices to show that every element $A$ of $\mathcal{E}'$ is nested with every element $S$ of $\mathcal{B}'$.
	Note that $A$ is not contained in any element of $\mathcal{V}$.
	So by \cref{lem:consistenorientationfromflower}, if some orientation of $S$ is contained in $V \cup \partial(V)$ for some $V\in \mathcal{V}$ then there is an orientation of $A$ that points towards all elements of $V \cup \partial(V)$ and thus $A$ and $S$ are nested.
	Hence we consider the case that $S$ is not contained in $V \cup \partial(V)$ for any $V \in \mathcal{V}$.
	But then $S \in \mathcal{E}$.
	Now, if $A\in \mathcal{E}$, then $A$ and $S$ are both contained in $\mathcal{E}$ and thus nested.
	Otherwise some orientation of $A$ is contained in $\partial(W)$ for some $W \in \mathcal{V}$, and so again by \cref{lem:consistenorientationfromflower}, $A$ and $S$ are nested.
\end{proof}

For every $V\in \mathcal{V}$ and every $S\in \mathcal{E}'$ there is by \cref{lem:consistenorientationfromflower} an orientation of $S$ that points towards all elements of $V \cup \partial(V)$.
As these orientations are unique, the set of elements of $\mathcal{E}'$ which point towards the elements of $V$ is an orientation $O_V$ of $\mathcal{E}'$.
Furthermore, it is a consistent orientation:
\begin{lem}
	For $V\in \mathcal{V}$, the set $O_V$ is a consistent orientation of $\mathcal{E}'$.
\end{lem}
\begin{proof}
	As we already know that $O_V$ is an orientation of $\mathcal{E}'$, it suffices to show for any elements $S$ and $T$ of $O_V$ that $S\cup T \neq E$.
	By \cref{sepsinVnestedwithsepsnotinV2} there is an element $A$ of $\partial(V)$ such that some orientation $S'$ of $S$ is contained in $A$.
	Then $S'$ points towards all elements of $V \cup \partial(V)$, hence $S = S'$.
	Similarly there is an element $B$ of $\partial(V)$ that contains $T$.
	So $S\cup T$ is a subset of $A \cup B$, and as $\partial(V)$ contains at least four elements, $A\cup B \neq E$.
\end{proof}

Thus every element of $\mathcal{V}$ naturally corresponds to a vertex of a tree-like space constructed from $\mathcal{E}'$.
Similarly for every element $e$ of the ground set $E$, the set of elements in $\mathcal{E}'$ which do not contain $e$ is a consistent orientation $O_e$ of $\mathcal{E}$.
These orientations are essentially all distinct:

\begin{lem}\label{treedistinguishesall}
	Let $X$ and $Y$ be distinct elements of $\mathcal{V}\cup E$.
	Then either the consistent orientations $O_X$ and $O_Y$ of $\mathcal{E}'$ are distinct or $X$ and $Y$ are both elements of $E$ and there is no element of $\mathcal{B}'$ which contains only one of $X$ and $Y$.
\end{lem}
\begin{proof}
	First consider the case that one of $X$ and $Y$, $X$ say, is contained in $E$ and the other is not.
	Then $X\in S$ for some $S\in \partial(Y)$, and $S\in \mathcal{E}'$ by definition of $\mathcal{E}'$.
	Then $S$ points towards all elements of $Y$, implying $S\in O_Y$ while $E \setminus S \in O_X$.
	
	Next consider the case that both $X$ and $Y$ are contained in $E$.
	If there is no element of $\mathcal{B}'$ which contains only one of $X$ and $Y$, then there is nothing to show.
	So assume that there is an element $S$ of $\mathcal{B}$ which contains $X$ but not $Y$.
	If $S\in \mathcal{E}$, then $S\in O_Y$ and $E\setminus S\in O_X$.
	If $S\in V \cup \partial(V)$ for some $V\in \mathcal{V}$, then the element $S'$ of $\partial(V)$ that contains $X$ does not contain $Y$, and so $S'\in O_Y$ while $E \setminus S' \in O_X$.
	
	Last consider the case that both $X$ and $Y$ are contained in $\mathcal{V}$.
	As $X$ and $Y$ are distinct, there is by \cref{lem:fromseptopetal} some $A\in \partial(X)$ such that $E \setminus A$ points towards some element of $Y$.
	Then $A$ cannot point towards all elements of $Y \cup \partial(Y)$, so $E \setminus A$ is contained in $O_Y$ while $A$ is contained in $O_X$.
\end{proof}

Recall the three properties of the tree decomposition for finite $\mathcal{B}$ that we wanted to also obtain for infinite $\mathcal{B}$.
Every element of $\mathcal{V}$ induces a consistent orientation of $\mathcal{E}'$ that is distinct from the consistent orientations of $\mathcal{E}'$ that are induced by other elements of $\mathcal{V}$ and by elements of $E$.
As $\mathcal{E}'=\mathcal{E}$ for finite $\mathcal{B}$, this implies that the second property essentially holds in the infinite case.
In order to obtain this result, the first property had to be modified, $\mathcal{E}$ had to be extended to $\mathcal{E}'$.
By doing so the third property is also essentially obtained.
Indeed, by construction $\partial(V)$ is now contained in $\mathcal{E}'$ for every $V \in \mathcal{V}$; and $O_V$ consists of exactly those elements of $\mathcal{E}'$ that are contained in some element of $\partial(V)$.

\begin{rem}[Tree of tangles]\label{ex:treeoftangles}
	Let $\mathcal{U}$ be a finite submodular universe, $k\in \mathbb{N}$ and $\mathcal{P}$ a set of regular $l$-profiles with $l\leq k$ which is closed under taking truncations.
	A common strategy for proving tree-of-tangles theorems is to find, recursively from $l=1$ to $l=k$, for every $l-1$-profile $Q$ in $\mathcal{P}$ a tree decomposition which distinguishes all the profiles in $\mathcal{P}$ whose truncation is $Q$, and to combine all the tree sets into a large one.
	To a certain extent, this strategy can be mimicked with the tools from this and the last section:
	Given an $l$-profile $Q\in \mathcal{P}$ with $l\leq k-1$, let $\mathcal{P}_Q$ be the set of profiles in $\mathcal{P}$ whose truncation to an $l$-profile is $Q$.
	Then $\phi_Q$, $\mathcal{B}_Q$ and $\mathcal{E}_Q$ can be defined as in this section, where $\mathcal{P}_Q$ takes the role of the set of $l+1$-profiles whose truncation is $Q$.
	Call the decomposition tree of the tree decomposition that displays the elements of $\mathcal{E}_Q$ (see also \cref{treefortreeset}) the \emph{abstract tree} of $Q$.
	Note that this tree is unique up to isomorphism of graphs.
	If $Q$ is not the empty set, then $k-1\neq 0$ and thus the truncation $Q'$ of $Q$ to a $k-2$-profile exists.
	In that case, the orientation $O_Q$ of $\mathcal{E}_{Q'}$ which $Q$ induces is a vertex of the abstract tree of $Q'$.
	Thus the set of abstract trees has itself a tree-structure.
	For an example see \cref{fig:treeoftangles}.
	Also for every tree set $\mathcal{E}_Q$ there is by \cref{cornersinSexist} a tree set $T_Q$ contained in $S$ which is mapped isomorphically to $\mathcal{E}_Q$ by $\phi_Q$.
	
	It is not always possible to combine the tree sets $T_Q$ into one tree set, as elements of distinct tree sets $T_Q$ need not be nested.
	Less abstract, the following can happen:
	There are integers $k\leq l$ and two $k+1$-profiles $P_1$ and $P_2$ which are distinguished by exactly two separations $\overrightarrow{r}$ and $\overleftarrow{r}$ of $S$ that have order $k$.
	Also, there are two $l+1$-profiles $Q_1$ and $Q_2$ which are distinguished by exactly two separations $\overrightarrow{s}$ and $\overleftarrow{s}$ of $S$ that have order $l$, and such that $\overrightarrow{s}$ crosses $\overrightarrow{r}$.
	As $k\leq l$, one of the orientations of $\overrightarrow{r}$ is contained in both $Q_1$ and $Q_2$, say $\overrightarrow{r}\in Q_1\cap Q_2$.
	Under these assumptions, $\overrightarrow{r}\vee \overrightarrow{s}$ has order at least $l+1$, as otherwise it would distinguish $Q_1$ from $Q_2$.
	Thus $\overrightarrow{r}\wedge \overrightarrow{s}$ has order less than $k$ and it does not distinguish $P_1$ from $P_2$, thus it must be contained in $P_1\cap P_2$.
	Likewise, $\overrightarrow{r}\wedge \overleftarrow{s}$ has order less than $k$ and is contained in $P_1\cap P_2$.
	Thus one of $P_1$ and $P_2$ contains not only $\overrightarrow{r}\wedge \overrightarrow{s}$ and $\overrightarrow{r}\wedge \overleftarrow{s}$ but also $\overleftarrow{r}$.
	\cite{confing} contains in Section 6 a graph with $k$-blocks (special cases of $k+1$-profiles) and $k+1$-blocks behaving exactly as described here, and introduces the notion of robustness.
	
	In \cite{ProfilesNew} that same notion of robustness is formulated in terms better suited to the context of this paper:
	A $k$-profile $P$ of a universe is robust if for all $\overrightarrow{r}\in P$ and all separations $\overrightarrow{s}$, if both $\overleftarrow{r}\wedge \overrightarrow{s}$ and $\overleftarrow{r}\wedge \overleftarrow{s}$ have order less than the order of $\overrightarrow{r}$, then they are not both contained in $P$.
	Profiles being robust has the following implication:
	A separation distinguishes two profiles $P_1$ and $P_2$ \emph{efficiently} if it distinguishes them but no separation of lesser order than $\overrightarrow{s}$ distinguishes $P_1$ and $P_2$.
	Separations contained in any $T_Q$ efficiently distinguish some profiles in $\mathcal{P}$, and typically separations in trees of tangles also efficiently distinguish two profiles of the set of profiles under consideration.
	Robustness of the profiles in $\mathcal{P}$ now ensures that, given two crossing separations of different orders which each distinguish two profiles efficiently, there is a corner of the two profiles that distinguishes all those profiles efficiently that are distinguished efficiently by the one separation of the original ones that has the bigger order.
	This property is used in the proof of \cite[Theorem 3.6]{ProfilesNew} to show that, under a slightly weaker notion of robustness, every robust set of profiles in a submodular separation system has a tree of tangles.
	In the same way, if all elements of $\mathcal{P}$ are robust then it can be shown that the tree sets $T_Q$ can be chosen in such a way that the union of all $T_Q$ is a nested set.
	
	Theorem 3.3 of \cite{ProfilesNew} makes a statement about the existence of canonical trees of tangles.
	The proof constructs a tree set $\tau$ such that the restriction of $\phi$ to $\tau$ has $\mathcal{E}$ as its image and is nearly injective or injective.
	The proof can be translated into a construction of $\tau$ from $\mathcal{E}$ as follows:
	Let $T$ be a tree that is the decomposition tree of a tree decomposition whose set of displayed separations is $\mathcal{E}$.
	Let $v$ be a vertex of $T$ such that the maximal distance from $v$ to a leaf is minimized.
	If $v$ is unique, then let $O\subseteq \mathcal{E}$ be the set of separations whose inverse is contained in the consistent orientation $v$.
	If $v$ is not unique, then there are only two possible choices $v$ and $v'$ which are joined by an edge of $T$, let $O$ be the set of separations whose inverse is contained in $v\cup v'$.
	By \cref{biggestpreimage} there is for every $\overrightarrow{r}\in O$ a biggest element of $\phi^{-1}(\overrightarrow{r})$.
	The tree set consisting of all these biggest elements equals $\tau$.
	Thus if $v$ is unique, then there is a canonical tree set contained in $S$ which is mapped to $\mathcal{E}$ isomorphically by $\phi$.
\end{rem}

\begin{figure}
	\newcommand{\dotradius}{1.3pt}
	\begin{tikzpicture}[dot/.style={fill, circle, inner sep=1.3pt}]
		\begin{scope}
			\foreach \xco/\yco/\name in {0/-1/A,-1/0/B,0/0/C,0/1/D,1/0/E,1.7/0.7/F,1.7/-0.7/G}
			{\coordinate (A\name) at (\xco cm,\yco cm);
				\draw (A\name) node [dot] {};}
			\draw (AA)--(AC)--(AB) (AD) -- (AC) -- (AE) -- (AF) (AE) -- (AG);
		\end{scope}
		\begin{scope}[xshift=4.5cm]
			\foreach \xco/\yco/\name in {-0.7/0.7/A,0/1.4/B,0.7/0.7/C,1.4/0/D,0.7/-0.7/E,0/-1.4/F,-0.7/-0.7/G,-1.4/0/H}
			{\coordinate (B\name) at (\xco cm,\yco cm);
				\draw (B\name) node [dot] {};}
			\draw (BA)--(BB)--(BC)--(BD)--(BE)--(BF)--(BG)--(BH)--(BA)--(BC)--(BE)--(BG)--(BA);
		\end{scope}
		\begin{scope}[xshift=7.5cm, yshift=-0.5cm]
			\foreach \xco/\yco/\name in {0/1/A,1/1/B,1/0/C,0/0/D}
			{\coordinate (C\name) at (\xco cm, \yco cm);
				\draw (C\name) node [dot] {};}
			\draw (CA) -- (CB) -- (CC) -- (CD) -- (CA) -- (CC) (CB) -- (CD);
		\end{scope}
	\end{tikzpicture}
	\begin{tikzpicture}[dot/.style={fill, circle, inner sep=\dotradius}]
		\draw (-1cm-\dotradius,0cm) -- (-4.5cm,1.5cm) (-1cm+\dotradius,0cm) -- (-0.5cm,1.5cm);
		\draw (-\dotradius,0cm) -- (0.5cm,1.5cm) (\dotradius,0cm) -- (1.5cm,1.5cm);
		\draw (1cm-\dotradius,0cm) -- (2.5cm,1.5cm) (1cm+\dotradius,0cm) -- (3.5cm,1.5cm);
		\draw (1cm-\dotradius,1.5cm) -- (-1.5cm,3.5cm) (1cm+\dotradius,1.5cm) -- (1.5cm,3.5cm);
		\draw (3cm-\dotradius,1.5cm) -- (2.5cm,3.5cm) (3cm+\dotradius,1.5cm) -- (3.5cm,3.5cm);
		\begin{scope}[canvas is xz plane at y=0cm]
			\draw (-1,0) node [dot] {} -- (0,0) node [dot] {} -- (1,0) node [dot] {};
			\draw (0cm,0cm) ellipse [x radius=1.5cm, y radius=0.5cm];
		\end{scope}
		\begin{scope}[canvas is xz plane at y=1.5]
			\draw [xshift=-3cm] (-1,0) node [dot] {} -- (0,0) node [blue,dot] {} -- (0,1) node [dot] {} (0,-1) node [dot] {} -- (0,0) -- (1,0) node [red,dot] {} -- (1.7,0.7) node [dot] {} (1,0) -- (1.7,-0.7) node [dot] {};
			\draw [xshift=-3cm] (0.5,0) ellipse [x radius=2cm,y radius=1.5cm];
			
			\draw [xshift=1cm] (0,0) node [dot] {};
			\draw [xshift=1cm] (0,0) ellipse [radius=0.5cm];
			
			\draw [xshift=3cm] (0,0) node [dot] {};
			\draw [xshift=3cm] (0,0) ellipse [radius=0.5cm];
		\end{scope}
		\begin{scope}[canvas is xz plane at y=3.5]
			\draw (-1,0) node [dot] {} -- (0,0) node [green!70!black,dot] {} -- (0,1) node [dot] {} (0,-1) node [dot] {} -- (0,0) -- (1,0,0) node [dot] {};
			\draw (0,0) ellipse [radius=1.5cm];
			\draw [xshift=3cm] (0,0) node [dot] {} ellipse [radius=0.5cm];
		\end{scope}
		\path (0cm,5cm);
	\end{tikzpicture}
	\caption{A graph with three components and its collection of abstract trees for the set of all $k$-profiles with $k\leq 3$. The black vertices correspond to profiles, the blue vertex corresponds to an anemone, the green vertex corresponds to a daisy and the red vertex corresponds to nothing.}\label{fig:treeoftangles}
\end{figure}
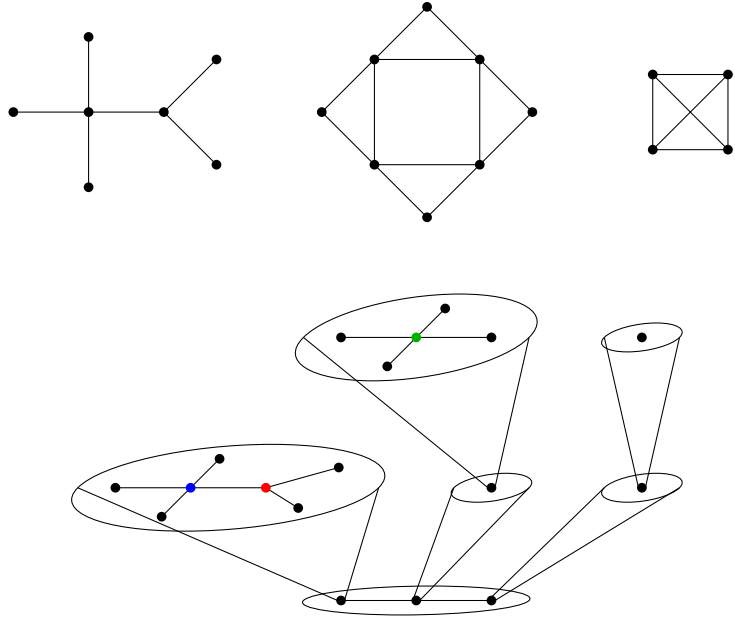

\subsection{Profiles in the abstract setting}\label{sec:abstractprofiles}
Recall from \cref{sec:abstractionfromunderlyingsepsys} that if $\mathcal{B}$ is obtained from a submodular universe $\mathcal{U}$ and a set $\mathcal{P}$ of $k$-profiles of $\mathcal{U}$ then the elements of $\mathcal{P}$ induce profiles of $\mathcal{B}$ and profiles of $\mathcal{B}$ induce $k$-profiles of $\mathcal{U}$.
Furthermore every

Also the profiles of $\mathcal{B}$ can be related to the tree.
For every profile $P$ of $\mathcal{P}$, the intersection of $P$ with $\mathcal{E}$ is a consistent orientation $O_P$ of $\mathcal{E}$, and thus corresponds to a vertex of the tree.
That vertex cannot simultaneously belong to an element of~$\mathcal{V}$.

The following lemma shows that if $O$ is a vertex of the tree-like space obtained from $\mathcal{E}'$ that has degree at least $4$ and is not of the form $O_V$ for any $V\in \mathcal{V}$ then there is a unique profile of $\mathcal{B}$ that contains $O$.

\begin{lem}\label{profilefordegfourvertices}
	Let $O$ be a consistent orientation of $\mathcal{E}$ such that $O\neq O_V\cap \mathcal{E}$ for all $V\in \mathcal{V}$.
	Then there is a unique consistent orientation of $\mathcal{B}$ which contains $O$, and if $O=O_e$ for some $e\in E$ or the restriction of $\subseteq$ to $O$ has at least four maximal elements\myfootnote{There is a notion of the degree of a vertex in a tree-like space. The assumption that $O$ has at least four maximal elements is a stronger assumption than that the vertex $O$ has degree at least $4$.}, then the consistent orientation is in fact a profile.
\end{lem}
\begin{proof}
	For $V\in \mathcal{V}$ there is $S_V\in O\setminus O_V$, and every element of $\sepinn{V}$ has a unique orientation which points towards $S_V$.
	This orientation does not depend on the choice of $S_V$.
	The union $O'$ of $O$ and, for all $V$ in $\mathcal{V}$, the set of elements of $\sepinn{V}$ which point towards $S_V$, is the unique consistent orientation of $\mathcal{B}$ which contains $O$.
	
	In order to show that $O'$ is a profile it suffices to show that for any three elements $Q$, $R$ and $R'$ of $O'$ the union is not the whole ground set.
	Let $S_1,\ldots, S_n$ be the maximal elements of $O'$.
	These are also the maximal elements of $O$.
	Then $Q\cup R\cup R'\leq S_i\cup S_j\cup S_l$ for some, not necessarily distinct, indices $i$, $j$ and $l$.
	If $O=O_e$ for some $e\in E$, then no $S_m$ contains $e$ and hence $S_i\cup S_j\cup S_l\neq E$.
	If the degree of $O$ in $T$ is at least $4$, then $S_i\cup S_j\cup S_l$ is less than or equal to the inverse of another maximal element of $O'$.
	As $O'$ is a subset of $\mathcal{B}$ and hence only contains non-trivial separations, $S_i\cup S_j\cup S_l\neq E$.
	In both cases $Q\cup R\cup R'\neq E$ and thus $O'$ is a profile.
\end{proof}

If $O = O_V \cap \mathcal{E}$ for some $V \in \mathcal{V}$, then the previous lemma cannot be applied to $O$.
But $O$ can still be extended to a consistent orientation of $\mathcal{B} \setminus V$ such that, in order to obtain a profile of $\mathcal{B}$, it suffices to add a profile of $\sepclos{V}$ that contains all elements of $\partial(V) \cap \mathcal{E}$.
Recall that no such profile exists if $V$ is finite, but that is not necessarily so if $V$ is infinite.

\begin{lem}
	Let $O$ be a consistent orientation of $\mathcal{E}$ such that $O = O_V \cap \mathcal{E}$ for some $V \in \mathcal{V}$.
	Then the set $O'$ of all elements of $\mathcal{B} \setminus V$ that point towards all elements of $V$ is a consistent orientation of $\mathcal{B} \setminus V$.
	If furthermore $P$ is a profile of $\sepclos{V}$ that contains $\partial(V) \cap \mathcal{E}$ then $O' \cup P$ is a profile of $\mathcal{B}$.
\end{lem}
\begin{proof}
	By \cref{lem:consistenorientationfromflower}, every element of $\mathcal{E}$ has an orientation that points towards all elements of $V$.
	Thus $O'$ is an orientation of $\mathcal{E}$.
	Also, if $S$ is a subset of an element of $\mathcal{E}$ that points towards all elements of $V$, then also $A$ points towards all elements of $V$.
	Thus $O'$ is down-closed and thus a consistent orientation.
	
	In order to show that $O' \cup P$ is an orientation of $\mathcal{B}$ let $S$ be an element of $O'$ that is an orientation of an element of $\sepclos{V}$.
	Then $S$ is not contained in $V$, so in order to be contained in $\sepclos{V}$ it has to be an orientation of an element of $\partial(V)$.
	As all elements of $\partial(V)$ are contained in $\mathcal{E}'$ and thus nested with all elements of $\mathcal{B}$, this implies that some orientation of $S$ is contained in $\partial(V) \cup \mathcal{E}$.
	Because all elements of $\partial(V)$ point towards all elements of $V$, and $S$ does the same, $S$ itself has to be contained in $\partial(V) \cap \mathcal{E}$.
	So $S$ is contained in $P$, and hence $O' \cup P$ contains exactly one orientation of every element of $\mathcal{B}$.
	Thus $O' \cup P$ is an orientation of $\mathcal{B}$.
	
	In order to show that $O' \cup P$ is a consistent orientation, let $S$ and $T$ be elements of $\mathcal{B}$ with $S \subseteq T \in O \cup P'$.
	It suffices to show that if $S\notin V$ and $T \in V$ then $S \in O'$.
	By \cref{sepsinVnestedwithsepsnotinV} there is $A \in \partial(V)$ such that some orientation of $S$ is contained in $A$.
	If $A \subseteq T$, then the fact that $S\neq \emptyset$ implies that $E \setminus S$ is not contained in $A$ and thus $S$ is contained in $A$.
	Hence $S$ points towards all elements of $V$ and thus is contained in $O'$.
	So assume that $A$ is not a subset of $T$, implying $A \cap T = \emptyset$.
	Then also $A \cap S = \emptyset$, so the orientation of $S$ that is contained in $A$ has to be $E \setminus S$.
	So $S = E \setminus A$.
	As $S$ is contained in $T$ and the latter is disjoint from $A$, this implies that $S = T$ and hence $S\in O'$.
	
	In order to show that $O' \cup P$ is a profile, let $S$ and $T$ be elements of $O' \cup P$ such that $S \cup T$ is contained in $\mathcal{B}$.
	Again it suffices to consider the case that $S \in \mathcal{B} \setminus V$ and $T \in V$, and again there is $A \in \partial(V)$ such that some orientation of $S$ is contained in $A$.
	As $S$ is non-empty and points towards all elements of $V$, this orientation has to be $S$, so $S \subseteq A$.
	Because $S \cup T$ is contained in $\mathcal{B}$, it has to be nested with $A$, so $S=A$ and thus $S\in P$.
	By the profile property of $P$, this implies $S \cup T \in P \subseteq O' \cup P$.
\end{proof}

Now to the question of when profiles of $\sepclos{V}$ exist.
One easy sufficient condition for the existence of a profile of $\sepclos{V}$ is the existence of some $A \in \partial(V)$ that is not contained in $\mathcal{B}$.
Then the set $P$ of all elements of $\sepclos{V}$ that are disjoint from $A$ is a profile of $\sepclos{V}$ that contains $\partial(V) \cap \mathcal{E}$.
For a second method of constructing profiles of $\sepclos{V}$, the notion of an ultrafilter from topology is useful, see also \cref{defn:ultrafilter,existencefreeultrafilter}.

Note that in the terminology of separation systems, if $| E | \geq 2$ then a set $U$ is an ultrafilter of $E$ if and only if the set $\{A \subseteq E \colon A \notin U\}$ is a profile of the universe of bipartitions of $E$.
A free ultrafilter of $\partial(V)$ induces a profile of $\sepclos{V}$ that contains every element of $\partial(V)$.

\begin{lem}
	If there is a free ultrafilter of $\partial(V)$ then there also is a profile of $\sepclos{V}$ that contains all elements of $\partial(V)$.
\end{lem}
\begin{proof}
	Let $U$ be a free ultrafilter of $\partial(V)$.
	Define $U' = \{S \subseteq E \colon S = \bigcup U' \text{ for some $U' \in U$}\}$ and let $P$ consist of all those elements of $\sepclos{V}$ that are not contained in $U'$.
	Then $P$ is a profile of $\sepclos{V}$.
	Let $A$ be an element of $\partial(V)\cap \mathcal{E}$.
	Then $\{A\}$ is not contained in $U$, so $A$ is not contained in $U'$.
	But $A$ is contained in $\sepclos{V}$ and thus in $P$.
\end{proof}

It is a well-known result from topology that every infinite set has a free ultrafilter, so if $V$ is infinite then $\partial(V)$ has a free ultrafilter that induces a profile of $\sepclos{V}$, but distinct ultrafilters can induce the same profile of $\sepclos{V}$.

\section{\texorpdfstring{Distinct maximal $k$-pseudo\-flowers are essentially nested}{Distinct maximal pseudoflowers are essentially nested}}

This section works again in the framework of a connectivity system.
Let $k\geq 1$ be an integer and let $\mathcal{P}$ be a set of regular $k$-profiles that have the same truncation.
Then by \cref{sec:abstractionfromunderlyingsepsys}, $\mathcal{P}$ induces a separation system $\mathcal{B}$ to which the results from \cref{sec:abstract} can be applied.
In particular, $k$-pseudo\-flowers that distinguish elements of $\mathcal{P}$ can be shown to relate to elements of $\mathcal{V}$ (defined for $\mathcal{B}$) and to thus be essentially nested.
Call a separation of order $k-1$ \emph{relevant} if it distinguishes two elements of the set $\mathcal{P}$, that is, if its image under $\phi$ is contained in $\mathcal{B}$.
We first associate $k$-pseudoflowers with elements of $\mathcal{V}$.

\begin{lem}\label{flowertoabstractflower}
	Let $\Phi$ be a $k$-pseudo\-flower which distinguishes at least four profiles from $\mathcal{P}$.
	Then there is a unique $V(\Phi)\in \mathcal{V}$ such that for every separation $S$ displayed by $\Phi$ the set $\phi(S)$ is contained in $\sepclos{V(\Phi)}\cup \{\emptyset, \mathcal{P}\}$.
\end{lem}
\begin{proof}
	Let $\mathcal{S}$ be the set of separations displayed by $\Phi$ which are contained in two profiles which are distinguished by $\Phi$ and whose inverses are also contained in two profiles which are distinguished by $\Phi$.
	As $\Phi$ distinguishes at least four profiles, the set $\mathcal{S}$ contains at least two elements and $\phi(\mathcal{S})$ is a pre-flower.
	Thus there is $V(\Phi)$ in $\mathcal{V}$ such that $\phi(S)\in V(\Phi)$ for all separations $S$ in $\mathcal{S}$.
	
	Let $S$ be a relevant separation displayed by $\Phi$ that is not contained in $\mathcal{S}$.
	First consider the case that $\Phi$ does not distinguish elements of $\mathcal{P}$ that contain $S$.
	As $\Phi$ distinguishes at least four profiles, there are crossing elements $S_1$ and $S_2$ in $\Phi$ such that $S=S_1 \cup S_2$ and thus $\phi(S)$, which equals $\phi(S_1)\cup \phi(S_2)$, is contained in $\sepclos{V(\Phi)}$.
	Similarly, if $\Phi$ does not distinguish elements of $\mathcal{P}$ that contain $E \setminus S$ then $\phi(S)$ is contained in $\sepclos{V}$.
	Thus $\phi(S)\in \sepclos{V(\Psi)}\cup \{\emptyset, \mathcal{P}\}$ for all separations $S$ displayed by $\Phi$.
\end{proof}

\begin{lem}
	Every $k$-pseudo\-flower $\Phi$ which distinguishes at least four profiles of $\mathcal{P}$ has a concatenation into an anemone if and only if $V(\Phi)$ is a $k$-anemone. Similarly, $\Phi$ has a concatenation into a $k$-daisy if and only if $V(\Phi)$ is a daisy.
\end{lem}
\begin{proof}
	Let $\Phi'$ be a $k$-flower with exactly four petals $P_1$, $P_2$, $P_3$, and $P_4$ (such that the cyclic order is induced by the linear order on $\{1,2,3,4\}$) which is a concatenation of $\Phi$ and distinguishes four profiles.
	Then $V'$ defined as $\{\phi(P_1\cup P_2),\phi(P_2\cup P_3),\phi(P_3\cup P_4), \phi(P_4\cup P_1)\}$ is a pre-flower which is contained in $V(\Phi)$.
	It suffices to show that $\Phi'$ is a $k$-anemone if and only if $V'$ is a pre-anemone.
	If $\Phi'$ is a $k$-anemone, then the set $P_1\cup P_3$ has order $k-1$ and thus $\phi(P_1\cup P_3)$ witnesses that $V'$ is a pre-anemone.
	So assume that $V'$ is a pre-anemone and let $S$ be an element of $S$ such that $\phi(S)=\phi(P_1)\cup \phi(P_3)$.
	By repeated application of \cref{cornersinSexist} there is a separation $T$ in $S$ which is equivalent to $S$ and satisfies $P_1\subseteq T$, $P_3\subseteq T$, $P_2\subseteq E \setminus T$ and $P_4\subseteq E \setminus T$.
	But then $T=P_1\cup P_3$ and thus $\Phi'$ is a $k$-anemone.
\end{proof}

Now we can translate the fact that distinct elements of $\mathcal{V}$ are nested back to $k$-pseudoflowers.

\begin{lem}\label{petalforsep}
	Let $\Phi$ be a $k$-pseudo\-flower which distinguishes at least four profiles and let $S$ be a relevant separation of order $k-1$ such that $\phi(S)$ points towards all elements of $V(\Phi)$ and such that $\phi(S)$ has at least two elements\myfootnote{Recall that $\phi(S)$ is a subset of $\mathcal{P}$ and that, if it has at least two elements, then there are at least two profiles in $\mathcal{P}$ which contain $E\setminus S$.}.
	Then there is a petal $Q$ of $\Phi$ such that $\phi(S) \subseteq \phi(Q)$.
\end{lem}
\begin{proof}
	Assume without loss of generality that if $\Phi$ is a strong $k$-pseudo\-anemone, then it is a $\leq_A$-maximal strong $k$-pseudo\-anemone and in particular $\leq$-maximal.
	Let $P_1$ and $P_2$ be two elements of $\mathcal{P}$ which are contained in $S$.
	As $\phi(S)$ points towards all elements of $V(\Phi)$, it also points towards all $\phi(T)$ where $T$ is a union of petals of $\Phi$ which is also contained in $S$.
	In particular no union of petals of $\Phi$ distinguishes $P_1$ and $P_2$.
	Thus by \cref{twoprofilesatnonpetal} the profiles $P_1$ and $P_2$ are located at the same petal $Q$ of $\Phi$.
	As $P_1\in \phi(S)\cap \phi(Q)$, the fact that $\phi(S)$ points towards $\phi(Q)$ implies that $\phi(S) \subseteq \phi(Q)$.
\end{proof}

\begin{cor}
	Let $\Phi$ and $\Phi'$ be $k$-pseudo\-flowers which distinguish at least four profiles each and such that $V(\Phi)\neq V(\Phi')$.
	Then $\Phi$ and $\Phi'$ have petals $P_i$ and $P_j'$ respectively such that $\phi(P_i)^* \subseteq \phi(P_j')$.
\end{cor}
\begin{proof}
	Let $S(I')$ be a separation displayed by $\Psi$ such that $\phi(S(I'))\in V(\Psi)$.
	As $V(\Phi)$ and $V(\Psi)$ are distinct and thus disjoint, $\phi(S(I'))\notin V(\Phi)$.
	So there is by \cref{sepsinVnestedwithsepsnotinV} an orientation of $\phi(S(I'))$ which points towards all elements of $V(\Phi)$.
	Assume, by replacing $S(I')$ with its inverse if necessary, that $\phi(S(I'))$ points towards the elements of $V(\Phi)$.
	By \cref{petalforsep} there is a petal $P_i$ of $\Phi$ such that $\phi(S(I')) \subseteq \phi(P_i)$.
	As $\phi(P_i)^*$ points towards $\phi(S(I'))$ and thus towards all elements of $V(\Psi)$, there is again by \cref{petalforsep} a petal $P_j'$ of $\Psi$ such that $\phi(P_i)^* \subseteq \phi(P_j')$.
\end{proof}

\section{The existence of infinite daisies}

Most statements about flowers in matroids are already true in separation systems of bipartitions whose order function is limit-closed.
One fact which does need extra properties of matroids is the fact that there are no infinite daisies, which will be proven now.
From here until \cref{noinfchainofdaisies} let $M$ be a (possibly infinite) matroid on ground set $E$.

Basics about infinite matroids can be found in \cite{InfMatroidAxioms} and about connectivity in infinite matroids in \cite{InfMatFinConn}.
The contraction and deletion of $X$ will be denoted $M\contract X$ and $M\delete X$ respectively, and the contraction onto $X$ and restriction to $X$ will be denoted $M\contractonto X$ and $M\restrict X$.
We will need the following basic lemma about bases in an infinite matroid.
In it, $|B\setminus B'|-|B'\setminus B|$ is essentially the definition of the connectivity of $X$, except that $B'$ need not be a subset of $B$ and thus the deviation of $B'$ from being a subset of $B$ has to be substracted as an error term.

\begin{lem}\label{connectivityfornonnestedbases}
	Let $X\subseteq E$, $B$ a base of $M\restrict X$ and $B'$ a base of $M\contractonto X$ such that $B\setminus B'$ is finite.
	Then $\lambda(X)=|B\setminus B'|-|B'\setminus B|$.
\end{lem}
\begin{proof}
	Let $A$ be a base of $M\contractonto X$ that is contained in $B$.
	Then $A\setminus B'$ is a subset of $B\setminus B'$ and thus finite, so $|B'\setminus A| - | A \setminus B'| =0$.
	So
	\begin{align*}
		|B\setminus B'| &= |B\setminus B'| + |B'\setminus A| - |A\setminus B'|\\
		&= |B\setminus A| + |B'\setminus B| + |A\setminus B'| - |A\setminus B'|\\
		&= |B\setminus A| + |B'\setminus B|.
	\end{align*}
	As $B\setminus B'$ is finite, also $B'\setminus B$ has to be finite and thus $\lambda(X)=|B\setminus A| = |B\setminus B'| - |B'\setminus B|$.
\end{proof}

Recall that the local connectivity $\localconn_M(X,Y)$ in a finite matroid $M$ of two disjoint sets $X$ and $Y$ is defined as the connectivity of $X$ in the restriction of $M$ to $X\cup Y$.
Define local connectivity in infinite matroids the same.

We want to apply the following characterisation of local connectivities in $k$-flowers:

\begin{lem}[\cite{AikinOxley08}, part of Theorem 1.3 and Lemma 3.6]\label{localconninfinflowers}
	Let $k$ be an integer, $k\geq 1$, and let $(P_1,\ldots, P_n)$ be a $k$-flower with at least five petals.
	Denote $\localconn\,(P_1,P_2)$ by $c$ and $\localconn\,(P_1,P_3)$ by $d$.
	Then the local connectivity of any two non-adjacent petals is $d$.
	$(P_1,\ldots,P_n)$ is a $k$-anemone if and only if $c=d$.
	Furthermore, the local connectivity of a non-trivial interval $I'$ of $I$ and a non-trivial subset $I''$ of $I\setminus I'$ is
	\begin{displaymath}
		\localconn_M(S(I'),\bigcup_{i \in I''} P_i)=
		\begin{cases}
			 d & \text{if no element of $I''$ is adjacent to an element of $I'$}\\
			 2c-d & \text{if two elements of $I''$ are adjacent to elements of $I'$}\\
			 c & \text{otherwise}.
		\end{cases}
	\end{displaymath}
\end{lem}

In particular any two adjacent petals have the same local connectivity $c$ and any two non-adjacent petals have the same local connectivity $d$.
In order to be able to apply \cref{localconninfinflowers} in infinite matroids, we need the following lemma.
It tells us that when we work with (local) connectivity in a partition of $E$, we may assume every partition class to be finite:

\begin{lem}\label{infmattofinmat}
	Let $X\subseteq E$ be a set of finite connectivity.
	Then there are disjoint subsets $C$ and $D$ of $X$ such that for $F:=X\setminus (C\cup D)$ and $N:=M\contract C \delete D$ the equations $M\contract X=N\contract F$, $M\delete X = N\delete F$ and $|F|=\lambda_M(X)$ hold.
	Furthermore, for all disjoint subsets $Y$ and $Z$ of $E\setminus X$, the equations $\lambda_M(Y)=\lambda_N(Y)$, $\localconn_M(Y,Z)=\localconn_N(Y,Z)$ and $\localconn_M(Z, X \cup Y)=\localconn_N(Z,F\cup Y)$ hold.
\end{lem}
\begin{proof}
	Let $B$ be a base of $M\restrict X$ and $C\subseteq B$ a base of $M.X$.
	Denote $X \setminus B$ by $D$.
	Then, by definition of connectivity, $X\setminus (C\cup D)$ has size $\lambda(X)$.
	Every element $d$ of $D$ is spanned by $B$ in $M$ and thus a loop of $M\contract B$.
	Hence $D$ is a union of connected components of $M\contract B$, so
	\begin{displaymath}
		M \contract X = M \contract B \delete D = M\contract C\delete D\contract F.
	\end{displaymath}
	Dually, $D$ is a base of $M^*\contractonto X$ and $X\setminus C$ is a base of $M^*\restrict X$, so
	\begin{displaymath}
		M\delete X = (M^*\contract X)^* = (M^*\contract D\delete C\contract F)^*=M\contract C \delete D \delete F.
	\end{displaymath}
	Furthermore, for every subset $S$ of $E\setminus X$, the equalities $M\restrict S = N\restrict S$ and $M\contractonto S$ and $N\contractonto S$ hold.
	In particular $M\restrict (Y\cap Z)=N\restrict (Y\cup Z)$, so
	\begin{displaymath}
		\localconn_M(Y,Z)=\lambda_{M\restrict (Y\cup Z)}(Y) = \lambda_{N\restrict (Y\cup Z)}(Y)=\localconn_N(Y,Z).
	\end{displaymath}
	Also the equations
	\begin{displaymath}
		(M\restrict (X\cup Y \cup Z))\restrict Z = M \restrict Z = N\restrict Z = (N\restrict(F\cup Y\cup Z))\restrict Z
	\end{displaymath}
	and
	\begin{displaymath}
		(M\restrict(X\cup Y\cup Z)).Z=(M\contract X\contract Y)\restrict Z = (N\contract F\contract Y)\restrict Z = (N\restrict (F\cup Y\cup Z)).Z
	\end{displaymath}
	hold and thus $\localconn_M(Z,X\cup Y)=\localconn_N(Z,F\cup Y)$.
	In particular
	\begin{displaymath}
		\lambda_M(Y)=\localconn_M(Y,E\setminus Y)=\localconn_N(Y,E(N)\setminus Y)=\localconn_N(Y).\qedhere
	\end{displaymath}
\end{proof}

By applying \cref{infmattofinmat} to every partition class, \cref{localconninfinflowers} can be applied to finite flowers in infinite matroids.
As the connectivity function of a matroid is the same as the connectivity function of its dual, every $k$-flower $\Phi$ of $M$ is also a $k$-flower of $M^*$.
We can now relate the parameters $c$ and $d$ of $\Phi$ in $M$ to those of $\Phi$ in $M^*$.
In order to do that, we use the following observation about the connectivity function:

\begin{lem}\label{connectivitybasic1}
	Let $C$ and $D$ be disjoint subsets of $E$.
	Then $\lambda_M(C\cup D)=\lambda_{M\contract C}(D) + \lambda_{M\delete D}(C)$.
\end{lem}
\begin{proof}
	Let $B_1$ be a base of $M\delete (C\cup D)$, $B_2$ a base of $M\contract C \delete D$ and $B_3$ a base of $M\contract (C\cup D)$ such that $B_3\subseteq B_2\subseteq B_1$.
	Denote $E\setminus (C\cup D)$ by $A$.
	Then
	\begin{align*}
		\lambda_M(C\cup D)&=\lambda_M(A) = |B_1\setminus B_3| =| B_1 \setminus B_2| + |B_2\setminus B_3|\\
		&= \lambda_{M\delete D}(A) + \lambda_{M\contract C}(A) = \lambda_{M\delete D}(C) + \lambda_{M\contract C}(D).\qedhere
	\end{align*}
\end{proof}

\begin{lem}\label{localconnindual}
	Let $\Phi$ be a finite $k$-flower of $M$ with at least five petals and parameters $\localconn_M(P_1,P_2)=c$ and $\localconn_M(P_1,P_3)=d$.
	Denote $\localconn_{M^*}(P_1,P_2)$ by $c^*$ and $\localconn_{M^*}(P_1,P_3)$ by $d^*$.
	Then $c+c^*=k-1$ and $c^*-d^*=c-d$.
\end{lem}
\begin{proof}
	Let $i$ and $j$ be distinct indices. Then
	\begin{equation*}
		\begin{split}
			\localconn_{M^*}(P_i,P_j)&=\lambda_{M^*\restrict(P_i\cup P_j)}(P_j)=\lambda_{M\contractonto(P_i\cup P_j)}(P_j)\\
			&=\lambda_M(E\setminus P_i) - \lambda_{M\delete P_j}(E\setminus (P_i\cup P_j))\\
			&=k-1-\localconn_M(P_i,E\setminus (P_i\cup P_j))
		\end{split}
	\end{equation*}
	where the third equality holds by \cref{connectivitybasic1}.
	So
	\begin{equation*}
		c^*=\localconn_{M^*}(P_1,P_2) =k-1-\localconn_M(P_1,E\setminus(P_1\cup P_2)) =k-1-c
	\end{equation*}
	where the last equation holds by \cref{localconninfinflowers}.
	Furthermore,
	\begin{equation*}
		\begin{split}
			d^*&=\localconn_{M^*}(P_1,P_3)=k-1-\localconn_M(P_1,E\setminus (P_1\cup P_3))\\
			&=k-1-(2c-d)=k-1-c-(c-d)=c^*-(c-d),
		\end{split}
	\end{equation*}
	where again the third equality also holds by \cref{localconninfinflowers}.
\end{proof}

Now we can determine the effect of deleting or contracting a petal.
In both cases, the $k$-flower essentially becomes an $l$-flower for some $l\leq k$.
Deleting a petal reduces $k$ as far as possible while keeping $c$ and $d$ constant:

\begin{lem}\label{localconnindeletion}
	Let $(P_1,\ldots,P_n)$ be a $k$-flower with at least five petals and denote $\localconn_M(P_1,P_2)$ by $c$ and $\localconn_M(P_1,P_3)$ by $d$.
	Then $(P_n\cup P_2,P_3,\ldots,P_{n-1})$ is a $2c-d+1$-flower of $M\delete P_1$ in which adjacent petals have local connectivity $c$ and non-adjacent petals have local connectivity $d$.
\end{lem}
\begin{proof}
	Deleting $P_1$ does not change the local connectivity of sets which are disjoint from $P_1$.
	So it suffices to show that $(P_n\cup P_2,P_3,P_4,\ldots,P_{n-1})$ is a $2c-d+1$-flower in $M\delete P_1$. In order to show that let $I'$ be a non-empty interval of the set $\{3,4,\ldots,n-1\}$.
	Then by \cref{localconninfinflowers}
	\begin{equation*}
		\lambda_{M\delete P_1}(S(I'))=\localconn_M(S(I'),E\setminus (P_1\cup S(I')))=2c-d.\qedhere
	\end{equation*}
\end{proof}

Contracting a petal reduces $d$ to $0$, while keeping $c-d$ and $k-c$ constant:

\begin{cor}\label{localconncontraction}
	The partition $(P_n\cup P_2,P_3,\ldots,P_{n-1})$ is a $k-d$-flower of $M\contract P_1$ in which adjacent petals have local connectivity $c-d$ and non-adjacent petals have local connectivity $0$.
\end{cor}
\begin{proof}
	By \cref{localconnindual} the partition $(P_1,\ldots,P_n)$ is a $k$-flower of $M^*$ with local connectivities $k-1-c$ for adjacent petals and $(k-1-c)-(c-d)$ for non-adjacent petals. So by \cref{localconnindeletion} $(P_n\cup P_2,P_3,\ldots,P_{n-1})$ is a $k-d$-flower of $M^*\delete P_1$ with local connectivities $k-1-c$ for adjacent petals and $(k-1-c)-(c-d)$ for non-adjacent petals. Applying \cref{localconnindual} again yields that $(P_n\cup P_2,P_3,\ldots, P_{n-1})$ is a $k-d$-flower of $M\contract P_1$ with local connectivities $c-d$ for adjacent petals and $0$ for non-adjacent petals.
\end{proof}

Now we have the tools to prove that in a matroid there are no infinite daisies:

\begin{lem}\label{noinfinitedaisies}
	For every $k\in \mathbb{N}$: There are no infinite $k$-daisies.
\end{lem}
\begin{proof}
	Assume for a contradiction that there is an infinite $k$-daisy.
	Then there also is a $k$-daisy of the form $(P_i)_{i\in \mathbb{N}}$ where the cyclic order is the one induced by the usual linear order of $\mathbb{N}$.
	Denote the local connectivity of adjacent petals by $c$ and the local connectivity of non-adjacent petals by $d$.
	For all $n\geq 5$, apply \cref{localconnindeletion} to the $k$-flowers $(P_0,\ldots,P_n,E \setminus \bigcup_{0 \leq i \leq n} P_i)$.
	So the partitions $(P_0\cup P_2, P_3,,\ldots, P_n,E \setminus \bigcup_{0 \leq i \leq n} P_i)$ are $2c-d+1$-flowers of $M\delete P_1$ with local connectivities $c$ for adjacent petals and $d$ for non-adjacent petals.
	Thus the partition $(P_0\cup P_2,P_3,\ldots)$ is a $2c-d+1$-daisy for $M\delete P_1$ with the same local connectivities.
	Similarly, by applying \cref{localconncontraction} to finite flowers, $(P_0\cup P_2\cup P_4,P_5,\ldots)$ is a $2(c-d)+1$-flower of $M\delete P_1\contract P_3$ with local connectivities $c-d$ for adjacent petals and $0$ for non-adjacent petals.
	Thus there is a matroid $N$ with an infinite daisy of the form $(Q_i)_{i\in \mathbb{N}}$ and $c\in \mathbb{N}$ such that the connectivity of intervals is $2c$ and the local connectivity between adjacent and non-adjacent petals is $c$ and $0$ respectively.
	By \cref{localconnindual}, $(Q_i)_{i\in \mathbb{N}}$ is also a $2c+1$-flower in $N^*$ with local connectivities $c$ for adjacent petals and $0$ for non-adjacent petals.
	
	For each $i\in \mathbb{N}$ let $B_i$ be a base of $N\restrict Q_i$ and $B_i'$ a base of $N\contractonto Q_i$ such that $B_i'\subseteq B_i$.
	As the local connectivity of $Q_i$ and $Q_{i+2}$ is $0$, $B_i\cup B_{i+2}$ is a base of $N\restrict (Q_i\cup Q_{i+2})$.
	Thus $B_i\cup B_{i+1}'\cup B_{i+2}$ is independent in $N\restrict Q_{[i,i+2]}$.
	Dually, $B_i'\cup B_{i+1}\cup B_{i+2}'$ is spanning in $N.Q_{[i,i+2]}$.
	Let $X$ be a base of $N\restrict Q_{[i,i+2]}$ such that $B_i\cup B_{i+1}'\cup B_{i+2}\subseteq X\subseteq B_i\cup B_{i+1}\cup B_{i+2}$.
	Let $Y$ be a base of $N.Q_{[i,i+2]}$ such that $B_i'\cup B_{i+1}'\cup B_{i+2}'\subseteq Y\subseteq B_i'\cup B_{i+1}\cup B_{i+2}'$.
	Now $X\setminus Y$ is finite, so by \cref{connectivityfornonnestedbases}
	\begin{equation*}
		\begin{split}
			\lambda_N(Q_{[i,i+2]})&=\left|X\setminus Y\right|-\left|Y\setminus X\right|\\
			&=\left|B_i\setminus B_i'\right|+\left|B_{i+2}\setminus B_{i+2}'\right|+\left|(X\setminus Y)\cap Q_{i+1}\right|-\left|Y\setminus X\right|\\
			&\geq 4c+0-2c.
		\end{split}
	\end{equation*}
	As the connectivity of $Q_{[i,i+2]}$ in $N$ is $2c$, the inequality has to be an equality, so $X=B_i\cup B_{i+1}'\cup B_{i+2}$ and $Y=B_i'\cup B_{i+1}\cup B_{i+2}'$.
	In particular, $B_i\cup B_{i+1}'\cup B_{i+2}$ is a base of $N\restrict Q_{[i,i+2]}$ and $B_i'\cup B_{i+1}\cup B_{i+2}'$ is a base of $N\contractonto Q_{[i,i+2]}$.
	
	Denote the set $\bigcup_{i\text{ even}}B_i\cup \bigcup_{i\text{ odd}} B_i'$ by $B$.
	Then $B$ spans all sets of the form $Q_{[i,i+2]}$ where $i$ is even, so $B$ is spanning in $N$.
	Assume for a contradiction that $B$ contains a circuit $C$. Then for every odd $i\in \mathbb{N}$, the set $C\cap Q_{[i,i+2]}$ is a scrawl of $N\contractonto Q_{[i,i+2]}$.
	Also $C\cap Q_{[i,i+2]}$ is contained in $B_i'\cup B_{i+1}\cup B_{i+2}'$ which was already shown to be a base of $N\contractonto Q_{[i,i+2]}$.
	So $C\cap Q_{[i,i+2]}$ is empty for all odd $i\in \mathbb{N}$.
	Thus $C$ is a subset of $Q_0$, so $C\subseteq B_0$, which is a contradiction to the fact that $B_0$ is independent in $N$.
	So $B$ is a base of $N$.
	For every $i\in \mathbb{N}$, the set $E\setminus B_i'$ is a base of $N^*\restrict Q_i$ and the set $E\setminus B_i$ is a base of $N^*\contractonto Q_i$.
	Thus, just as $B$ is a base of $N$, the set $\bigcup_{i\text{ even}}E \setminus B_{i}'\cup \bigcup_{i\text{ odd}}E\setminus B_i$ is a base of $N^*$, and thus $B':=\bigcup_{i\text{ even}}B_{i}'\cup \bigcup_{i\text{ odd}}B_i$ is a base of $N$.
	
	As $B_0'\cup B_1$ is independent in $N\restrict (Q_0\cup Q_1)$ and $B_0\cup B_1$ is spanning in $N\restrict (Q_0\cup Q_1)$, there is a base $R$ of $N\restrict (Q_0\cup Q_1)$ such that $B_0'\cup B_1\subseteq R\subseteq B_0\cup B_1$.
	As the partition $(Q_0\cup Q_1,Q_2,Q_3,\ldots)$ with the induced cyclic order is also a $2c+1$-daisy of $N$ with local connectivities $c$ for adjacent petals and $0$ for non-adjacent petals, also
	\begin{displaymath}
		R\cup \bigcup_{i\text{ even, }i\neq 0}B_i'\cup \bigcup_{i\text{ odd, }i\neq 1}B_i
	\end{displaymath}
	is a base of $N$.
	As this base contains $B'$ as a subset, and a base cannot be properly contained in another one, $R=B_0'\cup B_1$.
	So $B_0'\cup B_1$ is a base of $N\restrict (Q_0\cup Q_1)$, implying that $B_0'$ is a base of $N\delete Q_{[2,\infty[}\contract Q_1$.
	As $B_0'$ is also a base of $N\contract Q_{[2,\infty[}\contract Q_1$, the connectivity of $Q_0$ in $N\contract Q_1$ is $0$.
	But by \cref{connectivitybasic1}
	\begin{align*}
		\lambda_{N\contract Q_1}(Q_0)& =\lambda_N(Q_0\cup Q_1)-\lambda_{N\delete Q_0}(Q_1)\\
		&=\lambda_N(Q_0\cup Q_1) -\localconn_N(Q_1,Q_{[2,\infty[})\\
		&= 2c - c = c.
	\end{align*}
	So $c=0$, contradicting the fact that $(Q_0,Q_1,\ldots)$ is a $2c+1$-daisy of $N$.
\end{proof}

The following corollary implies that finding $k$-daisies that distinguish as many profiles as possible in an infinite matroid is the same as doing so in a finite matroid.

\begin{cor}\label{noinfchainofdaisies}
	Let $\Phi$ be a $k$-daisy which distinguishes at least two profiles with the same truncation to $k-1$-profiles.
	Then the set of $k$-flowers extending $\Phi$ has no infinite increasing chain.
\end{cor}
\begin{proof}
	Assume for a contradiction that there is an infinite increasing chain $(\Phi_j)_{j\in J}$ of $k$-flowers extending $\Phi$.
	As a concatenation of a $k$-anemone cannot be a $k$-daisy, all $\Phi_j$ are $k$-daisies.
	Given a $k$-pseudo\-flower $(Q_i)_{i\in \mathbb{N}}$, with cyclic order induced by the linear order of $\mathbb{N}$, denote for $i\geq 5$ the concatenation to $(Q_1,Q_2, \ldots, Q_{i-1}, Q_0\cup Q_i \cup Q_{i+1} \cup \cdots)$ by $\Psi_i$.
	Then there is a partition $(Q_i)_{i\in \mathbb{N}}$ such that all $\Psi_i$ are concatenations of some $\Phi_j$ and such that $\Phi$ equals some $\Psi_i$.
	As the order function is limit-closed, $(Q_i)_{i\in \mathbb{N}}$ is a $k$-pseudo\-flower.
	Consider the $k$-pseudo\-flower $\Psi$ which arises from $(Q_i)_{i\in \mathbb{N}}$ by concatenating $Q_0$ and $Q_1$ into one petal.
	As $(Q_2, Q_3, Q_4, Q_0\cup Q_1 \cup Q_5 \cup \cdots)$ is a concatenation of some $\Phi_j$ and thus is a $k$-daisy, $\Psi$ is not a $k$-anemone.
	By \cref{noinfinitedaisies} $\Psi$ is also not a $k$-daisy, so it is not a $k$-flower.
	As all the $\Psi_i$ are $k$-flowers, there is some $i\geq 1$ such that $Q_0\cup\cdots \cup Q_i$ has order less than $k-1$.
	Then for all $i'\geq i$
	\begin{equation*}
		\lambda(Q_0\cup \cdots \cup Q_{i'})\leq \lambda(Q_0\cup \cdots \cup Q_i)+\lambda(Q_i\cup \cdots\cup Q_{i'})-\lambda(Q_i)<k-1
	\end{equation*}
	and for all $i'<i$
	\begin{align*}
		\lambda(Q_0\cup\cdots\cup Q_{i'})&\leq \lambda(Q_0\cup \cdots \cup Q_i)+\lambda(Q_0\cup \cdots \cup Q_{i'}\cup Q_{i+2}\cup\cdots)\\
		&\quad-\lambda(Q_0\cup \cdots\cup Q_i\cup Q_{i+2}\cup \cdots)\\
		&=\lambda(Q_0\cup \cdots\cup Q_i)<k-1.
	\end{align*}
	As $\Phi$ distinguishes two profiles $P_1$ and $P_2$ with the same truncation to $k-1$-profiles and is a concatenation of some $\Psi_i$, all $\Psi_i$ with sufficiently large index $i$ distinguish $P_1$ and $P_2$.
	Then also there is some $i'$ such that $Q_0\cup Q_1\cup \cdots \cup Q_i$ distinguishes $P_1$ and $P_2$.
	That is a contradiction to $Q_0\cup \cdots \cup Q_i$ having order less than $k-1$.
\end{proof}

\Cref{characteriseinfdaisy} relates two problems:
The first problem is to determine when a general infinite connectivity system has infinite $k$-daisies.
The second problem is the question of under which circumstances a chain $(S_i)_{i\in I}$ of sets that all have the same finite connectivity $k$ can have a union whose connectivity is also $k$.
The answer to the second problem is that this can only happen if there are infinite $k$-flowers involved.

In order to prove \cref{characteriseinfdaisy} we need the following infinite version of Ramsey's Theorem, which can for example be found in \cite{DiestelBook5}:
\begin{thm}[{\cite[Theorem 9.1.2]{DiestelBook5}}]
	Let $X$ be an infinite set and $k\in \mathbb{N}$.
	Also let $l\in \mathbb{N}$ and let $c$ be a map from the subsets of $X$ of size $k$ to $\{1,\ldots,l\}$.
	Then there is an infinite set $Y\subseteq X$ such that $c$ maps all subsets of $Y$ of size $k$ to the same element of $\{1,\ldots,l\}$.
\end{thm}

\begin{thm}\label{characteriseinfdaisy}
	The following are equivalent:
	\begin{itemize}
		\item There is no $k\in \mathbb{N}$ for which there is an infinite $k$-daisy.
		\item For all $k\in \mathbb{N}$ and for all chains $(S_i)_{i\in \mathbb{N}}$ of subsets of $E$ with connectivity $\leq k-1$ either the supremum $S$ has connectivity $<k-1$ or there is a $k$-pseudo\-anemone displaying infinitely many $S_i$ with two petals $P_1$ and $P_2$ such that combining $P_1$ and $P_2$ into one petal yields an infinite $l$-anemone with $l\leq k$.
	\end{itemize}
\end{thm}
\begin{proof}
	First assume that there is an infinite $k$-daisy for some $k\in \mathbb{N}$.
	Then there also is an infinite $k$-daisy $(P_i)_{i\in \mathbb{N}}$ such that the cyclic order is the one induced by the natural linear order of $\mathbb{N}$.
	In this case, $(S_i)_{i\in \mathbb{N}}$ with $S_i=P_{i+1}\setminus P_0$ form a chain of sets of connectivity $k-1$ whose union is $E\setminus P_0$ and thus also has connectivity $k$.
	Assume for a contradiction that there is a $k$-pseudo\-flower $\Phi$ that displays infinitely many $S_i$ and has two petals $P_1$ and $P_2$ such that combining $P_1$ and $P_2$ into one petal yields an infinite $l$-anemone $\Psi$ for some $l\leq k$.
	Let $i_1$ be an index such that $S_{i_1}$ is displayed by $\Phi$ and such that if any $S_i$ displayed by $\Phi$ contains $P_1$ then so does $S_{i_1}$, and similarly for $P_2$.
	Let $i_2$ and $i_3$ be indices such that $i_1<i_2<i_3$ and such that $S_{i_2}$ and $S_{i_3}$ are displayed by $\Phi$.
	Then $S_{i_2}\setminus S_{i_1}$ and $S_{i_3}\setminus S_{i_2}$ are disjoint from $P_1\cup P_2$ and hence unions of petals of $\Psi$, and infinitely many petals of $\Psi$ are disjoint from $S_{i_3}$.
	Hence
	\begin{displaymath}
		\lambda(S_{i_1}\cup (S_{i_3}\setminus S_{i_2})) = \mu(S_{i_1}) = \lambda(S_{i_2})=k-1,
	\end{displaymath}
	where $\mu$ is the function from \cref{muforpetal} defined for the concatenation of $\Psi$ that arises from combining $S_{i_1}$ into one petal.
	But also $S_{i_1}\cup (S_{i_3}\setminus S_{i_2})$ is a union of petals of a concatenation of $(P_i)_{i\in \mathbb{N}}$ that is a finite $k$-daisy, and thus its connectivity is bigger than $k-1$, a contradiction.
	
	In the other direction, assume that there is no $k\in \mathbb{N}$ for which there is an infinite $k$-daisy.
	Let $k\in \mathbb{N}$ and let $(S_i)_{i\in \mathbb{N}}$ be a chain of subsets of $E$ that all have connectivity at most $k-1$, such that the union $S$ of all $S_i$ has connectivity $k-1$.
	Then at most finitely many $S_i$ have connectivity less than $k-1$, so without loss of generality all $S_i$ have connectivity exactly $k-1$.
	By Ramsey's theorem there is a subchain $(T_i)_{i\in \mathbb{N}}$ such that $\lambda(T_i\cap T_j^*)$ (for $j<i$) and $\lambda(S\cap T_i^*)$ do not depend on $i$ and $j$.
	As $\lambda$ is limit-closed, $\lambda(S\cap T_j^*)\leq \lambda(T_i\cap T_j^*)$ for $j<i$.
	If $\lambda(S\cap T_j^*)<\lambda(T_i\cap T_j^*)$, then
	\begin{align*}
		\lambda(S)&\leq \lambda(T_2) + \lambda(S \cap T_1^*) - \lambda(T_2 \cap T_1^*)<\lambda(T_2)=k-1
	\end{align*}
	and we are done.
	So assume that $\lambda(S\cap T_j^*)=\lambda(T_i\cap T_j^*)$ for $j<i$.
	
	Then the partition $\Phi$ that has $T_0\cup S^*$ as one partition class and the sets $T_{i+1} \cap T_i^*$ as its other partition classes is a $\lambda(S\cap T_0^*)+1$-flower with the cyclic order induced by $\mathbb{N}$.
	As there are no infinite daisies, it is a $\lambda(S\cap T_0^*)+1$-anemone.
	Then $T_1$ is the union of a petal of $\Phi$ together with a subset of another petal, and hence by \cref{muforpetal} $k-1=\lambda(T_2)\geq \lambda(S\cap T_0^*)$.
	So $\Phi$ is an infinite $l$-anemone for some $l\leq k$.
	Let $\Psi$ be the partition obtained from $\Phi$ by replacing $T_0\cup S^*$ with $T_0$ and $S^*$.
	It now suffices to show that any union of $T_0$ together with finitely many petals $R_1,\ldots,R_n$ of $\Phi$ not containing $T_0$ has connectivity at most $k-1$.
	For one petal $T_{i+1}\cap T_i^*$ with $i\geq 1$ this follows from
	\begin{align*}
		\lambda(T_0\cup (T_{i+1}\cap T_i^*)) \leq \lambda(T_{i+1}) + \lambda(T_0\cup S^*\cup (T_{i+1}\cap T_i^*)) - \lambda(T_{i+1}\cup S^*)=\lambda(T_{i+1})=k-1.
	\end{align*}
	For several petals it follows from submodularity of the connectivity function by induction on the number of petals, as $\lambda(T_0) = k - 1$.	 
\end{proof}

This result can be combined with the fact that there are no infinite daisies in a matroid:

\begin{cor}
	Let $\lambda$ be the connectivity function of a matroid on ground set $E$.
	Let $k\in \mathbb{N}$ and let $(S_i)_{i\in \mathbb{N}}$ be a strictly increasing chain of subsets of $E$ of connectivity at most $k-1$.
	If the union $S$ of the $S_i$ has connectivity $k-1$, then there is a $k$-pseudo\-flower $\Phi$ that displays infinitely many $S_i$ and that has petals $P_1$ and $P_2$ such that combining $P_1$ and $P_2$ into one petal yields an infinite $l$-anemone for some $l\leq k$.\qed
\end{cor}
\bibliographystyle{amsplain}
\bibliography{infiniteconnectivitysystems_flowers}
\end{document}